\documentclass{statsoc}
\usepackage{natbib}
\usepackage[left=7mm,right=15mm]{geometry}
\usepackage{vmargin}
\usepackage{amsmath,amssymb,hyperref,amsfonts}
\usepackage{natbib}
\usepackage{color}
\usepackage{pdfsync}
\usepackage{float}
\usepackage{graphicx}
\usepackage{subfigure}
\usepackage{xspace, framed}
\usepackage[ruled,vlined]{algorithm2e}
\usepackage{tabularx}
\newtheorem{lemma}{Lemma}[section]
\newtheorem{proposition}{Proposition}[section]
\newtheorem{thm}{Theorem}[section]

\newtheorem{corollary}{Corollary}[section]
\newtheorem{remark}{Remark}[section]

\title[Sparse CCA]{Sparse CCA via Precision Adjusted Iterative Thresholding}
\author[Chen et al.]{Mengjie Chen}
\address{Computational Biology and Bioinformatics, Yale University, USA}
\email{mengjie.chen@yale.edu}
\author[Chen et al.]{Chao Gao}
\address{Department of Statistics, Yale University, USA}
\email{chao.gao@yale.edu}
\author[Chen et al.]{Zhao Ren}
\address{Department of Statistics, Yale University, USA}
\email{zhao.ren@yale.edu}
\author[Chen et al.]{Harrison H. Zhou}
\address{Department of Statistics, Yale University, USA}
\email{huibin.zhou@yale.edu}
\setkeys{Gin}{width=1.0\textwidth}

\begin{document}

\begin{abstract}
Sparse Canonical Correlation Analysis (CCA)\ has received considerable attention in high-dimensional data analysis to study the relationship
between two sets of random variables. However, there has been remarkably
little theoretical statistical foundation on sparse CCA in high-dimensional
settings despite active methodological and applied research
activities. In this paper, we introduce an elementary sufficient and necessary characterization such that  the solution of CCA is indeed sparse, 
propose a computationally efficient procedure, called CAPIT, to estimate the canonical
directions, and show that the procedure is rate-optimal under various
assumptions on nuisance parameters. The procedure is applied to a breast
cancer dataset from The Cancer Genome Atlas project. We identify methylation
probes that are associated with genes, which have been previously characterized
as prognosis signatures of the metastasis of breast cancer.
\end{abstract}


\section{Introduction}

Last decades witness the delivery of an incredible amount of information
through the development of high-throughput technologies. Researchers now
routinely collect a catalog of different measurements from the same group of
samples. It is of great importance to elucidate the phenomenon in the
complex system by inspecting the relationship between two or even more sets
of measurements. Canonical correlation analysis is a popular tool to study
the relationship between two sets of variables. It has been successfully
applied to a wide range of disciplines, including psychology and
agriculture, and more recently, information retrieving \citep
{hardoon2004canonical}, brain-computer interface \citep{bin2009online},
neuroimaging \citep{avants2010dementia}, genomics \citep{witten2009extensions}
and organizational research \citep{bagozzi2011measurement}.

In this paper, we study canonical correlation analysis (CCA) in the
high-dimensional setting. The CCA in the classical setting, a celebrated
technique proposed by \cite{hotelling36}, is to find the linear combinations
of two sets of random variables with maximal correlation. Given two centered
random vectors $X\in \mathbb{R}^{p_{1}}$ and $Y\in \mathbb{R}^{p_{2}}$ with
joint covariance matrix 
\begin{equation}
\Sigma =%
\begin{pmatrix}
\Sigma _{1} & \Sigma _{12} \\ 
\Sigma _{21} & \Sigma _{2}%
\end{pmatrix}%
,  \label{eq:covariance}
\end{equation}%
the population version of CCA solves 
\begin{equation}
(\theta ,\eta )=\arg \max_{(a,b)}\left\{ a^{T}\Sigma _{12}b:a^{T}\Sigma
_{1}a=1,b^{T}\Sigma _{2}b=1\right\} .  \label{eq:CCAopt}
\end{equation}%
The optimization problem (\ref{eq:CCAopt}) can be solved by applying
singular value decomposition (SVD) on the matrix $\Sigma _{1}^{-1/2}\Sigma
_{12}\Sigma _{2}^{-1/2}$. In practice, \cite{hotelling36} proposed to
replace $\Sigma _{1}^{-1/2}\Sigma _{12}\Sigma _{2}^{-1/2}$ by the sample
version $\hat{\Sigma}_{1}^{-1/2}\hat{\Sigma}_{12}\hat{\Sigma}_{2}^{-1/2}$.
This leads to consistent estimation of the canonical directions $(\theta
,\eta )$ when the dimensions $p_{1}$ and $p_{2}$ are fixed and sample size $%
n $ increases. However, in the high-dimensional setting, when the dimensions 
$p_{1}$ and $p_{2}$ are large compared with sample size $n$, this SVD
approach may not work. In fact, when the dimensions exceed the sample size,
SVD cannot be applied because the inverse of the sample covariance does not
exist. 

The difficulty motivates people to impose structural assumptions on
the canonical directions in the CCA problem. For example, sparsity has been
assumed on the canonical directions 
\citep{wiesel08, witten09, parkhomenko09, hardoon2011sparse,
le2009sparse, waaijenborg2009sparse, avants2010dementia}. The sparsity
assumption implies that most of the correlation between two random vectors
can be explained by only a small set of features or coordinates, which
effectively reduces the dimensionality and at the same time improves the
interpretability in many applications. However, to our best knowledge, there
is no full characterization of the probabilistic CCA model that the
canonical directions are indeed sparse. As a result, there has been
remarkably little theoretical study on sparse CCA in high-dimensional
settings despite recent active developments in methodology. This motivates
us to find a sufficient and necessary condition on the covariance structure (%
\ref{eq:covariance}) such that the solution of CCA is sparse. We show in Section \ref{sec:model} that 
$(\theta _{1},\eta _{1})$ is the solution of (\ref{eq:CCAopt}) if and only
if (\ref{eq:covariance}) satisfies 
\begin{equation}
\Sigma _{12}=\Sigma _{1}\left( \sum_{i=1}^{r}\lambda _{i}\theta _{i}\eta
_{i}^{T}\right) \Sigma _{2},  \label{eq:characterization}
\end{equation}%
where $\lambda _{i}$ decreases $\lambda _{1}>\lambda _{2}\geq \ldots \geq
\lambda _{r}>0$, $r=\text{rank}\left(\Sigma _{12}\right) $,  and $\left\{ (\theta _{i}, \eta
_{i}) \right\} $ are orthonormal w.r.t. metric $\Sigma _{1}$ and $\Sigma _{2}\ 
$respectively. i.e. $\theta _{i}^{T}\Sigma _{1}\theta _{j}=\mathbb{I}\{i=j\}$
and $\eta _{i}^{T}\Sigma _{2}\eta _{j}=\mathbb{I}\{i=j\}$. With this characterization, the canonical directions
are sparse if and only if $\theta _{1}$ and $\eta _{1}$ in (\ref%
{eq:characterization}) are sparse. Hence, sparsity assumption can be made
explicit in this probabilistic model.

Motivated by the characterization (\ref{eq:characterization}), we propose a
method called CAPIT, standing for Canonical correlation Analysis via
Precision adjusted Iterative Thresholding, to estimate the sparse canonical
directions. Our basic idea is simple. First, we obtain a
good estimator of the precision matrices $(\Sigma _{1}^{-1},\Sigma
_{2}^{-1}) $. Then, we transform the data by the estimated precision
matrices to adjust the influence of the nuisance covariance $(\Sigma
_{1},\Sigma _{2})$. Finally, we apply iterative thresholding on the
transformed data. The method is fast to implement in the sense that it
achieves the optimal statistical accuracy in only finite steps of iterations.

Rates of convergence for the proposed estimating procedure are obtained
under various sparsity assumptions on canonical directions and covariance
assumptions on $(\Sigma _{1},\Sigma _{2})$. In Section \ref{sec:lowerbound}
we establish the minimax lower bound for the sparse CCA problem. The rates
of convergence match the minimax lower bound as long as the estimation of
nuisance parameters $(\Sigma _{1},\Sigma _{2})$ is not dominating in
estimation of the canonical directions. To the best of our
knowledge, this is the first theoretically guaranteed method proposed in the
sparse CCA literature.

We point out that the sparse CCA methods proposed in the literature may
have both computational and statistical drawbacks. On the
computational side, regularized versions of (\ref{eq:CCAopt}) such as \cite%
{waaijenborg2009sparse} and \cite{wiesel08} are proposed in the literature
based on heuristics to avoid the non-convex nature of (\ref{eq:CCAopt}), but
there is no theoretical guarantee whether these algorithms would lead to consistent estimators. 
On the statistical side, methods proposed in the
literature do not explicitly take into account of the influence of the
nuisance parameters. For example, \cite{witten09} and \cite{parkhomenko09}
implicitly or explicitly use diagonal matrix or even identity matrix to
approximate the unknown precision matrices $(\Sigma _{1}^{-1},\Sigma
_{2}^{-1})$. Such approximation could be valid when the covariance matrices $%
(\Sigma _{1},\Sigma _{2})$ are nearly diagonal, otherwise there is no
theoretical guarantee of consistency of the procedures. We illustrate this
fact by a numerical example. We draw data from a multivariate Gaussian
distribution and then apply the proposed method and the Penalized Matrix Decomposition method 
by \cite{witten09}. We show the results in Figure \ref{scca_precision_plot}.
By taking into account of the structure of the nuisance
parameters, the CAPIT accurately recovers the sparse canonical directions,
while the PMD is not consistent. In this simulation study, we consider
sparse precision matrices and sparse canonical directions, where the sparse
assumption of precision matrices has a sparse graphical interpretation of $X$
and $Y$ when the distribution is Gaussian. See Section \ref{sec:sce2} for
more details.

\begin{figure}
\centering
\includegraphics[width=4in]{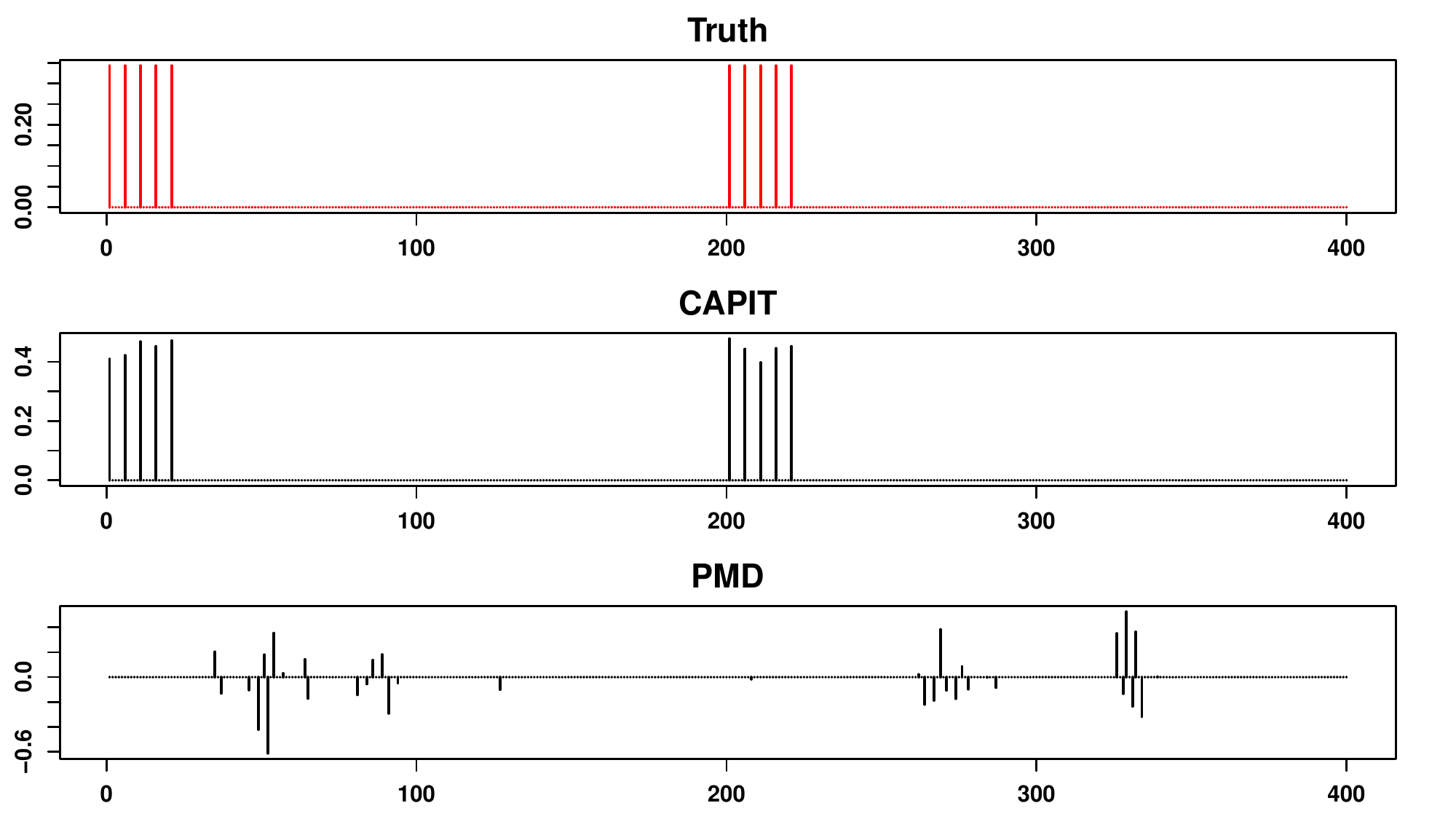}
\caption{Visualization of the simulation results of estimating $(\protect%
\theta ,\protect\eta )$ for a replicate from Scenario II in Section \protect
\ref{sec:sce2} when $p=200$ and $n=500$. The $\{1,2,...,200\}$-th
coordinates represent $\protect\theta $ and the $\{201,202,...,400\}$-th
coordinates represent $\protect\eta $. }
\label{scca_precision_plot}
\end{figure}

A closely related problem is the principal component analysis (PCA) \citep
{hotelling1933analysis}. In high-dimensional setting, sparse PCA is studied
in \cite{johnstone09}, \cite{ma2013sparse} and \cite{cai12b}. However, the
PCA and CCA problems are fundamentally different. With the characterization
of covariance structure in (\ref{eq:characterization}), such difference
becomes clear. We illustrate the simplest rank-one case. Assuming the
correlation rank $r$ in (\ref{eq:characterization}) is one, the covariance
structure is reduced to 
\begin{equation}
\Sigma _{12}=\lambda \Sigma _{1}\theta \eta ^{T}\Sigma _{2}.
\label{eq:SCPintro}
\end{equation}%
We refer to (\ref{eq:SCPintro}) as the Single Correlation Pair (SCP) model.
In the PCA literature, the corresponding rank-one model is called the
single-spike model. Its covariance structure can be written as 
\begin{equation}
\Sigma =\lambda \theta \theta ^{T}+I,  \label{eq:SPintro}
\end{equation}%
where $\theta $ is the principal direction of the random variable. A
comparison of (\ref{eq:SCPintro}) and (\ref{eq:SPintro}) reveals that
estimation of the CCA is more involved than that of the PCA because of the
presence of the nuisance parameters $(\Sigma _{1},\Sigma _{2})$, and the
difficulty of estimating covariance matrices and precision matrices is
known in high-dimensional statistics \citep{cai12a,ren13}. In the sparse PCA
setting, the absence of nuisance parameter in (\ref{eq:SPintro}) leads to
algorithms directly applied on the sample covariance matrix $\hat{\Sigma}$,
and the corresponding theoretical analysis is more tractable. In contrast,
in the sparse CCA setting, not only do we need to adapt to the underlying
sparsity of $(\theta ,\eta )$, but we also need to adapt to the unknown
covariance structure $(\Sigma _{1},\Sigma _{2})$. We are going to show in
Section \ref{sec:theory} how various structures of $(\Sigma _{1},\Sigma
_{2}) $ influence the convergence rate of the proposed method.

In addition, we demonstrate the CAPIT method by a real data example. We
apply the proposed method to the data arising in the field of cancer genomics where
methylation and gene expression are profiled for the same group of breast
cancer patients. The method explicitly takes into account the sparse
graphical model structure among genes. Interestingly, we identify
methylation probes that are associated with genes that are previously
characterized as prognosis signatures of the metastasis of breast cancer.
This example suggests the proposed method provides a reasonable framework
for exploratory and interpretive analysis of multiple datasets in
high-dimensional settings.

The contributions in the paper are two-fold. First, we characterize the
sparse CCA problem by proposing the probabilistic model and establish the
minimax lower bound under certain sparsity class. Second, we propose the
CAPIT method to adapt to both sparsity of the canonical direction and the
nuisance structure. The CAPIT procedure is computationally efficient and
attains optimal rate of convergence under various conditions. The paper is
organized as follows. We first provide a full characterization of the sparse
CCA model in Section \ref{sec:model}. The CAPIT method and its associated
algorithms are presented in Section \ref{sec:method}. Section \ref%
{sec:theory} is devoted to a theoretical analysis of our method. This
section also presents the minimax lower bound. Section \ref{sec:simu} and
Section \ref{sec:real} investigate the numerical performance of our
procedure by simulation studies and a real data example. The proof of the
main theorem, Theorem \ref{thm:main}, is gathered in Section \ref{sec:proof}%
. The proofs of all technical lemmas and Theorem \ref{thm:lower} are
gathered in Appendix.

\subsection{Notations}

For a matrix $A=(a_{ij})$, we use $||A||$ to denote its largest singular
value and call it the spectral norm of $A$. The Frobenius norm is defined as 
$||A||=\sqrt{\sum_{ij}a_{ij}^2}$. The matrix $l_1$ norm is defined as $%
||A||_{l_1}=\max_j\sum_i|a_{ij}|$. The norm $||\cdot||$, when applied to a
vector, is understood as the usual Euclidean $l_2$ norm. For any two real
numbers $a$ and $b$, we use notations $a\vee b=\max(a,b)$ and $a\wedge
b=\min(a,b)$. Other notations will be introduced along with the text.

\section{The Sparse CCA Model}

\label{sec:model}

Let $X\in \mathbb{R}^{p_{1}}$ and $Y\in \mathbb{R}^{p_{2}}$ be two centered
multivariate random vectors with dimension $p_{1}$ and $p_{2}$ respectively.
Write the covariance matrix of $(X^T, Y^T)^T$ as follows, 
\begin{equation*}
\mbox{\rm Cov}%
\begin{pmatrix}
X \\ 
Y%
\end{pmatrix}%
=%
\begin{pmatrix}
\Sigma _{1} & \Sigma _{12} \\ 
\Sigma _{21} & \Sigma _{2}%
\end{pmatrix}%
,
\end{equation*}%
where $\Sigma _{1}$ is the covariance matrix of $X$ with $\Sigma _{1}=%
\mathbb{E}XX^{T} $, $\Sigma _{2}$ is the covariance matrix of $Y$ with $%
\Sigma _{2}=\mathbb{E}YY^{T}$, and $\Sigma _{21}=\Sigma _{12}^{T}$ the
covariance structure between $X$ and $Y$ with $\Sigma _{12}=\Sigma _{21}^{T}=%
\mathbb{E}XY^{T}$. The canonical directions $\theta \in \mathbb{R}^{p_{1}}$
and $\eta \in \mathbb{R}^{p_{2}}$ are solutions of 
\begin{equation}
\max_{a\in \mathbb{R}^{p_{1}},b\in \mathbb{R}^{p_{2}}}\frac{a^{T}\Sigma
_{12}b}{\sqrt{a^{T}\Sigma _{1}a}\sqrt{b^{T}\Sigma _{2}b}}.  \label{eq:cca}
\end{equation}%
where we assume $\Sigma _{1}$ and $\Sigma _{2}$ are invertible and $\Sigma
_{12}$ is nonzero such that the maximization problem is not degenerate.
Notice when $(\theta ,\eta )$ is the solution of (\ref{eq:cca}), $(\sigma
_{1}\theta ,\sigma _{2}\eta )$ is also the solution with arbitrary scalars $%
(\sigma _{1},\sigma _{2})$ satisfying $\sigma _{1}\sigma _{2}>0$. To achieve
identifiability up to a sign, (\ref{eq:cca}) can be reformulated into the
following optimization problem. 
\begin{equation}
\mbox{\rm maximize }a^{T}\Sigma _{12}b,\quad \mbox{\rm subject to }%
a^{T}\Sigma _{1}a=1\quad \mbox{\rm and}\quad b^T\Sigma _{2}b=1.
\label{eq:cca*}
\end{equation}

\begin{proposition}
\label{prop:characterization}When $\Sigma _{12}$ is of rank $1$, the
solution (up to sign jointly) of Equation (\ref{eq:cca*}) is $\left( \theta
,\eta \right) $ if and only if the covariance structure between $X$ and $Y$
can be written as 
\begin{equation*}
\Sigma _{12}=\lambda \Sigma _{1}\theta \eta ^{T}\Sigma _{2},
\end{equation*}%
where $0<\lambda \leq 1,$ $\theta ^{T}\Sigma _{1}\theta =1$ and $\eta
^{T}\Sigma _{2}\eta =1$. In other words, the correlation between $a^{T}X$
and $b^{T}Y$ are maximized by $\mbox{\rm corr}(\theta ^{T}X,\eta ^{T}Y)$,
and $\lambda $ is the canonical correlation between $X$ and $Y$.
\end{proposition}

The Proposition above is just an elementary consequence of SVD after
transforming the parameters $\theta $ and $\eta $ into $\Sigma
_{1}^{1/2}\theta $ and $\Sigma _{2}^{1/2}\eta $ respectively. For the
reasons of space, the proof is omitted. For general $\Sigma _{12}$ with rank 
$r\geq 1,$ it's a routine extension to see that the unique (up to sign
jointly) solution of Equation (\ref{eq:cca*}) is $\left( \theta _{1},\eta
_{1}\right) $ if and only if the covariance structure between $X$ and $Y$
can be written as 
\begin{equation*}
\Sigma _{12}=\Sigma _{1}\left( \sum_{i=1}^{r}\lambda _{i}\theta _{i}\eta
_{i}^{T}\right) \Sigma _{2},
\end{equation*}%
where $\lambda _{i}$ decreases $\lambda _{1}>\lambda _{2}\geq \ldots \geq
\lambda _{r}>0$, $\left\{ \theta _{i}\right\} $ and $\left\{ \eta
_{i}\right\} $ are orthonormal w.r.t. metric $\Sigma _{1}$ and $\Sigma _{2}\ 
$respectively. i.e. $\theta _{i}^{T}\Sigma _{1}\theta _{j}=\mathbb{I}\{i=j\}$
and $\eta _{i}^{T}\Sigma _{2}\eta _{j}=\mathbb{I}\{i=j\}$.

Inspired by (\ref{eq:cca*}), we propose a probabilistic model of $(X,Y)$, so
that the canonical directions $(\theta ,\eta )$ are explicitly modeled in
the joint distribution of $(X,Y)$. 
\begin{framed}
\noindent \textbf{The Single Canonical Pair Model}
\begin{equation}
\begin{pmatrix}
X \\ Y
\end{pmatrix}\sim N\begin{pmatrix}
\begin{pmatrix}
0 \\ 0
\end{pmatrix}, \begin{pmatrix}
\Sigma_1 & \lambda\Sigma_1\theta\eta^T\Sigma_2 \\
\lambda\Sigma_2\eta\theta^T\Sigma_1 & \Sigma_2
\end{pmatrix}
\end{pmatrix}, \label{eq:model}
\end{equation}
with $\Sigma_1\succ 0$, $\Sigma_2\succ 0$, $\theta^T\Sigma_1\theta=1$, $\eta^T\Sigma_2\eta=1$ and $0<\lambda\leq 1$.
\end{framed}

Just as the single-spike model in PCA \citep{tipping99, johnstone09}, the
model (\ref{eq:model}) explicitly models $(\lambda ,\theta ,\eta )$ in the
form of the joint distribution of $(X,Y)$. Besides, it can be generalized to
multiple canonical-pair structure as in the multi-spike model %
\citep{birnbaum12}. On the other hand, it is fundamentally different from
the single-spike model, because $(\Sigma _{1},\Sigma _{2})$ are typically
unknown, so that estimating $(\theta ,\eta )$ is much harder than estimating
the spike in PCA. Even when both $\Sigma _{1}$ and $\Sigma _{2}$ are
identity and $\Sigma _{12}=\lambda \theta \eta ^{T}$, it cannot be reduced
into the form of spike model. \cite{bach05} also proposed a statistical
model for studying CCA in a probabilistic setting. Under their model, the
data has a latent variable representation. It can be shown that the model we
propose is equivalent to theirs in the sense that both can be written into
the form of the other. The difference is that we explicitly model the
canonical directions $(\theta ,\eta )$ in the covariance structure for
sparse CCA.

\section{Methodology}

\label{sec:method}

In this section, we introduce the CAPIT algorithm to estimate the sparse
canonical direction pair $(\theta ,\eta )$ in the single canonical pair
model in details. We start with the main part of the methodology in Section \ref%
{sec:method IT}, an iterative thresholding algorithm, requiring an
initializer and consistent estimators of precision matrices (nuisance
parameters). Then in Section \ref{sec:method CT} we introduce a coordinate
thresholding algorithm to provide a consistent initializer. Finally, in Section %
\ref{sec::estimation of precision} rate-optimal estimators of precision
matrices are reviewed over various settings.

The procedure is motivated by the power method, a standard technique to
compute the leading eigenvector of a given symmetric matrix $S$ %
\citep{golub1996matrix}. Let $S$ be a $p\times p$ symmetric matrix. We
compute its leading eigenvector. Starting with a vector $v^{(0)}$
non-orthogonal to the leading eigenvector, the power method generates a
sequence of vectors $v^{(i)}$, $i=1,2,\ldots ,$ by alternating the
multiplication step $w^{(i)}=Sv^{(i-1)}$ and the normalization step $%
v^{(i)}=w^{(i)}/\left\Vert w^{(i)}\right\Vert $ until convergence. The limit
of the sequence, denoted by $v^{(\infty )}$, is the leading eigenvector. The
power method can be generalized to compute the leading singular vectors of
any $p_{1}\times p_{2}$ dimensional rectangular matrix $M$. Suppose the SVD of
a rank $d$ matrix $M$ is $M=UDV^{T}$, where $D$ is the $d$ dimensional
diagonal matrix with singular values on the diagonal. Suppose we are given
an initial pair $\left( u^{(0)},v^{(0)}\right) $, non-orthogonal to the
leading singular vectors. To compute the leading singular vectors, power
method alternates the following steps until $\left( u^{(0)},v^{(0)}\right) $
converges to $\left( u^{(\infty )},v^{(\infty )}\right) $, which are the
left and right leading singular vectors.

\begin{enumerate}
\item Right Multiplication: $w_{l}^{(i)}=Mv^{(i-1)},$

\item Left Normalization: $u^{(i)}=w_{l}^{(i)}/\left\Vert
w_{l}^{(i)}\right\Vert ,$

\item Left Multiplication: $w_{r}^{(i)}=u^{(i)}M,$

\item Right Normalization: $v^{(i)}=w_{r}^{(i)}/\left\Vert
w_{r}^{(i)}\right\Vert .$
\end{enumerate}

Our goal is to estimate the canonical direction pair $\left( \theta ,\eta
\right) $. The power method above motivates us to find a matrix $\hat{A}$
close to $\lambda \theta \eta ^{T}$ of which $\left( \theta /\left\Vert
\theta \right\Vert ,\eta /\left\Vert \eta \right\Vert \right) $ is the
leading pair of singular vectors. Note that the covariance structure is $\Sigma
_{12}=\lambda \Sigma _{1}\theta \eta ^{T}\Sigma _{2}$. Suppose we know the
marginal covariance structures of $X$ and $Y$, i.e. $\Omega _{1} = \Sigma_1^{-1}$ and $%
\Omega _{2} = \Sigma_2^{-1}$ are given, it is very natural to consider $\Omega _{1}\hat{%
\Sigma}_{12}\Omega _{2}$ as the target matrix, where $\hat{\Sigma}_{12}$ is
the sample cross-covariance between $X$ and $Y$. Unfortunately, the
covariance structures $\Omega _{1}$ and $\Omega _{2}$ are unknown as
nuisance parameters, but a rate-optimal estimator of $\Omega _{j}$ ($j=1,2$)
usually can be obtained under various assumptions on the covariance or
precision structures of $X$ and $Y$ in many high-dimensional settings. In
literature, some commonly used structures are sparse precision matrix,
sparse covariance matrix, bandable covariance matrix and Toeplitz covariance
matrix structures. Later we will discuss the estimators of the precision
matrices and their influences to the final estimation error of canonical
direction pair $\left( \theta ,\eta \right) $.

We consider the idea of data splitting. Suppose we have $2n$ i.i.d. copies $%
(X_{i},Y_{i})_{1\leq i\leq 2n}$. We use the first half to compute the sample
covariance $\hat{\Sigma}_{12}=\frac{1}{n}\sum_{i=1}^{n}X_{i}Y_{i}^{T}$, and
use the second half to estimate the precision matrices by $\hat{\Omega}_{1}$
and $\hat{\Omega}_{2}$. Hence the matrix $\hat{A}=\hat{\Omega}_{1}\hat{\Sigma%
}_{12}\hat{\Omega}_{2}$ is available to us. The reason for data splitting is
that we can write the matrix $\hat{A}$ in an alternative form. That is, 
\begin{equation*}
\hat{A}=\frac{1}{n}\sum_{i=1}^{n}\tilde{X}_{i}\tilde{Y}_{i}^{T},
\end{equation*}%
where $\tilde{X}_{i}=\hat{\Omega}_{1}X_{i}$ and $\tilde{Y}_{i}=\hat{\Omega}%
_{2}Y_{i}$ for all $i=1,...,n$. Conditioning on $%
(X_{n+1},Y_{n+1}),...,(X_{2n},Y_{2n})$, the transformed data $(\tilde{X}_{i},%
\tilde{Y}_{i})_{1\leq i\leq n}$ are still independently identically
distributed. This feature allows us to explore some useful concentration
results in the matrix $\hat{A}$ to prove theoretical results. Conditioning
on the second half of data, the expectation of $\hat{A}$ is $\lambda \alpha
\beta ^{T}$, where $\alpha =\hat{\Omega}_{1}\Sigma _{1}\theta $ and $\beta =%
\hat{\Omega}_{2}\Sigma _{2}\eta $. Therefore, the method we develop is
targeted at $(\alpha ,\beta )$ instead of $(\theta ,\eta )$. However, as
long as the estimators $(\hat{\Omega}_{1},\hat{\Omega}_{2})$ are accurate in
the sense that 
\begin{equation*}
||\alpha -\theta ||\vee ||\beta -\eta ||=||(\hat{\Omega}_{1}\Sigma
_{1}-I)\theta ||\vee ||(\hat{\Omega}_{2}\Sigma _{2}-I)\eta ||
\end{equation*}%
is small, the final rate of convergence is also small.

If we naively apply the power method above to $\hat{A}=\hat{\Omega}_{1}%
\hat{\Sigma}_{12}\hat{\Omega}_{2}$ in high-dimensional setting, the
estimation variance accumulated across all $p_{1}$ and $p_{2}$ coordinates
of left and right singular vectors goes very large and it is possible that we can never
obtain a consistent estimator of the space spanned by the singular vectors. 
\cite{johnstone09} proved that when $p/n\nrightarrow 0$, the leading
eigenspace estimated directly from the sample covariance matrix can be
nearly orthogonal to the truth under the PCA setting in which $\hat{A}$ is
the sample covariance matrix with dimension $p_{1}=p_{2}=p$. Under the
sparsity assumption of $\left( \theta ,\eta \right) $, a natural way of
reducing the estimation variance is to only estimate those coordinates with
large values in $\theta $ and $\eta $ respectively and simply estimate the
rest coordinates by zero. Although bias is caused by this thresholding idea,
in the end the variance reduction dominates the biased inflation and this
trade-off minimizes the estimation error to the optimal rate. The idea of
combining the power method and the iterative thresholding procedure leads to the
algorithm in the next section which was also proposed by \cite%
{yang2013sparse} for a general data matrix $\hat{A}$ without a theoretical
analysis.

\subsection{Iterative Thresholding}

\label{sec:method IT}

We incorporate the thresholding idea into ordinary power method above for
SVD by adding a thresholding step after each right and left multiplication
steps before normalization. The thresholding step kills those coordinates
with small magnitude to zero and keep or shrink the rest coordinates through
a thresholding function $T\left( a,t\right) $ in which $a$ is a vector and $%
t $ is the thresholding level. In our theoretical analysis, we assume that $%
T\left( a,t\right) =\Big(a_k \mathbb{I}\{|a_k|\geq t\}\Big)$ is the
hard-thresholding function, but any function serves the same purpose in
theory as long as it satisfies (i) $\left\vert T\left( a,t\right)_k
-a_k\right\vert \leq t$ and (ii) $T\left( a,t\right)_k =0$ whenever $%
\left\vert a_k\right\vert <t$. Therefore the thresholding function can be
hard-thresholding, soft-thresholding or SCAD \citep{fan2001variable}. The
algorithm is summarized below.

\begin{algorithm}[H]
 \SetAlgoLined
 \KwIn{Sample covariance matrices $\hat{\Sigma}_{12}$\;
       Estimators of precision matrix $\hat{\Omega}_1,\hat{\Omega}_{2}$\;
       Initialization pair $\alpha^{(0)}, \beta^{(0)}$\;
       Thresholding level $\gamma_1,\gamma_2$.}
 \KwOut{Canonical direction estimator $\alpha^{(\infty)}, \beta^{(\infty)}$.}
 Set $\hat{A}=\hat{\Omega}_{1}\hat{\Sigma}_{12}\hat{\Omega}_{2}$\;
 \Repeat {Convergence of $\alpha^{(i)}$ and $\beta^{(i)}$}{
  Right Multiplication: $w^{l,(i)}=\hat{A}\beta^{(i-1)}$\;
  Left Thresholding: $w_{th}^{l,(i)}=T\left(w^{l,(i)},\gamma_1\right)$\;
  Left Normalization: $\alpha^{(i)}=w_{th}^{l,(i)}/\left\Vert w_{th}^{l,(i)}\right\Vert$\;
  Left Multiplication: $w^{r,(i)}=\alpha^{(i)}\hat{A}$\;
  Right Thresholding: $w_{th}^{r,(i)}=T\left(w^{r,(i)},\gamma_2\right)$\;
  Right Normalization: $\beta^{(i)}=w_{th}^{r,(i)}/\left\Vert w_{th}^{r,(i)}\right\Vert$;}
  \caption{CAPIT: Iterative Thresholding}
  \label{alg:: ITSCCA}
\end{algorithm}

\begin{remark}
In Algorithm \ref{alg:: ITSCCA}, we don't provide specific stopping rule
such as that the difference between successive iterations is small enough.
For the single canonical pair model, we are able to show in Section \ref%
{sec:theory} that the convergence is achieved in just one step. The
intuition is simple: when $A$ is of exact rank one, we can simply obtain the
left singular vector via right multiplying $A$ by any vector non-orthogonal
to the right singular vector. Although in the current setting $\hat{A}$ is
not a rank one matrix, the effect caused from the second singular value in
nature does not change the statistical performance of our final estimator.
\end{remark}

\begin{remark}
The thresholding level $(\gamma _{1},\gamma _{2})$ are user-specified.
Theoretically, they should be set at the level $O\Big(\sqrt{\frac{\log
(p_{1}\vee p_{2})}{n}}\Big)$. In Section \ref{sec:threshconst}, we present a
fully data-driven $(\gamma _{1},\gamma _{2})$ along with the proof.
\end{remark}

\begin{remark}
The estimator $(\alpha^{(\infty)},\beta^{(\infty)})$ does not directly
estimate $(\theta,\eta)$ because the former are unit vectors while the later
satisfies $\theta^T\Sigma_1\theta=\eta^T\Sigma_2\eta=1$. We are going to
prove they are almost in the same direction by considering the loss function 
$|\sin \angle (a,b)|^2$ in \cite{johnstone09}. Details are presented in
Section \ref{sec:theory}.
\end{remark}

\begin{remark}
The estimators of precision matrices $\hat{\Omega}_{1}$ and $\hat{\Omega}_{2}
$ depend on the second half of the data and the estimator $\hat{\Sigma}_{12}$
depends on the first half of the data. In practice, after we apply Algorithm %
\ref{alg:: ITSCCA}, we will swap the two parts of the data and use the first
half to get $\hat{\Sigma}_{12}$ and the second half to obtain $\hat{\Omega}%
_{1},\hat{\Omega}_{2}$. Then, Algorithm \ref{alg:: ITSCCA} is run again on
the new estimators. The final estimator can be calculated through averaging
the two. More generally, we can do sample splitting many times and take an
average as Bagging, which is often used to improve the stability and
accuracy of machine learning algorithms.
\end{remark}

\subsection{Initialization by Coordinate Thresholding}

\label{sec:method CT}

In Algorithm \ref{alg:: ITSCCA}, we need to provide an initializer $\left(
\alpha ^{(0)},\beta ^{(0)}\right) $ as input. We generate a sensible
initialization in this section which is similar to the \textquotedblleft
diagonal thresholding\textquotedblright\ sparse PCA method proposed by \cite%
{johnstone09}. Specifically, we apply a thresholding step to pick index sets 
$B_{1}$ and $B_{2}$ of the coordinates of $\theta $ and $\eta $
respectively. Those index sets can be thought as strong signals. Then a
standard SVD is applied on the submatrix of $\hat{A}$ with rows and columns
indexed by $B_{1}$ and $B_{2}$. The dimension of this submatrix is
relatively low such that the SVD on it is fairly accurate. The leading pair
of singular vectors is of dimension $\left\vert B_{1}\right\vert $ and $%
\left\vert B_{2}\right\vert $, where $\left\vert \cdot \right\vert $ denotes
the cardinality. In the end, we zero-pad the leading pair of singular
vectors into dimension $p_{1}$ and $p_{2}$ respectively to provide our
initializer $\left( \alpha ^{(0)},\beta ^{(0)}\right) $. The algorithm is
summarized in Algorithm \ref{alg:: CTSCCA}.

\begin{algorithm}[H]
 \SetAlgoLined
 \KwIn{Sample covariance matrices $\hat{\Sigma}_{12}$\;
       Estimators of precision matrix $\hat{\Omega}_1,\hat{\Omega}_{2}$\;
       Thresholding level $t_{ij}$.}
 \KwOut{Initializer $\alpha^{(0)}$ and $\beta^{(0)}$.}
 \setcounter{AlgoLine}{1}
 Set $\hat{A}=\hat{\Omega}_{1}\hat{\Sigma}_{12}\hat{\Omega}_{2}$ \;
 \ShowLn Coordinate selection: pick the index sets $B_{1}$ and $B_{2}$ of the
coordinates of $\theta $ and $\eta $ respectively as follows,
$B_{1}=\left\{ i,\max_{j}\left\vert \hat{a}_{ij}\right\vert /t_{ij}\geq \sqrt{%
\frac{\log p_{1}}{n}}\right\} ,B_{2}=\left\{ j,\max_{i}\left\vert \hat{a}%
_{ij}\right\vert /t_{ij}\geq \sqrt{\frac{\log p_{2}}{n}}\right\} $\;
 \setcounter{AlgoLine}{2}
 \ShowLn Reduced SVD: compute the leading pair of singular vectors $\left(
\alpha ^{(0),B_{1}},\beta ^{(0),B_{2}}\right) $ on the submatrix $\hat{A}_{B_{1},B_{2}}$\;
 \setcounter{AlgoLine}{3}
 \ShowLn Zero-Padding procedure: construct the initializer $\left( \alpha
^{(0)},\beta ^{(0)}\right) $ by zero-padding $\left( \alpha
^{(0),B_{1}},\beta ^{(0),B_{2}}\right) $ on index sets $B_{1}^{c}$ and $%
B_{2}^{c}$ respectively,
$
 \alpha _{B_{1}}^{(0)}=\alpha ^{(0),B_{1}},\alpha _{B_{1}^{c}}^{(0)}=0,\beta
_{B_{2}}^{(0)}=\beta ^{(0),B_{2}},\beta _{B_{2}^{c}}^{(0)}=0.
$
  \caption{CAPIT: Initialization by Coordinate Thresholding}
  \label{alg:: CTSCCA}
\end{algorithm}

The thresholding level $t_{ij}$ in Algorithm \ref{alg:: CTSCCA} is a user
specified constant and allowed to be adaptive to each location $\left(
i,j\right) $. The theoretical data-driven constant for each $t_{ij}$ is
provided in\ Section \ref{sec:threshconst}. It is clear the initializer is
not unique since if $\left( \alpha ^{(0)},\beta ^{(0)}\right) $ serves as
the output, $\left( -\alpha ^{(0)},-\beta ^{(0)}\right) $ is also a solution
of Algorithm \ref{alg:: CTSCCA}. However either pair works as an initializer
and provides the same result because in the end we estimate the space
spanned by leading pair of singular vectors.

\subsection{Precision Estimation}

\label{sec::estimation of precision}

Algorithms \ref{alg:: ITSCCA} and \ref{alg:: CTSCCA} require precision
estimators $\hat{\Omega}_{1}$ and $\hat{\Omega}_{2}$ to start with. As we
mentioned, we apply the second half of the data to estimate the precision
matrix $\hat{\Omega}_{1}$ and $\hat{\Omega}_{2}$. In this section, we
discuss four commonly assumed covariance structures of $X$ itself and
provide corresponding estimators. We apply the same procedure to $Y$.

\subsubsection{Sparse Precision Matrices}

\label{sec:sparseprecision}

Precision matrix is closely connected to the undirected graphical model
which is a powerful tool to model the relationships among a large number of
random variables in a complex system. It is well known that recovering the
structure of an undirected Gaussian graph is equivalent to recovering the
support of the precision matrix. In this setting, it is natural to impose
sparse graph structure among variables in $X$ by assuming sparse precision
matrices $\Omega _{1}$. Many algorithms targeting on estimating sparse
precision matrix were proposed in literature. See, e.g. \cite{MB06}, \cite%
{friedman2008sparse}, \cite{cai11} and \cite{ren13}. In the current paper,
we apply the CLIME method to estimate $\Omega _{1}$. For details of the
algorithm, we refer to \cite{cai11}.

\subsubsection{Bandable Covariance Matrices}

Motivated by applications in time series, where there is a natural ``order"
on the variables, the bandable class of covariance matrices was proposed by 
\cite{BL08A}. In this setting, we assume that $\sigma _{ij}$ decay to zero
at certain rate as $\left\vert i-j\right\vert $ goes away from the diagonal.
Usually regularizing the sample covariance matrix by banding or tapering
procedures were applied in literature. We apply the tapering method proposed
in \cite{CZZH10}. Let $\omega =(\omega _{m})_{0\leq m\leq p-1}$ be a weight
sequence with $\omega _{m}$ given by 
\begin{equation}
\mathbb{\omega }_{m}=\left\{ 
\begin{array}{cc}
1, & \mbox{\rm when }m\leq k/2 \\ 
2-\frac{2m}{k}, & \mbox{\rm when }k/2<m\leq k \\ 
0, & \mbox{\rm Otherwise}%
\end{array}%
\right. ,  \label{tapering weight}
\end{equation}%
where $k$ is the bandwidth. The tapering estimator $\hat{\Sigma}_{1}$ of the
covariance matrix of $X$ is given by $\hat{\Sigma}_{1}=(\hat{\sigma}%
_{ij}^{sam}\mathbb{\omega }_{\left\vert i-j\right\vert })$, where $\hat{%
\sigma}_{ij}^{sam}$ is the $(i,j)$-th entry of the sample covariance matrix.
The bandwidth $k$ is chosen through cross-validation in practice. An alternative adaptive method was proposed by \cite{caiyuan2012}.  In the
end, our estimator is $\hat{\Omega}_{1}=\hat{\Sigma}_{1}^{-1}$.

\subsubsection{Toeplitz Covariance Matrices}

Toeplitz matrix is the symmetric matrix that the entries are constant along
the off-diagonals which are parallel to the main diagonal. Class of Toeplitz
covariance matrices arises naturally in the analysis of stationary
stochastic processes. If $X$ is a stationary process with autocovariance
sequence $(\alpha _{m})\equiv \left( \alpha _{0},\alpha _{1},\cdots ,\alpha
_{p-1},\cdots \right) ,$ then the covariance matrix $\Sigma _{1}=\left(
\sigma _{ij}\right) _{p_{1}\times p_{1}}$ has a Toeplitz structure $\sigma
_{ij}=\alpha _{\left\vert i-j\right\vert }$. In this setting, it is natural
to assume certain rate of decay of the autocovariance sequence. We apply the
following tapering method proposed in \cite{CRZ12}. Define $\tilde{\sigma}%
_{m}={\frac{1}{p-m}}\sum_{s-t=m}\hat{\sigma}_{st}^{sam},$ the average of
sample covariance along each off-diagonal. Then the tapering estimator $\hat{%
\Sigma}_{1}=\left( \hat{\sigma}_{st}\right) $ with bandwidth $k$ is defined
as $\hat{\sigma}_{st}=\omega _{|s-t|}\tilde{\sigma}_{|s-t|},$ where $\omega
=(\omega _{m})_{0\leq m\leq p-1}$ is defined in Equation (\ref{tapering
weight}). In practice, we pick bandwidth $k$ using cross-validation. The
final estimator of $\Omega _{1}$ is then defined as $\hat{\Omega}_{1}=\hat{%
\Sigma}_{1}^{-1}$.

\subsubsection{Sparse Covariance Matrices}

In many applications, there is no natural order on the variables like we
assumed in bandable and Toeplitz covariance matrices. In this setting,
permutation-invariant estimators are favored and general sparsity assumption
is usually imposed on the whole covariance matrix, i.e. most of entries in
each row/column of covariance matrix are zero or negligible. We apply a hard
thresholding procedure proposed in \cite{BL08B} under this assumption.
Again, let $\hat{\sigma}_{ij}^{sam}$ be the $(i,j)$-th entry of the sample
covariance matrix of $X$. The thresholding estimator $\hat{\Sigma}%
_{1}=\left( \hat{\sigma}_{st}\right) $ is given by $\hat{\sigma}_{ij}=\hat{%
\sigma}_{ij}^{sam}I\left( |\hat{\sigma}_{ij}^{sam}|\geq \gamma \sqrt{\frac{%
\log p}{n}}\right) $ for some constant $\gamma $ which is chosen through
cross-validation. In the end, our estimator is $\hat{\Omega}_{1}=\hat{\Sigma}%
_{1}^{-1}$.

\section{Statistical Properties and Optimality}

\label{sec:theory}

In this section, we present the statistical properties and optimality of our
proposed estimator. We first present the convergence rates of our procedure,
and then we provide a minimax lower bound for a wide range of parameter
spaces. In the end, we can see when estimating the nuisance parameters is
not harder than estimating canonical direction pair, the rates of
convergence match the minimax lower bounds. Hence we obtain the minimax
rates of convergence for a range of sparse parameter spaces.

\subsection{Convergence Rates}

Notice that our model is fully determined by the parameter $%
(\Sigma_1,\Sigma_2,\lambda,\theta,\eta)$, among which we are interested in
estimating $(\theta,\eta)$. To achieve statistical consistency, we need some
assumptions on the interesting part $(\theta,\eta)$ and nuisance part $%
(\Sigma_1,\Sigma_2,\lambda)$.

\noindent\textbf{Assumption A - Sparsity Condition on $(\theta,\eta)$}:

We assume $\theta $ and $\eta $ are in the weak $l_{q}$ ball, with $0\leq
q\leq 2$. i.e.%
\begin{equation*}
\left\vert \theta _{\left( k\right) }\right\vert ^{q}\leq s_{1}k^{-1}%
\mbox{\rm ,
}\left\vert \eta _{\left( k\right) }\right\vert ^{q}\leq s_{2}k^{-1}%
\mbox{\rm ,}
\end{equation*}%
where $\theta _{\left( k\right) }$ is the $k$-th largest coordinate by
magnitude. Let $p=p_{1}\vee p_{2}$ and $s=s_{1}\vee s_{2}$. The sparsity
levels $s_{1}$ and $s_{2}$ satisfy the following condition, 
\begin{equation}
s=o\left( \left( \frac{n}{\log p}\right) ^{\frac{1}{2}-\frac{q}{4}}\right) %
\mbox{\rm .}  \label{SP condition}
\end{equation}

\begin{remark}
In general, we can allow $\theta $ to be in the weak $l_{q_{1}}$ ball and $%
\eta $ to be in the weak $l_{q_{2}}$ ball with $q_{1}\neq q_{2}$. In that
case, we require $s_{i}=o\left( \left( n/\log p\right) ^{\frac{1}{2}-\frac{%
q_{i}}{4}}\right) $ for $i=1,2$. There is no fundamental difference in the
analysis and procedures. For simplicity, in the paper we only consider $%
q_{1}=q_{2}$.
\end{remark}

\noindent\textbf{Assumption B - General Conditions on $(\Sigma_1,\Sigma_2,%
\lambda)$}:

\begin{enumerate}
\item We assume there exist constants $w$ and $W$, such that 
\begin{equation*}
0<w\leq \lambda_{\min}(\Sigma_i)\leq \lambda_{\max}(\Sigma_i)\leq W<\infty,
\end{equation*}
for $i=1,2$.

\item In order that the signals do not vanish, we assume the canonical
correlation is bounded below by a positive constant $C_{\lambda }$, i.e. $%
0<C_{\lambda }\leq \lambda \leq 1$.

\item Moreover, we require that estimators $(\hat{\Omega}_{1},\hat{\Omega}%
_{2})$ are consistent in the sense that 
\begin{equation}
\xi _{\Omega }=||\hat{\Omega}_{1}\Sigma _{1}-I||\vee ||\hat{\Omega}%
_{2}\Sigma _{2}-I||=o(1),  \label{Consistency condition}
\end{equation}%
with probability at least $1-O(p^{-2})$.
\end{enumerate}

\noindent\textbf{Loss Function}

For two vectors $a,b$, a natural way to measure the discrepancy of their
directions is the sin of the angle $|\sin\angle (a,b)|$, see \cite%
{johnstone09}. We consider the loss function $L(a,b)^2=2|\sin\angle(a,b)|^2$%
. It is easy to calculate that 
\begin{equation*}
L(a,b)=\left\|\frac{aa^T}{||a||^2}-\frac{bb^T}{||b||^2}\right\|_F.
\end{equation*}

The convergence rate of the CAPIT procedure is presented in the following
theorem.

\begin{thm}
\label{thm:main} Assume the Assumptions A and B above hold. Let $(\alpha
^{(k)},\beta ^{(k)})$ be the sequence from Algorithm \ref{alg:: ITSCCA},
with the initializer $(\alpha ^{(0)},\beta ^{(0)})$ calculated by Algorithm %
\ref{alg:: CTSCCA}. The thresholding levels are 
\begin{equation*}
t_{ij},\quad \gamma _{1}=c_{1}\sqrt{\frac{\log p}{n}},\quad \gamma _{2}=c_{2}%
\sqrt{\frac{\log p}{n}},
\end{equation*}%
for sufficiently large constants $(t_{ij},c_{1},c_{2})$. Then with
probability at least $1-O(p^{-2})$, we have 
\begin{equation*}
L(\alpha ^{(k)},\theta )^{2}\vee L(\beta ^{(k)},\eta )^{2}\leq C\left( s\Big(%
\frac{\log p}{n}\Big)^{1-q/2}+||(\hat{\Omega}_{1}\Sigma _{1}-I)\theta
||^{2}\vee ||(\hat{\Omega}_{2}\Sigma _{2}-I)\eta ||^{2}\right) ,
\end{equation*}%
for all $k=1,2,...,K$ with $K=O(1)$ and some constant $C>0$.
\end{thm}

\begin{remark}
Notice the thresholding levels depend on some unknown constants $%
(t_{ij},c_{1},c_{2})$. This is for the simplicity of presentation. A more
involved fully data-driven choice of thresholding levels are presented in
Section \ref{sec:threshconst} along with the proof.
\end{remark}

The upper bound in Theorem \ref{thm:main} implies that the estimation of
nuissance parameters $\hat{\Omega}_{i}$ affect the estimation canonical
directions in terms of $||(\hat{\Omega}_{1}\Sigma _{1}-I)\theta ||^{2}$ and $%
||(\hat{\Omega}_{2}\Sigma _{2}-I)\eta ||^{2}$. In Section \ref%
{sec::estimation of precision}, we discussed four different settings in
which certain structure assumptions are imposed on the nuance parameters $%
\Omega _{1}$ and $\Omega _{2}$. In the literature, optimal rates of
convergence in estimating ${\Omega}_{i}$ under spectral norm have been
established and can be applied here in each of the four settings, noting
that $||(\hat{\Omega}_{1}\Sigma _{1}-I)\theta ||^{2}\leq ||(\hat{\Omega}%
_{1}-\Omega _{1})||^{2}\left\Vert \Sigma _{1}\theta \right\Vert ^{2}\leq W||(%
\hat{\Omega}_{1}-\Omega _{1})||^{2}$. Due to the limited space, we only
discuss one setting in which we assume sparse precision matrix structure on $%
\Omega _{i}$.

Besides the first general condition in Assumption B, we assume each
row/column of $\Omega _{i}$ is in a weak $l_{q_{0}}$ ball with $0\leq
q_{0}\leq 1$. i.e. $\Omega _{i}\in \mathcal{G}_{q_{0}}\left(
s_{0},p_{i}\right) $ for $i=1,2$, where 
\begin{equation*}
\mathcal{G}_{q_{0}}\left( s_{0},p\right) =\left\{ \Omega =\left( \omega
_{ij}\right) _{p\times p}:\max_{j}\left\vert \omega _{j(k)}\right\vert
^{q_{0}}\leq s_{0}k^{-1}\mbox{\rm  for all }k\right\} ,
\end{equation*}%
and the matrix $l_{1}$ norm of $\Omega _{i}$ is bounded by some constant $%
\left\Vert \Omega _{i}\right\Vert _{l_{1}}\leq w^{-1}$. The notation $\omega
_{j(k)}$ means the $k$-th largest coordinate of $j$-th row of $\Omega $ in
magnitude. Recall that $p=p_{1}\vee p_{2}$. Under the assumptions that $%
s_{0}^{2}=O\left( \left( n/\log p\right) ^{1-q_{0}}\right) $, Theorem 2 in 
\cite{cai11} implies that CLIME estimator with an appropriate tuning
parameter attaining the following rate of convergence $||(\hat{\Omega}%
_{1}\Sigma _{1}-I)\theta ||^{2}$ with probability at least $1-O\left(
p^{-2}\right) ,$ 
\begin{equation*}
||(\hat{\Omega}_{1}\Sigma _{1}-I)\theta ||^{2}\leq W||(\hat{\Omega}%
_{1}-\Omega _{1})||^{2}\leq Cs_{0}^{2}\left( \frac{\log p}{n}\right)
^{1-q_{0}}\mbox{\rm .}
\end{equation*}%
Therefore we obtain the following corollary.

\begin{corollary}
\label{coro:Sparse Precision} Assume the Assumptions A and B holds, $\Omega
_{i}\in \mathcal{G}_{q_{0}}\left( s_{0},p_{i}\right) $ $i=1,2$, $\left\Vert
\Omega _{i}\right\Vert _{l_{1}}\leq w^{-1}$ and $s_{0}^{2}=O\left( \left(
n/\log p\right) ^{1-q_{0}}\right) $. Let $(\alpha ^{(k)},\beta ^{(k)})$ be
the sequence from Algorithm \ref{alg:: ITSCCA}, with the initializer $%
(\alpha ^{(0)},\beta ^{(0)})$ calculated by Algorithm \ref{alg:: CTSCCA} and 
$\hat{\Omega}_{i}$ obtained by applying CLIME procedure in \cite{cai11}. The
thresholding levels are the same as those in Theorem \ref{thm:main}. Then
with probability at least $1-O(p^{-2})$, we have 
\begin{equation*}
L(\alpha ^{(k)},\theta )^{2}\vee L(\beta ^{(k)},\eta )^{2}\leq C\left( s\Big(%
\frac{\log p}{n}\Big)^{1-q/2}+s_{0}^{2}\Big(\frac{\log p}{n}\Big)%
^{1-q_{0}}\right) ,
\end{equation*}%
for all $k=1,2,...,K$ with $K=O(1)$ and some constant $C>0$.
\end{corollary}

\begin{remark}
It can be seen from the analysis that similar upper bounds hold in Corollary %
\ref{coro:Sparse Precision} with probability $1-O(p^{-h})$ by picking
different thresholding constants in Algorithms \ref{alg:: ITSCCA}, \ref%
{alg:: CTSCCA} and CLIME procedure for any $h>0$. Assuming that $n=o(p^{h}),$
the boundedness of loss function implies that Corollary \ref{coro:Sparse
Precision} is valid in the risk sense.
\end{remark}

\subsection{Minimax Lower Bound}

\label{sec:lowerbound}

\label{subsec:lower}

In this section, we establish a minimax lower bound in a simpler setting in
which we know the covariance matrices $\Sigma _{1}$ and $\Sigma _{2}$. We
assume $\Sigma _{i}=I_{p_{i}\times p_{i}}$ for $i=1,2$ for simplicity.
Otherwise, we can transfer the data accordingly and make $\Sigma
_{i}=I_{p_{i}\times p_{i}}$. The purpose of establishing this minimax lower
bound is to measure the difficulty of estimation problems in sparse CCA
model. In view of the upper bound given in Theorem \ref{thm:main} by the
iterative thresholding procedure Algorithms \ref{alg:: ITSCCA} and \ref%
{alg:: CTSCCA}, this lower bound is minimax rate optimal under conditions
that estimating nuisance precision matrices is not harder than estimating
the canonical direction pair. Consequently, assuming some general structures
on the nuance parameters $\Sigma _{1}$ and $\Sigma _{2}$, we establish the
minimax rates of convergence for estimating the canonical directions.

Before proceeding to the precise statements, we introduce the parameter
space of $\left( \theta ,\eta ,\lambda \right) $ in this simpler setting.
Define 
\begin{equation}
\mathcal{F}_{q}^{p_{1},p_{2}}\left( s_{1},s_{2},C_{\lambda }\right) =\left\{ 
\begin{array}{c}
N\left( 0,\Sigma \right) :\Sigma \mbox{ is specified in
(\ref{eq:model}) },\lambda \in (C_{\lambda },1) \\ 
\Sigma _{i}=I_{p_{i}\times p_{i}}\mbox{\rm , }i=1,2, \\ 
\left\vert \theta \right\vert _{\left( k\right) }^{q}\leq s_{1}k^{-1}%
\mbox{\rm ,
}\left\vert \eta \right\vert _{\left( k\right) }^{q}\leq s_{2}k^{-1}%
\mbox{\rm ,}\mbox{\rm  for all }k.%
\end{array}%
\right\} .  \label{Parameter Space}
\end{equation}%
In the sparsity class (\ref{Parameter Space}), the covariance matrices $%
\Sigma _{i}=I_{p_{i}\times p_{i}}$ for $i=1,2$ are known and unit vectors $%
\theta ,\eta $ are in the weak $l_{q}$ ball, with $0\leq q\leq 2$. We allow
the dimensions of two random vectors $p_{1}$ and $p_{2}$ to be very
different and only require that $\log p_{1}$ and $\log p_{2}$ are comparable
with each other, 
\begin{equation}
\log p_{1}\asymp \log p_{2}.  \label{Assumption on log dimension}
\end{equation}%
Remember $s=s_{1}\vee s_{2}$ and $p=p_{1}\vee p_{2}$.

\begin{thm}
\label{thm:lower} For any $q\in \left[ 0,2\right] ,$ we assume that $%
s_{i}\left( \frac{n}{\log p_{i}}\right) ^{q/2}=o(p_{i})$ for $i=1,2$ and (%
\ref{Assumption on log dimension}) holds. Moreover, we also assume $s\left( 
\frac{\log p}{n}\right) ^{1-\frac{q}{2}}\leq c_{0}$, for some constant $%
c_{0}>0$. Then we have 
\begin{equation*}
\inf_{(\hat{\theta},\hat{\eta})}\sup_{P\in \mathcal{F}}\mathbb{E}_{P}\left(
L^{2}(\hat{\theta},\theta )\vee L^{2}(\hat{\eta},\eta )\right) \geq Cs\left( 
\frac{\log p}{n}\right) ^{1-q/2},
\end{equation*}%
where $\mathcal{F}=\mathcal{F}_{q}^{p_{1},p_{2}}\left(
s_{1},s_{2},C_{\lambda }\right) $ and $C$ is a constant only depending on $q$
and $C_{\lambda }$.
\end{thm}

Theorem \ref{thm:lower} implies the minimaxity for the sparse CCA problem
when the covariance matrices $\Sigma _{1}$ and $\Sigma _{2}$ are unknown.
The lower bound directly follows from Theorem \ref{thm:lower} and the upper
bound follows from Theorem \ref{thm:main}. Define the parameter space 
\begin{equation*}
\mathcal{F}_{q,q_{0}}^{p_{1},p_{2}}\left( s_{0},s_{1},s_{2},C_{\lambda
},w,W\right) =\left\{ 
\begin{array}{c}
N\left( 0,\Sigma \right) :\Sigma \mbox{ is specified in
(\ref{eq:model}) },\lambda \in (C_{\lambda },1), \\ 
\Sigma _{i}^{-1}\in \mathcal{G}_{q_{0}}\left( s_{0},p_{i}\right) ,W^{-1}\leq
\lambda _{\min }(\Sigma _{i}^{-1}),\left\Vert \Sigma _{i}^{-1}\right\Vert
_{l_{1}}\leq w^{-1}, \\ 
\left\vert \theta \right\vert _{\left( k\right) }\leq s_{1}k^{-1}%
\mbox{\rm ,
}\left\vert \eta \right\vert _{\left( k\right) }\leq s_{2}k^{-1}%
\mbox{\rm
for all }k.%
\end{array}%
\right\} .
\end{equation*}%
Since $\mathcal{F}_{q}^{p_{1},p_{2}}\left( s_{1},s_{2},C_{\lambda }\right)
\subset \mathcal{F}_{q,q_{0}}^{p_{1},p_{2}}\left(
s_{0},s_{1},s_{2},C_{\lambda },w,W\right) $, the lower bound for the smaller
space holds for the larger one. Combining the Corollary \ref{coro:Sparse
Precision} and the minimax lower bound in Theorem \ref{thm:lower}, we obtain
that the minimax rate of convergence of estimating canonical directions over
parameter spaces $\mathcal{F}_{q,q_{0}}^{p_{1},p_{2}}\left(
s_{0},s_{1},s_{2},C_{\lambda },w,W\right) $.

\begin{corollary}
Under the assumptions in Corollary \ref{coro:Sparse Precision} and Theorem %
\ref{thm:lower} and assume $n=o(p^h)$ for some $h>0$, we have 
\begin{equation*}
\inf_{(\hat{\theta},\hat{\eta})}\sup_{P\in \mathcal{F}}\mathbb{E}_P\Bigg(L(%
\hat{\theta},\theta)^2\vee L(\hat{\eta},\eta)^2\Bigg)\asymp s\Bigg(\frac{%
\log p}{n}\Bigg)^{1-q/2},
\end{equation*}
for $\mathcal{F}=\mathcal{F}_{q,q_{0}}^{p_{1},p_{2}}\left(
s_{0},s_{1},s_{2},C_{\lambda },w,W\right)$, provided that $s_0^2\Bigg(\frac{%
\log p}{n}\Bigg)^{1-q_0}\leq Cs\Bigg(\frac{\log p}{n}\Bigg)^{1-q/2}$ for
some constant $C>0$.
\end{corollary}

\section{Simulation Studies}

\label{sec:simu}

We present simulation results of our proposed method in this section. In the
first scenario, we assume the covariance structure is sparse, and in the
second scenario, we assume the precision structure is sparse. Comments on
both scenarios are addressed at the end of the section.

\subsection{Scenario I: Sparse Covariance Matrix}

In the first scenario, we consider covariance matrices $\Sigma_1$ and $%
\Sigma_2$ are sparse. More specifically, the covariance matrix $\Sigma_1 =
\Sigma_2 = (\sigma_{ij})_{1 \leq i, j \leq p} $ takes the form 
\begin{equation*}
\sigma_{ij} = \rho^{|i-j|} \quad \mbox{\rm with} \quad \rho = 0.3 .
\end{equation*}
The canonical pair $(\theta,\eta)$ is generated by normalizing a vector
taking the same value at the coordinates $(1,6,11,16,21)$ and zero elsewhere
such that $\theta^T \Sigma_1 \theta = 1$ and $\eta^T \Sigma_2 \eta = 1$. The
canonical correlation $\lambda$ is taken as $0.9$. We generate the $2n
\times p$ data matrices $X$ and $Y$ jointly from (\ref{eq:model}). As
described in the methodology section, we split the data into two halves. In
the first step, we estimate the precision matrices $\Omega _{1}$ and $\Omega
_{2}$ using the first half of the data. Note that this covariance matrix has
a Toeplitz structure. We estimate the covariance matrix under three
different assumptions: 1) we assume that the Toeplitz structure is known and
estimate $\hat{\Sigma}_{1}$ and $\hat{\Sigma}_{2}$ by the method proposed in 
\cite{CRZ12} (denoted as CAPIT+Toep); 2) we assume that it is known that
covariance $\sigma _{ij}$ decay as they move away from the diagonal and
estimate $\hat{\Sigma}_{1}$ and $\hat{\Sigma}_{2}$ by the tapering procedure
proposed in \cite{CZZH10} (denoted as CAPIT+Tap); 3) we assume only the
sparse structure is known and estimate $\hat{\Sigma}_{1}$ and $\hat{\Sigma}%
_{2}$ by hard thresholding \citep{BL08B} (denoted as CAPIT+Thresh). In the
end the estimators $\hat{\Omega}_{i}$ is given by $\hat{\Omega}_{i}=\hat{%
\Sigma}_{i}^{-1}$ for $i=1,2$.

To select the tuning parameters for different procedures, we further split
the first part of the data into a $2:1$ training set and tuning set. We select
the tuning parameters by minimizing the distance of estimated covariance
from the training set and sample covariance matrix of the tuning set in term
of the Frobenius norm. More specifically, the tuning parameters $k_{1}$ in
the Toeplitz method and $k_{2}$ in the Tapering method are selected through
a screening on numbers in the interval of $(1,p)$. The tuning parameter 
$\lambda $ in the Thresholding method is selected through a screening on 50
numbers in the interval of $[0.01,0.5]$.

After obtaining estimator $\hat{\Omega}_{1}$ and $\hat{\Omega}_{2}$, we
perform Algorithms \ref{alg:: ITSCCA} and \ref{alg:: CTSCCA} by using $\hat{%
\Sigma}_{12}$ estimated from the second half of the data. The thresholding
parameters $\gamma _{1}$ and $\gamma _{2}$ are set to be $2.5\sqrt{\frac{%
\log p}{n}}$ for the Tapering and Thresholding methods, while the
thresholding parameter $t_{ij}$ is set to be $2.5$ for all $(i,j)$. For the
Toeplitz method, the thresholding parameters $\gamma _{1}=\gamma _{2}=2\sqrt{%
\frac{\log p}{n}}$ while parameter $t_{ij}=2$ for all $(i,j)$. The resulted
estimator is denoted as $(\hat{\theta}_{[1]},\hat{\eta}_{[1]})$.

Then we swap the data, repeat the above procedures and obtain $(\hat{\theta}%
_{[2]},\hat{\eta}_{[2]})$. The final estimator $(\hat{\theta}, \hat{\eta})$
is the average of $(\hat{\theta}_{[1]},\hat{\eta}_{[1]})$ and $(\hat{\theta}%
_{[2]},\hat{\eta}_{[2]})$.

We compare our method with penalized matrix decomposition proposed by \cite%
{witten09} (denoted as PMD) and the vanilla singular vector decomposition
method for CCA (denoted as SVD). For PMD, we use the R function implemented
by the authors \citep{witten2013package}, which performs sparse CCA by $l_1$%
-penalized matrix decomposition and selects the tuning parameters using a
permutation scheme.

We evaluate the performance of different methods by the loss function $L(%
\hat{\theta},\theta)\vee L(\hat{\eta},\eta)$. The results from $100$
independent replicates are summarized in Table 1.

\begin{table}  
	\caption{Scenario I: Sparse covariance matrix. Estimation errors for $(\protect\theta, \protect\eta)$ as measured by $L(\hat{\protect\theta},\protect\theta)\vee L(\hat{\protect\eta},\protect\eta)$ based on the median
of 100 replications. Numbers in parentheses are the simulation median
absolute deviations.} 
\centering
\fbox{ 
\begin{tabular}{ccccccc}
$p_1 = p_2$ & $n$ & CAPIT+Toep & CAPIT+Tap & CAPIT+Thresh & PMD & SVD \\ \hline
200 & 750 & 0.11(0.03) & 0.12(0.06) & 0.11(0.03) & 0.16(0.03) & 0.32(0.01) \\ 
300 & 750 & 0.11(0.03) & 0.13(0.07) & 0.11(0.03) & 0.36(0.02) & 0.44(0.01) \\ 
200 & 1000 & 0.1(0.02) & 0.1(0.05) & 0.09(0.03) & 0.14(0.02) & 0.27(0.01) \\ 
500 & 1000 & 0.09(0.03) & 0.09(0.04) & 0.1(0.02) & 0.11(0.03) & 0.53(0.02) \\ \hline
\end{tabular}
}
\label{sparse_covariance}
\end{table}

\subsection{Scenario II: Sparse Precision Matrix}

\label{sec:sce2}

In the second scenario, we consider that the precision matrices $\Omega _{1}$
and $\Omega _{2}$ are sparse. In particular, $\Omega _{1}=\Omega
_{2}=(\omega _{ij})_{1\leq i,j\leq p}$ take the form: 
\begin{equation*}
\omega _{ij}=%
\begin{cases}
1 & \mbox{if }i=j \\ 
0.5 & \mbox{if }|i-j|=1 \\ 
0.4 & \mbox{if }|i-j|=2 \\ 
0 & \mbox{otherwise}.%
\end{cases}%
\end{equation*}
The canonical pair $(\theta,\eta)$ is the same as described in Scenario I.
We generate the $2n \times p$ data matrices $X$ and $Y$ jointly from (\ref%
{eq:model}).

As described in the methodology section, we split the data into two halves.
In the first step, we estimate the precision matrices by the CLIME proposed
in \cite{cai11} (denoted as CAPIT+CLIME). The tuning parameter $\lambda$ is
selected by maximizing the log-likelihood function. In the second step, we
perform Algorithms \ref{alg:: ITSCCA} and \ref{alg:: CTSCCA} with $\hat{%
\Sigma}_{12}$ estimated from the second half. The thresholding parameter $%
\gamma _{1}$ and $\gamma _{2}$ are set to be $1.5\sqrt{\frac{\log p}{n}}$
and $t_{ij}$ is set to be $1.5$. The resulted
estimator is denoted as $(\hat{\theta}_{[1]},\hat{\eta}_{[1]})$. Then we swap the data, repeat the above
procedures and obtain $(\hat{\theta}_{[2]},\hat{\eta}_{[2]})$. The final
estimator $(\hat{\theta}, \hat{\eta})$ is the average of $(\hat{\theta}%
_{[1]},\hat{\eta}_{[1]})$ and $(\hat{\theta}_{[2]},\hat{\eta}_{[2]})$.

For comparison, we also apply PMD and SVD in this case. The results from $%
100 $ independent replicates are summarized in Table 2.
A visualization of the estimation from a replicate in from the case $n=500,
p=200$ under Scenario II is shown in Figure \ref{scca_precision_plot}.

\begin{table}	
	\caption{Scenario II: Sparse precision matrix. Estimation errors for $(%
\protect\eta, \protect\theta)$ as measured by $L(\hat{\protect\theta},%
\protect\theta) \vee L(\hat{\protect\eta},\protect\eta)$ based on the median
of 100 replications. Numbers in parentheses are the simulation median
absolute deviations.}
\centering
\fbox{
\begin{tabular}{ccccc}
$p_1 = p_2$ & $n$ & CAPIT+CLIME & PMD & SVD  \\ \hline
200 & 500 & 0.41(0.35) & 1.41(0) & 0.52(0.03)   \\ 
200 & 750 & 0.2(0.05) & 1.19(0.33) & 0.39(0.02)   \\ 
500 &750 & 0.21(0.12) & 1.41(0) & 0.84(0.03)\\ \hline
\end{tabular}
}
 \label{sparse_precision}
\end{table}

\subsection{Discussion on the Simulation Results}

The above results (Table 1 and Table 2) show that our method outperforms the PMD method proposed
by \cite{witten09} and the vanilla SVD method \citep{hotelling36}. It is not
surprising that the SVD method does not perform better than our method
because of the sparse assumption in the signals. We focus our discussion on
the comparison of our method and the PMD method.

The PMD method is defined by the solution of the following optimization
problem 
\begin{equation*}
(\hat{\theta}_{PMD},\hat{\eta}_{PMD})=\arg \max_{(u,v)}\left\{ u^{T}\hat{%
\Sigma}_{12}v:||u||\leq 1,||v||\leq 1,||u||_{1}\leq c_{1},||v||_{1}\leq
c_{2}\right\} .
\end{equation*}%
As noted by \cite{witten09}, the PMD method approximates the covariance $%
\Sigma _{1}$ and $\Sigma _{2}$ by the identity matrices $I_{p_{1}\times
p_{1}}$ and $I_{p_{2}\times p_{2}}$. If we ignore the $l_{1}$
regularization, the population version of PMD is to maximize $u^{T}\Sigma
_{12}v$ subject to $||u||\vee ||v||\leq 1$, which gives the maximizer in the
direction of $(\Sigma _{1}\theta ,\Sigma _{2}\eta )$ instead of $(\theta
,\eta )$. When the covariance matrices $\Sigma _{1}$ and $\Sigma _{2}$ are
sufficiently sparse, $%
(\Sigma _{1}\theta ,\Sigma _{2}\eta )$ and $(\theta ,\eta )$ are close. This
explains that in Scenario I, the PMD method performs well. However, in
Scenario II, we assume the precision matrices $\Omega _{1}$ and $\Omega _{2}$
are sparse. In this case, the corresponding $\Sigma _{1}$ and $\Sigma _{2}$
are not necessarily sparse, implying that $(\Sigma _{1}\theta ,\Sigma
_{2}\eta )$ could be far away from $(\theta ,\eta )$. The PMD method is not
consistent  in this case, as is illustrated in Figure \ref%
{scca_precision_plot}. In contrast, our method takes advantage of the
sparsity of $\Omega _{1}$ and $\Omega _{2}$, and accurately recovers the
canonical directions.

\section{Real Data Analysis}

\label{sec:real}

DNA methylation plays an essential role in the transcriptional regulation 
\citep{vanderkraats2013discovering}. In tumor, DNA methylation patterns are
frequently altered. However, how these alterations contribute to the
tumorigenesis and how they affect gene expression and patient survival
remain poorly characterized. Thus it is of great interest to investigate the
relationship between methylation and gene expression and their interplay
with survival status of cancer patients. We applied the proposed method to a
breast cancer dataset from The Cancer Genome Atlas project \citep
{nature2012breast}. This dataset consists both DNA methylation and gene
expression data for 193 breast cancer patients. The DNA methylation was
measured from Illumina Human methylation 450 BeadChip, which contains
482,431 CpG sites that cover 96$\%$ of the genome-wide CpG islands. Since no
batch effect has either been reported from previous studies or been observed
from our analysis, we do not further process the data. For methylation data,
there are two popular metrics used to measure methylation levels, $\beta$%
-value and M-value statistics. $\beta$-value is defined as the proportion of
methylated probes at a CpG site. M-value is defined as the $\log2$ ratio of
the intensities of methylated probe versus un-methylated probe, which is
reported as approximately homoscedastic in a previous study \cite%
{du2010comparison}. We choose to use M-value for methylation data in our
analysis.

To investigate the relationship of methylation and gene expression and their
interplay with clinical outcomes, we follow the supervised sparse CCA
procedure suggested in \cite{witten2009extensions}. More specifically, we
first select methylation probes and genes that are marginally associated
with the disease free status by performing a screening on methylation and
gene expression data, respectively. There are 135 genes and 4907 methylation
probes marginally associated with disease free status with a P-value less
than 0.01. We further reduce the number of methylation probes to 3206 by
selecting the ones with sample variance greater than 0.5. Compared to the
sample size, the number of methylation probes is still too large. To control
the dimension of input data, we apply our methods to 135 genes with the
methylation probes on each chromosome separately. Since it is widely
believed that genes operate in biological pathways, the graph for gene
expression data is expected to be sparse. We apply the proposed procedure
under the sparse precision matrix setting (Section \ref{sec:sparseprecision}%
). As we have discussed in the simulation studies, the canonical correlation
structure under the sparse precision matrix setting cannot be estimated by
the current methods in the literature, such as PMD.

\begin{figure}
\centering
\includegraphics[width=6in]{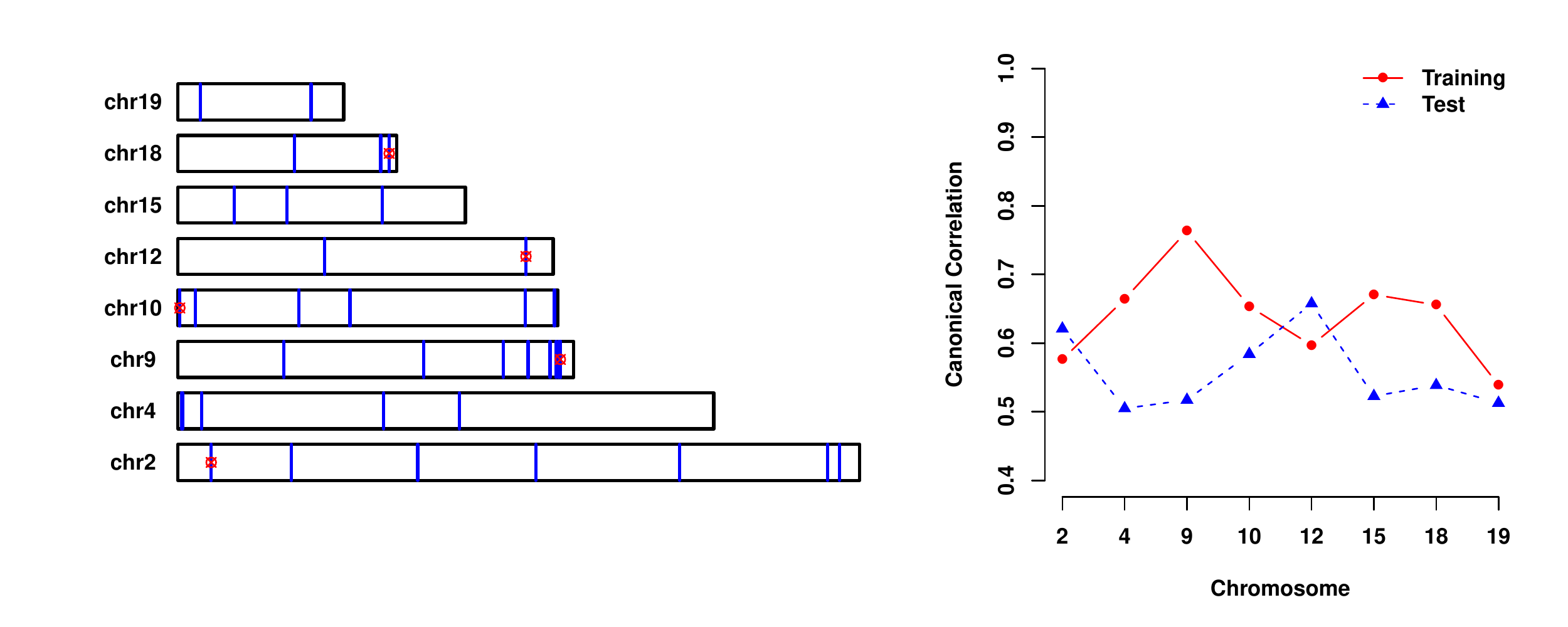}
\caption{Left: Visualization of the genomic coordinates of detected
methylation probes. Genes that represented by more than one probes are
highlighted by square symbols. Right: Canonical correlations of disease
associated genes and methylation probes on eight chromosomes in the training
set and the test set. }
\label{chr}
\end{figure}

\begin{table} 
	\caption{Sparse CCA results for methylation sites and gene expression that
are associated with disease free status for TCGA breast cancer data. In the
analysis, methylation and gene expression data are assumed to have sparse
precision matrix structure. Sparse CCA were performed on the same set of
genes with methylation probes on different chromosomes. Chromosomes with
canonical correlation greater than 0.5 on both the training and test set are
listed with the number of probes associated with disease free status and the
probes form the support of canonical directions. }  \label{cca_breast_cancer_data}
\centering
\fbox{\tiny 
\begin{tabular}{c|c|c}
\hline
& Number of & Genes \\ \cline{3-3}
Chromosome & probes & Methylation probes \\ \hline
&  & RRM2, ILF2, ORC6L, SUSD3, SHCBP1 \\ \cline{3-3}
2 & 269 & cg04799980, cg08022717, cg10142874, cg13052887, cg16297938,
cg24011073, cg26364080, cg27060355 \\ \hline
&  & SLC26A9, C15orf52, NPM2, DNAH11, RAB6B, LIN28, STC2 \\ \cline{3-3}
4 & 143 & cg04812351, cg14505741, cg15566751, cg15763121, cg17232991 \\ 
\hline
&  & RGS6, ORC6L, PTPRH, GPX2, QSOX2, NPM2, SCG3, RAB6B, L1CAM, STC2, REG1A
\\ \cline{3-3}
9 & 89 & cg01729066, cg02127980, cg03693099, cg13413384, cg13486627,
cg13847987, cg14004457, cg14443041, cg21123355 \\ \hline
&  & RRM2, SLC26A9, ORC6L, PTPRH, DNAH11, SCG3, LIN28, UMODL1, C11orf9 \\ 
\cline{3-3}
10 & 116 & cg00827318, cg01162610, cg01520297, cg03182620, cg14522790,
cg14999931, cg19302462 \\ \hline
&  & QSOX2, SCG3 \\ \cline{3-3}
12 & 175 & cg00417147, cg13074795, cg21881338 \\ \hline
&  & C15orf52, NPM2, DNAH11, SELE, RAB6B \\ \cline{3-3}
15 & 92 & cg11465404, cg18581777, cg21735516 \\ \hline
&  & ILF2, SPRR2D, ADCY4, RAB6B, C11orf9, REG1A, SHCBP1 \\ \cline{3-3}
18 & 37 & cg07740306, cg15531009, cg18935516, cg19363889 \\ \hline
&  & ORC6L, NPM2, GPR56 \\ \cline{3-3}
19 & 162 & cg06392698, cg06555246 \\ \hline
\end{tabular}
}
 \label{cca_breast_cancer_data}
\end{table}

\begin{table} 
    \caption{Detected methylation probes and their corresponding genes on
chromosome 9.} 
\centering
\fbox{\tiny 
\begin{tabular}{cccc}
Probe & Gene & Function & $\eta_i$ \\ \hline
cg01729066 & MIR600 & microRNA regulating estrogen factors & 0.282 \\ 
cg14004457 & MIR455 & microRNA regulating estrogen factors & 0.347 \\ 
cg02127980, cg13413384 & RXRA & retinoic X receptors & 0.269, 0.242 \\ 
cg03693099 & CEL & fat catalyzation and vitamin absorption & 0.286 \\ 
cg13486627 & RG9MTD3 & RNA (guanine-9-) methyltransferase & -0.479 \\ 
cg13847987 & ABL1 & a protein tyrosine kinase functioned in cell
differentiation and stress response & 0.334 \\ 
cg21123355 & VAV2 & a member of the VAV guanine nucleotide exchange factor
family of oncogenes & 0.312 \\ 
cg14443041 & Intergenic region &  & 0.384 \\ \hline
\end{tabular}
}
 \label{cca_breast_cancer_data2}
\end{table}

For the purpose of interpretation, the tuning parameters are selected such
that a sparse representation of $(\hat{\theta}, \hat{\eta})$ is obtained
while the canonical correlation is high. More specifically, we require the
number of non-zero genes or probes is less than 10 for each chromosome. We
split the data into two halves as a test set and a training set. Then we
applied the proposed procedure on the training set. To remove false
discoveries, we required the canonical correlation on the test set is
greater than 0.5. In total, there are eight chromosomes that have
methylation probes satisfying the above criteria (shown in Figure \ref{chr}%
). In Table 3, we list genes and methylation
probes on each chromosome that form the support of detected canonical
directions. A further examination of the genomic coordinates of detected
methylation probes reveal the physical closeness of some probes. Some
detected probes correspond to the same gene. LPIN1 on chromosome 2, RXRA on
chromosome 9, DIP2C on chromosome 10, AACS on chromosome 12, and NFATC1 on
chromosome 18 are represented by more than one methylation probes (shown in
Figure \ref{chr}). Moreover, 16 of the 25 genes listed in Table 3 are detected more than once as candidate genes
associated with methylation probes. ORC6L, RRM2, RAB6B
are independently detected from four chromosomes. All these genes have been
proposed as prognosis signature for the metastasis of breast cancer \citep
{weigelt2005molecular, ma2003gene, van2002gene}. Our results suggest the
interplay of their expression with detected methylation sites. We list the functional annotation of probes detected on Chromosome 9 in Table 4
\footnote{ \scriptsize{ The corresponding canonical vector on genes RGS6, ORC6L, PTPRH, GPX2, QSOX2,
NPM2, SCG3, RAB6B, L1CAM, STC2, REG1A is $(-0.252, -0.27, -0.286, -0.244,
-0.35, -0.282, -0.256, -0.367, -0.256, 0.357, -0.358)$.}}.

In this analysis, we assume there is one pair of canonical directions
between methylation and gene expression. We note that when the underlying
canonical correlation structure is low-rank, the pair of canonical
directions obtained from the proposed method lie in the subspace of true
canonical directions. The extracted canonical directions can still be used
to identify sets of methylate sites that are correlated with gene expression.

\section{Proof of Main Theorem}

\label{sec:proof}

We provide the proof of Theorem \ref{thm:main} in this section, which is
based on the construction of an oracle sequence. The proof is similar in
nature to that in \cite{ma2013sparse} which focuses on the sparse PCA
setting. Specifically, we are going to first define the strong signal sets,
and then define an oracle sequence $(\alpha ^{(k),ora},\beta ^{(k),ora})$
produced by Algorithms \ref{alg:: ITSCCA} and \ref{alg:: CTSCCA} only
operating on the strong signal sets. We then show the desired rate of
convergence for this oracle sequence. In the end, a probabilistic argument
shows that with the help of thresholding, the oracle sequence is identical
to the data-driven sequence with high probability. In the following proof,
we condition on the second half of the data $%
(X_{n+1},Y_{n+1}),...,(X_{2n},Y_{2n})$ and the event $\{||\hat{\Omega}%
_{1}\Sigma _{1}-I||\vee ||\hat{\Omega}_{2}\Sigma _{2}-I||=o(1)\}$. The
\textquotedblleft with probability" argument is understood to be with
conditional probability unless otherwise specified. We keep using the
notations $p=p_{1}\vee p_{2}$ and $s=s_{1}\vee s_{2}$.

\subsection{Construction of the Oracle Sequence}

We first define the strong signal set by 
\begin{equation}
H_1=\left\{k: |\alpha_k|\geq\delta_1\sqrt{\frac{\log p_1}{n}}\right\},\quad
H_2=\left\{k:|\beta_k|\geq\delta_2\sqrt{\frac{\log p_2}{n}}\right\}.
\label{eq:signalset}
\end{equation}
We denote their complement in $\{1,2,...,p_1\}$ and $\{1,2,...,p_2\}$ by $%
L_1 $ and $L_2$ respectively. Then we define the oracle version of $\hat{A}$
by taking those coordinates with strong signals. That is, 
\begin{equation*}
\hat{A}^{ora}=%
\begin{pmatrix}
\hat{A}_{H_1H_2} & 0 \\ 
0 & 0%
\end{pmatrix}%
.
\end{equation*}
We construct the oracle initializer $(\alpha^{(0),ora},\beta^{(0),ora})$
based on an oracle version of Algorithm \ref{alg:: CTSCCA} with the sets $%
B_1 $ and $B_2$ replaced by $B_1^{ora}=B_1\cap H_1$ and $B_2^{ora}=B_2\cap
H_2$. It is clear that $\alpha^{(0),ora}_{L_1}=0$ and $%
\beta^{(0),ora}_{L_2}=0$. Feeding the oracle initializer $%
(\alpha^{(0),ora},\beta^{(0),ora})$ and the matrix $\hat{A}^{ora}$ into
Algorithm \ref{alg:: ITSCCA}, we get the oracle sequence $%
(\alpha^{(k),ora},\beta^{(k),ora})$.

\subsection{Data-Driven Thresholding}

\label{sec:threshconst}

Algorithms \ref{alg:: ITSCCA} and \ref{alg:: CTSCCA} contain thresholding
levels $\gamma_1$, $\gamma_2$ and $t_{ij}$. These tuning parameters can be
specified by users. However, our theory is based on fully data-driven tuning
parameters depending on the matrix $\hat{\Omega}_1=(\hat{\omega}_{1,ij})$
and $\hat{\Omega}_2=(\hat{\omega}_{2,ij})$. In particular, we use 
\begin{equation*}
t_{ij}=\frac{20\sqrt{2}}{9}\left( \sqrt{|| \hat{\Omega}_{1}|| \hat{\omega}%
_{2,jj}}+\sqrt{||\hat{\Omega}_{2}|| \hat{\omega}_{1,ii}}+\sqrt{\hat{\omega}%
_{1,ii}\hat{\omega}_{2,jj}}+\sqrt{8|| \hat{\Omega}_{1}|| || \hat{\Omega}%
_{2}|| /3}\right) ,
\end{equation*}
and 
\begin{equation*}
\gamma_1=\Big(0.17\min_{i,j}t_{ij} ||\hat{\Omega}_2||^{1/2}+2.1||\hat{\Omega}%
_2||^{1/2}||\hat{\Omega}_1||^{1/2}+7.5||\hat{\Omega}_2||\Big)\sqrt{\frac{%
\log p}{n}},
\end{equation*}
\begin{equation*}
\gamma_2=\Big(0.17\min_{i,j}t_{ij} ||\hat{\Omega}_1||^{1/2}+2.1||\hat{\Omega}%
_1||^{1/2}||\hat{\Omega}_2||^{1/2}+7.5||\hat{\Omega}_1||\Big)\sqrt{\frac{%
\log p}{n}}.
\end{equation*}
The constants $(\delta_1,\delta_2)$ in (\ref{eq:signalset}) are set as $%
\delta_1=\delta_2=0.08 w^{1/2}\min_{i,j}t_{ij}.$ Such choice of thresholding
levels are used in both the estimating sequence $(\alpha^{(k)},\beta^{(k)})$
and the oracle sequence $(\alpha^{(k),ora},\beta^{(k),ora})$.

\subsection{Outline of Proof}

The proof of Theorem \ref{thm:main} can be divided into the following three
steps.

\begin{enumerate}
\item Show that $\hat{A}^{ora}$ is a good approximation of $%
A=\lambda\alpha\beta^T$ in the sense that their first pairs of singular
vectors are close. Namely, let $(\hat{\alpha}^{ora},\hat{\beta}^{ora})$ be
the first pair of singular vectors of $\hat{A}^{ora}$. We are going to bound 
$L(\hat{\alpha}^{ora},\alpha)$ and $L(\hat{\beta}^{ora},\beta)$.

\item Show that the oracle sequence $(\alpha^{(k),ora},\beta^{(k),ora})$
converges to $(\hat{\alpha}^{ora},\hat{\beta}^{ora})$ after finite steps of
iterations.

\item Show that the estimating sequence $(\alpha^{(k)},\beta^{(k)})$ and the
oracle sequence $(\alpha^{(k),ora},\beta^{(k),ora})$ are identical with high
probability up to the necessary number of steps for convergence. Here, we
need to first show that the oracle initializer $(\alpha^{(0),ora},%
\beta^{(0),ora})$ is identical to the actual $(\alpha^{(0)},\beta^{(0)})$.
Then we are going to show the thresholding step in Algorithm \ref{alg::
ITSCCA} kills all the small coordinates so that the oracle sequence is
identical to the estimating sequence under iteration.
\end{enumerate}

\subsection{Preparatory Lemmas}

In this part, we present lemmas corresponding to the three steps in the
outline of proof. The first lemma corresponds to Step 1.

\begin{lemma}
\label{lem:approxbias} Under Assumptions A and B, we have 
\begin{equation*}
L(\hat{\alpha}^{ora},\alpha )^{2}\vee L(\hat{\beta}^{ora},\beta )^{2}\leq C%
\Bigg(s\Big(\frac{\log p}{n}\Big)^{1-q/2}+||\theta -\alpha |^{2}|\vee ||\eta
-\beta ||^{2}\Bigg),
\end{equation*}%
with probability at least $1-O(p^{-2})$ for some constant $C>0$.
\end{lemma}

Let $(\hat{l}_{1},\hat{l}_{2})$ be the first and second singular values of $%
\hat{A}^{ora}$. Then we have the following results, corresponding to Step 2.

\begin{lemma}
\label{lem:iteration} Under Assumptions A and B, we have 
\begin{equation*}
L({\alpha}^{(1),ora},\hat{\alpha}^{ora})^2\leq 4\left|\frac{\hat{l}_2}{\hat{l%
}_1}\right|^2+\frac{32\gamma_1^2|H_1|}{|\hat{l}_1|^2},
\end{equation*}
\begin{equation*}
L({\beta}^{(1),ora},\hat{\beta}^{ora})^2\leq 4\left|\frac{\hat{l}_2}{\hat{l}%
_1}\right|^2+\frac{32\gamma_2^2|H_2|}{|\hat{l}_1|^2},
\end{equation*}
for $k=1$, and 
\begin{equation*}
L({\alpha}^{(k),ora},\hat{\alpha}^{ora})^2\leq \max\Bigg(4\left|\frac{\hat{l}%
_2}{\hat{l}_1}\right|^2\frac{64\gamma_2^2|H_2|}{|\hat{l}_1|^2}+\frac{%
64\gamma_1^2|H_1|}{|\hat{l}_1|^2},\Big(32\left|\frac{\hat{l}_2}{\hat{l}_1}%
\right|^4\Big)^{[k/2]}\Bigg),
\end{equation*}
\begin{equation*}
L({\beta}^{(k),ora},\hat{\beta}^{ora})^2\leq \max\Bigg(4\left|\frac{\hat{l}_2%
}{\hat{l}_1}\right|^2\frac{64\gamma_1^2|H_1|}{|\hat{l}_1|^2}+\frac{%
64\gamma_2^2|H_2|}{|\hat{l}_1|^2},\Big(32\left|\frac{\hat{l}_2}{\hat{l}_1}%
\right|^4\Big)^{[k/2]}\Bigg),
\end{equation*}
for all $k\geq 2$ with probability at least $1-O(p^{-2})$.
\end{lemma}

The quantities $|\hat{l}_1|$, $|\hat{l}_2/\hat{l}_1|$, $|H_1|$ and $|H_2|$
are determined by the following lemma.

\begin{lemma}
\label{lem:eigengap} With probability at least $1-O(p^{-2})$, 
\begin{eqnarray*}
|\hat{l}_2|^2 &\leq& C\Bigg(s\Big(\frac{\log p}{n}\Big)^{1-q/2}+||\theta-%
\alpha||^2\vee||\eta-\beta||^2\Bigg), \\
|\hat{l}_1|^{-2}&\leq& C.
\end{eqnarray*}
Moreover, 
\begin{equation*}
|H_1|\leq C\Bigg(s_1\Big(\frac{\log p_1}{n}\Big)^{-q/2}+\Big(\frac{\log p_1}{%
n}\Big)^{-1}||\theta-\alpha||^2\Bigg),
\end{equation*}
\begin{equation*}
|H_2|\leq C\Bigg(s_2\Big(\frac{\log p_2}{n}\Big)^{-q/2}+\Big(\frac{\log p_2}{%
n}\Big)^{-1}||\eta-\beta||^2\Bigg),
\end{equation*}
for some constant $C>0$.
\end{lemma}

Finally, we show that the oracle sequence and the actual sequence are
identical with high probability, corresponding to Step 3. For the
initializer, we have the following lemma.

\begin{lemma}
\label{lem:same0} Under Assumptions A and B, we have $B_1^{ora}=B_1 $ and $%
B_2^{ora}=B_2$ with probability at least $1-O(p^{-2})$. Thus, $%
(\alpha^{(0),ora},\beta^{(0),ora})=(\alpha^{(0)},\beta^{(0)})$.
\end{lemma}

We proceed to analyze the sequence for $k\geq 1$ using mathematical
induction. By iteration in Algorithm \ref{alg:: ITSCCA}, we have 
\begin{equation*}
\alpha^{(k),ora}=\frac{T(\hat{A}^{ora}\beta^{(k-1),ora},\gamma_1)}{||T(\hat{A%
}^{ora}\beta^{(k-1),ora},\gamma_1)||}, \quad \beta^{(k),ora}=\frac{T(\hat{A}%
^{ora,T}\alpha^{(k-1),ora},\gamma_2)}{||T(\hat{A}^{ora,T}\alpha^{(k-1),ora},%
\gamma_2)||}.
\end{equation*}
Suppose we have $(\alpha^{(k-1),ora},\beta^{(k-1),ora})=(\alpha^{(k-1)},%
\beta^{(k-1)})$. Then as long as 
\begin{eqnarray}
T(\hat{A}^{ora}\beta^{(k-1),ora},\gamma_1) &=& T(\hat{A}\beta^{(k-1),ora},%
\gamma_1)  \label{eq:thresh1} \\
T(\hat{A}^{ora,T}\alpha^{(k-1),ora},\gamma_2) &=& T(\hat{A}%
^T\alpha^{(k-1),ora},\gamma_2),  \label{eq:thresh2}
\end{eqnarray}
we have $(\alpha^{(k),ora},\beta^{(k),ora})=(\alpha^{(k)},\beta^{(k)})$.
Hence, it is sufficient to prove (\ref{eq:thresh1}) and (\ref{eq:thresh2}).
Then, the result follows from mathematical induction. Without loss of
generality, we analyze (\ref{eq:thresh1}) as follows. Since $%
\beta_{L_2}^{(0),ora}=0$, we may assume $\beta_{L_2}^{(k-1),ora}=0$ at the $%
k $-th step. The vectors $\hat{A}^{ora}\beta^{(k-1),ora}$ and $\hat{A}%
\beta^{(k-1),ora}$ are respectively 
\begin{equation*}
\hat{A}^{ora}\beta^{(k-1),ora}=%
\begin{pmatrix}
\hat{A}_{H_1H_2} & 0 \\ 
0 & 0%
\end{pmatrix}%
\begin{pmatrix}
\beta^{(k-1),ora}_{H_2} \\ 
0%
\end{pmatrix}%
=%
\begin{pmatrix}
\hat{A}_{H_1H_2}\beta^{(k-1),ora}_{H_2} \\ 
0%
\end{pmatrix}%
,
\end{equation*}
\begin{equation*}
\hat{A}\beta^{(k-1),ora}=%
\begin{pmatrix}
\hat{A}_{H_1H_2} & \hat{A}_{H_1L_2} \\ 
\hat{A}_{L_1H_2} & \hat{A}_{L_1L_2}%
\end{pmatrix}%
\begin{pmatrix}
\beta^{(k-1),ora}_{H_2} \\ 
0%
\end{pmatrix}%
=%
\begin{pmatrix}
\hat{A}_{H_1H_2}\beta^{(k-1),ora}_{H_2} \\ 
\hat{A}_{L_1H_2}\beta_{H_2}^{(k-1),ora}%
\end{pmatrix}%
.
\end{equation*}
Hence, as long as $||\hat{A}_{L_1H_2}\beta_{H_2}^{(k-1),ora}||_{\infty}\leq%
\gamma_1$, (\ref{eq:thresh1}) holds. Similarly, as long as $||\hat{A}%
_{H_1L_2}^T\alpha_{H_1}^{(k-1),ora}||_{\infty}\leq\gamma_2$, (\ref%
{eq:thresh2}) holds. This is guaranteed by the following lemma.

\begin{lemma}
\label{lem:same} For any sequence of unit vectors $(a^{(k)},b^{(k)})\in 
\mathbb{R}^{\left\vert H_{1}\right\vert }\times \mathbb{R}^{\left\vert
H_{2}\right\vert }$, with $k=1,2,...,K$ for some $K=O(1)$. We assume that
they only depend on $\hat{A}_{H_{1}H_{2}}$. Then, under the current choice
of $(\gamma _{1},\gamma _{2})$, we have 
\begin{equation*}
||\hat{A}_{L_{1}H_{2}}b^{(k)}||_{\infty }\leq \gamma _{1},\quad ||\hat{A}%
_{H_{1}L_{2}}^{T}a^{(k)}||_{\infty }\leq \gamma _{2},
\end{equation*}%
for all $k=1,...,K$ with probability at least $1-O(p^{-2})$.
\end{lemma}

\subsection{Proof of Theorem \protect\ref{thm:main}}

\begin{proof}[Proof of Theorem \ref{thm:main}]
In the following proof, $C$ denotes a generic constant which may vary from line to line.
Without loss of generality, we prove convergence of $\alpha^{(k)}$. By Lemma \ref{lem:same0} and Lemma \ref{lem:same}, $\alpha^{(k)}=\alpha^{(k),ora}$ for all $k=1,2,...,K$ with probability $1-O(p^{-2})$. Hence, it is sufficient to prove convergence of $\alpha^{(k),ora}$. Lemma \ref{lem:iteration} implies that for $k=2,3,...,K$,
\begin{eqnarray*}
L(\alpha^{(k),ora},\hat{\alpha}^{ora})^2 &\leq& \max\Bigg(4\left|\frac{\hat{l}_2}{\hat{l}_1}\right|^2\frac{64\gamma_2^2|H_2|}{|\hat{l}_1|^2}+\frac{64\gamma_1^2|H_1|}{|\hat{l}_1|^2},\Big(32\left|\frac{\hat{l}_2}{\hat{l}_1}\right|^4\Big)^{[k/2]}\Bigg) \\
&\leq& \max\Bigg(4\left|\frac{\hat{l}_2}{\hat{l}_1}\right|^2\frac{64\gamma_2^2|H_2|}{|\hat{l}_1|^2}+\frac{64\gamma_1^2|H_1|}{|\hat{l}_1|^2},32\left|\frac{\hat{l}_2}{\hat{l}_1}\right|^4\Bigg)
\end{eqnarray*}
According to Lemma \ref{lem:eigengap}, we have
$$
32\left|\frac{\hat{l}_2}{\hat{l}_1}\right|^4 \leq C\Bigg(s\Big(\frac{\log p}{n}\Big)^{1-q/2}+||\theta-\alpha||^2\vee||\beta-\eta||^2\Bigg)^2.
$$
We also have
\begin{eqnarray*}
&& 4\left|\frac{\hat{l}_2}{\hat{l}_1}\right|^2\frac{64\gamma_2^2|H_2|}{|\hat{l}_1|^2}+\frac{64\gamma_1^2|H_1|}{|\hat{l}_1|^2} \\
&\leq& C\Big(\gamma_2^2|H_2|+\gamma_1^2|H_1|\Big) \\
&\leq& C\frac{\log p}{n}\Bigg(s\Big(\frac{\log p}{n}\Big)^{-q/2}+\Big(\frac{\log p}{n}\Big)^{-1}||\theta-\alpha||^2\Bigg) \\
&\leq& C\Bigg(s\Big(\frac{\log p}{n}\Big)^{1-q/2}+||\theta-\alpha||^2\Bigg).
\end{eqnarray*}
Hence, the desired bound holds for $k=2,3,...,K$.
For $k=1$, by Lemma \ref{lem:iteration}, we have
$$L({\alpha}^{(1),ora},\hat{\alpha}^{ora})^2\leq 4\left|\frac{\hat{l}_2}{\hat{l}_1}\right|^2+\frac{32\gamma_1^2|H_1|}{|\hat{l}_1|^2}\leq C\Bigg(s\Big(\frac{\log p}{n}\Big)^{1-q/2}+||\theta-\alpha||^2\Bigg).$$
Therefore, we have proved the bound of $L(\alpha^{(k),ora},\alpha^{ora})^2$ for all $k=1,2,...,K$. Combining this result and Lemma \ref{lem:approxbias}, we have
$$L({\alpha}^{(k),ora},{\alpha})^2\leq C\Bigg(s\Big(\frac{\log p}{n}\Big)^{1-q/2}+||\theta-\alpha||^2\vee||\beta-\eta||^2\Bigg),$$
for $k=1,2,...,K$. The final bound follows from the triangle inequality applied to the equation above and the fact that
$$L(\alpha,\theta)\leq C \Big(||\alpha-\theta||\wedge ||\alpha+\theta||\Big)\leq C||\alpha-\theta||.$$
The same analysis applies for $L({\beta}^{(k),ora},{\eta})^2$. Thus, the result is obtained conditioning on $$(X_{n+1},Y_{n+1}),...,(X_{2n},Y_{2n}).$$ Since we assume $\xi_{\Omega}=o(1)$ with probability at least $1-O(p^{-2})$, the unconditional result also holds.
\end{proof}

\clearpage

\section*{Appendix}
\appendix

\section{Technical Lemmas}

We define the high-signal and low-signal set of $(\theta,\eta)$ by 
\begin{equation*}
H^{\prime}_1=\left\{k: |\theta_k|\geq \frac{1}{2}\delta_1\sqrt{\frac{\log p_1%
}{n}}\right\},\quad H^{\prime}_2=\left\{k: |\eta_k|\geq \frac{1}{2}\delta_2%
\sqrt{\frac{\log p_2}{n}}\right\},
\end{equation*}
and $L_1^{\prime}=\{1,...,p_1\}-H_1^{\prime}$ and $L_2^{\prime}=\{1,...,p_2%
\}-H_2^{\prime}$.

\begin{lemma}
\label{lem:cardinality} We have 
\begin{equation*}
|H^{\prime}_1|\leq (\delta_1/2)^{-q}s_1\Bigg(\frac{\log p_1}{n}\Bigg)%
^{-q/2},\quad |H^{\prime}_2|\leq (\delta_2/2)^{-q}s_2\Bigg(\frac{\log p_2}{n}%
\Bigg)^{-q/2},
\end{equation*}
\begin{equation*}
|H_1|\leq (\delta_1/2)^{-q}s_1\Bigg(\frac{\log p_1}{n}\Bigg)%
^{-q/2}+(\delta_1/2)^{-2}\Bigg(\frac{\log p_1}{n}\Bigg)^{-1}||\theta-%
\alpha||^2,
\end{equation*}
\begin{equation*}
|H_2|\leq (\delta_2/2)^{-q}s_2\Bigg(\frac{\log p_2}{n}\Bigg)%
^{-q/2}+(\delta_2/2)^{-2}\Bigg(\frac{\log p_2}{n}\Bigg)^{-1}||\eta-\beta||^2,
\end{equation*}
\begin{equation*}
|L_1-L_1^{\prime}|=|H^{\prime}_1-H_1|\leq 2(\delta_1/2)^{-q}s_1\Bigg(\frac{%
\log p_1}{n}\Bigg)^{-q/2}+(\delta_1/2)^{-2}\Bigg(\frac{\log p_1}{n}\Bigg)%
^{-1}||\theta-\alpha||^2,
\end{equation*}
\begin{equation*}
|L_2-L_2^{\prime}|=|H^{\prime}_2-H_2|\leq 2(\delta_2/2)^{-q}s_2\Bigg(\frac{%
\log p_2}{n}\Bigg)^{-q/2}+(\delta_2/2)^{-2}\Bigg(\frac{\log p_2}{n}\Bigg)%
^{-1}||\eta-\beta||^2.
\end{equation*}
\end{lemma}

For the transformed data $\{(\tilde{X}_i,\tilde{Y}_i)\}_{i=1}^n$, it has a
latent variable representation 
\begin{equation*}
\tilde{X}_i=\sqrt{\lambda}\alpha Z_i+X_i^{\prime},\quad \tilde{Y}_i=\sqrt{%
\lambda}\beta Z_i+ Y_i^{\prime},
\end{equation*}
where $Z_i,X_i^{\prime},Y_i^{\prime}$ are independent, $Z_i\sim N(0,1)$ and $%
X_i^{\prime}$ and $Y_i^{\prime}$ are Gaussian vectors.

\begin{lemma}
\label{lem:latent} The latent representation above exists in the sense that $%
\mbox{\rm Cov}(X^{\prime})\geq 0$ and $\mbox{\rm Cov}(Y^{\prime})\geq 0$.
Moreover, we have 
\begin{equation*}
||\mbox{\rm Cov}(X^{\prime})||\leq \big(1+o(1)\big)||\hat{\Omega}_1||,\quad%
\mbox{\rm and}\quad ||\mbox{\rm Cov}(Y^{\prime})||\leq \big(1+o(1)\big)||%
\hat{\Omega}_2||.
\end{equation*}
\end{lemma}

\begin{lemma}[\protect\cite{johnstone2001thresholding}]
\label{lem:concen1} Given $Z_1,...,Z_n$ i.i.d. $N(0,1)$. For each $t\in
(0,1/2)$, 
\begin{equation*}
\mathbb{P}\Bigg(\left|\frac{1}{n}\sum_{i=1}^nZ_i^2-1\right|>t\Bigg)\leq 2\exp%
\Big(-\frac{3nt^2}{16}\Big).
\end{equation*}
\end{lemma}

\begin{lemma}
\label{lem:concen2} Let $X_{H_1}^{\prime}$ be an $n\times|H_1|$ matrix with $%
X_{i,H_1}^{\prime}$ being the $i$-th row and $Y_{H_2}^{\prime}$ be an $%
n\times|H_2|$ matrix with $Y_{i,H_2}^{\prime}$ being the $i$-th row. We have
for any $t>0$, 
\begin{equation*}
\mathbb{P}\Bigg(||X_{H_1}^{\prime T}X_{H_1}^{\prime}||> 1.01||{\hat{\Omega}}%
_1||\Big(n+2\big(\sqrt{n|H_1|}+nt\big)+\big(\sqrt{|H_1|}+\sqrt{n}t\big)^2%
\Big)\Bigg)\leq 2e^{-nt^2/2},
\end{equation*}
\begin{equation*}
\mathbb{P}\Bigg(||Y_{H_2}^{\prime T}Y_{H_2}^{\prime}||>1.01||{\hat{\Omega}}%
_2||\Big(n+2\big(\sqrt{n|H_2|}+nt\big)+\big(\sqrt{|H_2|}+\sqrt{n}t\big)^2%
\Big)\Bigg)\leq 2e^{-nt^2/2},
\end{equation*}
\begin{equation*}
\mathbb{P}\Bigg(||X_{H_1}^{\prime T}Y_{H_2}^{\prime}||> 1.03||\hat{\Omega}%
_1||^{1/2}||\hat{\Omega}_2||^{1/2}\Big(\sqrt{|H_1|n}+\sqrt{|H_2|n}+t\sqrt{n}%
\Big)\Bigg)\leq \big(|H_1|\wedge|H_2|\big)e^{-3n/64}+e^{-t^2/2}.
\end{equation*}
\end{lemma}

\begin{lemma}
\label{lem:concen3} We have for any $t>0$, 
\begin{equation*}
\mathbb{P}\Bigg(\left\|\sum_{i=1}^n Z_i X_{i,H_1}^{\prime}\right\|>1.03 ||%
\hat{\Omega}_1||^{1/2}\Big((t+1)\sqrt{n}+\sqrt{|H_1|n}\Big)\Bigg)\leq
e^{-3n/64}+e^{-t^2/2},
\end{equation*}
\begin{equation*}
\mathbb{P}\Bigg(\left\|\sum_{i=1}^n Z_i Y_{i,H_2}^{\prime}\right\|>1.03 ||%
\hat{\Omega}_2||^{1/2}\Big((t+1)\sqrt{n}+\sqrt{|H_2|n}\Big)\Bigg)\leq
e^{-3n/64}+e^{-t^2/2},
\end{equation*}
\begin{equation*}
\mathbb{P}\Bigg(\left\|\sum_{i=1}^n
Z_iX_{i,L_1}^{\prime}\right\|_{\infty}>1.03 ||\hat{\Omega}_1||^{1/2}(t+2)%
\sqrt{n}\Bigg)\leq |L_1|\Big(e^{-3n/64}+e^{-t^2/2}\Big),
\end{equation*}
\begin{equation*}
\mathbb{P}\Bigg(\left\|\sum_{i=1}^n
Z_iY_{i,L_2}^{\prime}\right\|_{\infty}>1.03 ||\hat{\Omega}_2||^{1/2}(t+2)%
\sqrt{n}\Bigg)\leq |L_2|\Big(e^{-3n/64}+e^{-t^2/2}\Big).
\end{equation*}
\end{lemma}

\begin{lemma}
\label{lem:concen4} We have 
\begin{eqnarray*}
||\hat{A}_{H_{1}H_{2}}-A_{H_{1}H_{2}}|| &\leq &C\Bigg(s^{1/2}\Bigg(\frac{%
\log p}{n}\Bigg)^{1/2-q/4}+||\alpha -\theta ||\vee ||\beta -\eta ||\Bigg), \\
\max_{i=1,2}|\hat{l}_{i}-{l}_{i}| &\leq &C\Bigg(s^{1/2}\Bigg(\frac{\log p}{n}%
\Bigg)^{1/2-q/4}+||\alpha -\theta ||\vee ||\beta -\eta ||\Bigg), \\
l_1 &\geq& C^{-1}, \\
l_2 &=& 0
\end{eqnarray*}%
for some constant $C>0$ with probability at least $1-O(p^{-2})$. The
quantities $\hat{l}_{i}$ and $l_{i}$ are the $i$-th singular values of $\hat{%
A}_{H_{1}H_{2}}$ and $A_{H_{1}H_{2}}$ respectively.
\end{lemma}

\section{Analysis of the Initializer}

We show that the Algorithm \ref{alg:: CTSCCA} actually returns a consistent
estimator of leading pair of singular vectors $\left( \alpha ^{(0)},\beta
^{(0)}\right) $, which serves as a good candidate for the power method in
Algorithm \ref{alg:: ITSCCA}. To be specific, with high probability, the
initialization procedure correctly kills all low signal coordinates to
zeros, i.e. $\alpha _{L_{1}}^{(0)}=0$ and $\beta _{L_{2}}^{(0)}=0$, by the
thresholding step. On the other hand, although Algorithm \ref{alg:: CTSCCA}
cannot always correctly pick all strong signal coordinates, it does pick
those much stronger ones such that $\left( \alpha ^{(0)},\beta ^{(0)}\right) 
$ is still consistent up to a sign. The properties are summarized below.

\begin{lemma}
\label{lemma:: Initial} With probability $1-Cp^{-2}$, we have that,

\begin{enumerate}
\item $B_{i}\subset H_{i}$, $i=1,2$;

\item $\left\vert \hat{l}_{i}-\hat{l}_{i}^{B}\right\vert \rightarrow 0$ for $%
i=1,2.$ where $\hat{l}_{i}$ and $\hat{l}_{i}^{B}$ are the $i$th singular
value of $\hat{A}_{H_{1}H_{2}}$ and $\hat{A}_{B_{1}B_{2}}$;

\item $\alpha ^{(0)}$ and $\beta ^{(0)}$ are consistent, i.e., $L\left(
\alpha ^{(0)},\hat{\alpha}^{ora}\right) \rightarrow 0$ and $L\left( \beta
^{(0)},\hat{\beta}^{ora}\right) \rightarrow 0$.
\end{enumerate}
\end{lemma}

The procedure is to select those strong signal coordinates of $\alpha $ and $%
\beta $ and is similar to the \textquotedblleft diagonal
thresholding\textquotedblright\ sparse PCA method proposed by \cite%
{johnstone09}. However, unlike the PCA setting, we cannot get the
information of each coordinate through its corresponding diagonal entry of
the sample covariance matrix. Instead, we measure the strength of
coordinates in terms of the maximum entry among its corresponding row or
column of the sample covariance matrix and still can capture all coordinates
of $\alpha $ and $\beta $ above the level of $\left( \frac{\log p}{n}\right)
^{1/4}$. The sparsity assumption in Equation (\ref{SP condition}) is needed
to guarantee the consistency of the initial estimator $\left( \alpha
^{(0)},\beta ^{(0)}\right)$.

The proof is developed in the following. In the first part, we prove the
three results in Lemma \ref{lemma:: Initial} along with stating some useful
propositions. We then prove those propositions in the second part.

\subsection{Proof of Lemma \protect\ref{lemma:: Initial}}

\begin{proof}[Proof of Result 1] 
We start with showing the first result $B_{i}\subset H_{i}$. In fact we show
that the index set $B_{i}$screens out all weak signal coordinates as well as
captures all coordinates with much stronger signal coordinates $\sqrt{\frac{%
\log p}{n}}s^{\frac{1}{2-q}}$, compared with the thresholding $\sqrt{\frac{%
\log p}{n}}$of $H_{i}$. By the sparsity assumption (\ref{SP condition}),
clearly $B_{i}$also captures all coordinates with much stronger signal $%
\left( \frac{\log p}{n}\right) ^{1/4}$. To be specific, we show that with
probability $1-Cp^{-2}$, we have%
\begin{equation}
B_{i}^{-}\subset B_{i}\subset H_{i},  \label{Range of B}
\end{equation}%
where the index set of those stronger signal coordinates are defined as
follows,%
\begin{equation*}
B_{1}^{-}=\left\{ i,\left\vert \alpha _{i}\right\vert >\phi \sqrt{\frac{\log
p}{n}}s_{2}^{\frac{1}{2-q}}\right\} ,B_{2}^{-}=\left\{ i,\left\vert \beta
_{i}\right\vert >\phi \sqrt{\frac{\log p}{n}}s_{1}^{\frac{1}{2-q}}\right\} .
\end{equation*}%
The constant $\phi $ is determined in the analysis. $\phi
=2\max_{i,j}t_{ij}C_{\lambda }^{-1}\left( \frac{\sqrt{W}}{0.6}\left(
2W^{1/2}\right) ^{\frac{-q}{2-q}}\right) ^{1/2}$.

Recall that $\hat{A}=\frac{1}{n}\sum_{i}\tilde{X}_{i}\tilde{Y}_{i}^{T}$ with
conditional mean $\mathbb{E}\hat{A}=A=\lambda \alpha \beta ^{T}=\left(
a_{ij}\right) _{p_{1}\times p_{2}}$. The result (\ref{Range of B}) can be
shown in two steps. During the first step we show that with probability $%
1-Cp^{-2}$, the index set $B_{i}$ satisfies $B_{i}^{--}\subset B_{i}\subset
B_{i}^{++}$, where for simplicity we pick $\kappa _{-}=0.09,\kappa _{+}=2$
and 
\begin{equation*}
B_{1}^{++}=\left\{ i:\max_{j}\frac{|{a}_{ij}|}{t_{ij}}>\kappa _{-}\sqrt{%
\frac{1}{n}\log p}\right\} ,B_{1}^{--}=\left\{ i:\max_{j}\frac{|{a}_{ij}|}{%
t_{ij}}>\kappa _{+}\sqrt{\frac{1}{n}\log p}\right\} ,
\end{equation*}%
\begin{equation*}
B_{2}^{++}=\left\{ j:\max_{i}\frac{|{a}_{ij}|}{t_{ij}}>\kappa _{-}\sqrt{%
\frac{1}{n}\log p}\right\} ,B_{2}^{--}=\left\{ j:\max_{i}\frac{|{a}_{ij}|}{%
t_{ij}}>\kappa _{+}\sqrt{\frac{1}{n}\log p}\right\} .
\end{equation*}%
For the second step, we show that $B_{i}^{++}\subset H_{i}$ and $%
B_{i}^{--}\subset B_{i}^{-}$ with probability $1-Cp^{-2}$. We present these
two results in the following two propositions.

\begin{proposition}
\label{lemma:: Concentration B} With probability $1-Cp^{-2}$, we have $%
B_{i}^{--}\subset B_{i}\subset B_{i}^{++}$ for $i=1,2.$
\end{proposition}

\begin{proposition}
\label{lemma:: Concentration B part2} With probability $1-Cp^{-2}$, we have $%
B_{i}^{++}\subset H_{i}$ and $B_{i}^{--}\subset B_{i}^{-}$ for $i=1,2.$
\end{proposition}

Thus, the proof is complete. 
\end{proof}

Before proving Result 2 and Result 3, we need the following proposition.

\begin{proposition}
\label{lemma:: Consist (B_)^c} Define 
\begin{equation}
e_{B}^{2}=\max \left\{ \sum_{i\in \left( B_{1}^{-}\right) ^{c}}\alpha
_{i}^{2},\sum_{i\in \left( B_{2}^{-}\right) ^{c}}\beta _{i}^{2}\right\} %
\mbox{\rm .}  \label{e_B}
\end{equation}%
Then we have for some constants $C_{1},C_{2}>0$,%
\begin{equation}
e_{B}^{2}\leq C_{1}\left( \frac{\log p}{n}\right) ^{1-q/2}\left(
s_{1}^{2}\vee s_{2}^{2}\right) +C_{2}\xi _{\Omega }^{2}.  \label{e_B Bound}
\end{equation}%
Moreover, under our assumptions (\ref{SP condition}) and (\ref{Consistency
condition}), we obtain $e_{B}=o(1)$ with probability $1-Cp^{-2}$.
\end{proposition}

Now we restrict our attention on the event on which the result (\ref{Range
of B}) holds, which is valid with high probability $1-Cp^{-2}$. Define index
set $D_{i}=H_{i}\backslash B_{i}$ and%
\begin{equation*}
\hat{A}_{H_{1}H_{2}}^{B_{1}B_{2}}=\left( 
\begin{array}{cc}
\hat{A}_{B_{1}B_{2}} & 0 \\ 
0 & 0%
\end{array}%
\right) _{\left\vert H_{1}\right\vert \times \left\vert H_{2}\right\vert }.
\end{equation*}

\begin{proof}[Proof of Result 2] We show the second result $\left\vert \hat{l}%
_{i}-\hat{l}_{i}^{B}\right\vert \rightarrow 0$ for $i=1,2.$ Note that $%
\hat{l}_{i}^{B}$ is the $i$th singular value of $\hat{A}_{B_{1}B_{2}}$ and
hence is also the $i$th singular value of $\hat{A}_{H_{1}H_{2}}^{B_{1}B_{2}}$%
. Applying Weyl's theorem, we obtain that%
\begin{equation}
\left\vert \hat{l}_{i}-\hat{l}_{i}^{B}\right\vert \leq \left\Vert \hat{%
A}_{H_{1}H_{2}}-\hat{A}_{H_{1}H_{2}}^{B_{1}B_{2}}\right\Vert \leq \left\Vert
\hat{A}_{D_{1}B_{2}}\right\Vert +\left\Vert \hat{A}_{B_{1}D_{2}}\right\Vert
+\left\Vert \hat{A}_{D_{1}D_{2}}\right\Vert .  \label{initial: Weyl}
\end{equation}
To bound $\left\Vert \hat{A}_{D_{1}B_{2}}\right\Vert ,$ we apply the
latent variable representation in Lemma \ref{lem:latent}
and obtain that $\hat{A}_{D_{1}B_{2}}=\sum_{j=1}^{4}G_{j}$, where%
\begin{equation*}
G_{1}=\frac{\lambda Z^{T}Z}{n}\alpha _{D_{1}}\beta _{B_{2}}^{T},\mbox{\rm  }%
G_{2}=\frac{1}{n}X_{D_{1}}^{\prime T}Y_{B_{2}}^{\prime },\mbox{\rm  }G_{3}=%
\frac{\sqrt{\lambda }\alpha _{D_{1}}Z^{T}Y_{B_{2}}^{\prime }}{n},\mbox{\rm  }%
G_{4}=\frac{\sqrt{\lambda }X_{D_{1}}^{\prime T}Z\beta _{B_{2}}^{T}}{n}.
\end{equation*}%
Now we bound $G_{j}$ separately as follows. According to Lemma \ref{lem:concen1},
Proposition \ref{lemma:: Consist (B_)^c} and the fact $\left\Vert \beta
_{B_{2}}\right\Vert \leq \left\Vert \beta \right\Vert \leq \left\Vert \hat{%
\Omega}_{2}\Sigma _{2}\right\Vert \left\Vert \eta \right\Vert =\left(
1+o(1)\right) \left\Vert \eta \right\Vert $, we obtain that with probability
$1-Cp^{-2}$
\begin{eqnarray*}
\left\Vert G_{1}\right\Vert &\leq &\lambda \left( 1+O(\sqrt{\frac{\log p}{n}}%
)\right) \left\Vert \alpha _{D_{1}}\right\Vert \left\Vert \beta
_{B_{2}}\right\Vert \leq \left( 1+o(1)\right) \lambda \left\Vert \eta
\right\Vert \left\Vert \alpha _{\left( B_{1}^{-}\right) ^{c}}\right\Vert \\
&\leq &\left( 1+o(1)\right) \lambda \left\Vert \theta \right\Vert e_{B},
\end{eqnarray*}%
where the second inequality follows from $D_{1}\subset \left(
B_{1}^{-}\right) ^{c}$. The third inequality in Lemma \ref{lem:concen2} implies that with
probability $1-Cp^{-2},$
\begin{equation*}
\left\Vert G_{2}\right\Vert \leq \left( C(\sqrt{\frac{\log p}{n}}+\sqrt{%
\frac{\left\vert H_{1}\right\vert }{n}}+\sqrt{\frac{\left\vert
H_{2}\right\vert }{n}})\right) =o(e_{B}),
\end{equation*}%
where the last inequality is due to Lemma \ref{lem:cardinality} $\left\vert
H_{i}\right\vert =O\left( s_{i}\left( \frac{\log p}{n}\right) ^{-q/2}\right)
$ and Proposition \ref{lemma:: Consist (B_)^c}. Moreover, the first two inequalities in Lemma
\ref{lem:concen3}  further imply that with probability $1-Cp^{-2},$%
\begin{eqnarray*}
\left\Vert G_{3}\right\Vert &\leq &\left( C(\sqrt{\frac{\log p}{n}}+\sqrt{%
\frac{\left\vert H_{2}\right\vert }{n}})\left\Vert \alpha
_{D_{1}}\right\Vert \right) \leq o(\left\Vert \alpha _{\left(
B_{1}^{-}\right) ^{c}}\right\Vert )=o(e_{B}), \\
\left\Vert G_{4}\right\Vert &\leq &\left( C(\sqrt{\frac{\log p}{n}}+\sqrt{%
\frac{\left\vert H_{1}\right\vert }{n}})\left\Vert \beta _{B_{2}}\right\Vert
\right) =o(e_{B}).
\end{eqnarray*}%
Combining the above four results, we obtain that $\left\Vert \hat{A}%
_{D_{1}B_{2}}\right\Vert \leq \left( 1+o(1)\right) \lambda \left\Vert \theta
\right\Vert e_{B}$ with probability $1-Cp^{-2}$. Similarly we can obtain
that $\left\Vert \hat{A}_{B_{1}D_{2}}\right\Vert \leq \left( 1+o(1)\right)
\lambda \left\Vert \theta \right\Vert e_{B}$ with probability $1-Cp^{-2}$.
To bound $\left\Vert \hat{A}_{D_{1}D_{2}}\right\Vert $, similarly we can
write $\hat{A}_{D_{1}B_{2}}=\sum_{j=1}^{4}F_{j}$, where%
\begin{equation*}
F_{1}=\frac{\lambda Z^{T}Z}{n}\alpha _{D_{1}}\beta _{D_{2}}^{T},\mbox{\rm  }%
F_{2}=\frac{1}{n}X_{D_{1}}^{\prime T}Y_{D_{2}}^{\prime },\mbox{\rm  }F_{3}=%
\frac{\sqrt{\lambda }\alpha _{D_{1}}Z^{T}Y_{D_{2}}^{\prime }}{n},\mbox{\rm  }%
F_{4}=\frac{\sqrt{\lambda }X_{D_{1}}^{\prime T}Z\beta _{D_{2}}^{T}}{n}.
\end{equation*}%
Note that $D_{i}\subset \left( B_{i}^{-}\right) ^{c}$ for $i=1,2$. Lemma \ref{lem:concen2}, Lemma \ref{lem:concen3} and Proposition \ref{lemma::
Consist (B_)^c} imply that with probability $1-Cp^{-2}$,

\begin{eqnarray*}
\left\Vert F_{1}\right\Vert &\leq &\lambda \left( 1+O(\sqrt{\frac{\log p}{n}}%
)\right) \left\Vert \alpha _{D_{1}}\right\Vert \left\Vert \beta
_{D_{2}}\right\Vert \leq C\left\Vert \alpha _{\left( B_{1}^{-}\right)
^{c}}\right\Vert \left\Vert \beta _{\left( B_{2}^{-}\right) ^{c}}\right\Vert
=o(e_{B}), \\
\left\Vert F_{2}\right\Vert &\leq &\left( C(\sqrt{\frac{\log p}{n}}+\sqrt{%
\frac{\left\vert H_{1}\right\vert }{n}}+\sqrt{\frac{\left\vert
H_{2}\right\vert }{n}})\right) =o(e_{B}), \\
\left\Vert F_{3}\right\Vert &\leq &o(\left\Vert \alpha _{\left(
B_{1}^{-}\right) ^{c}}\right\Vert )=o(e_{B}),\left\Vert F_{4}\right\Vert
\leq o(\left\Vert \beta _{\left( B_{2}^{-}\right) ^{c}}\right\Vert
)=o(e_{B}).
\end{eqnarray*}%
Hence we obtain $\left\Vert \hat{A}_{D_{1}D_{2}}\right\Vert =o(e_{B}).$
These upper bounds for $\left\Vert \hat{A}_{D_{1}D_{2}}\right\Vert ,$ $%
\left\Vert \hat{A}_{D_{1}D_{2}}\right\Vert $ and $\left\Vert \hat{A}%
_{D_{1}D_{2}}\right\Vert $, together with Equation (\ref{initial: Weyl})
imply that with probability $1-Cp^{-2}$,%
\begin{equation}
\left\vert \hat{l}_{i}-\hat{l}_{i}^{B}\right\vert \leq \left\Vert \hat{A}%
_{H_{1}H_{2}}-\hat{A}_{H_{1}H_{2}}^{B_{1}B_{2}}\right\Vert \leq \left(
1+o(1)\right) \left( \left\Vert \theta \right\Vert +\left\Vert \eta
\right\Vert \right) \lambda e_{B}=o(1),  \label{initial: result2}
\end{equation}%
where the last inequality follows from Proposition \ref{lemma:: Consist
(B_)^c} and Assumption B ($\max \left\{ \left\Vert \theta \right\Vert
,\left\Vert \eta \right\Vert \right\} \leq w^{-1/2}$). 
\end{proof}

\begin{proof}[Proof of Result 3]
We show that last result $L\left( \alpha ^{(0)},\hat{\alpha}^{ora}\right)
\rightarrow 0$ and $L\left( \beta ^{(0)},\hat{\beta}^{ora}\right)
\rightarrow 0$. Note $\alpha ^{(0)}$ and $\hat{\alpha}^{ora}\in \mathbb{R}%
^{p_{1}}$ but all entries in the index set $H_{1}^{c}$ are zeros. Hence we
only need to compare them in $\mathbb{R}^{\left\vert H_{1}\right\vert }$.
Similarly we calculate $L\left( \beta ^{(0)},\hat{\beta}^{ora}\right) $ in
space $\mathbb{R}^{\left\vert H_{2}\right\vert }$. Constraint on coordinates
in $H_{1}\times H_{2}$, $\left( \alpha ^{(0)},\beta ^{(0)}\right) $ and $%
\left( \hat{\alpha}^{ora},\hat{\beta}^{ora}\right) $ are leading pair of
singular vectors of $\hat{A}_{H_{1}H_{2}}^{B_{1}B_{2}}$ and $\hat{A}%
_{H_{1}H_{2}}$ respectively. We apply Wedin's theorem (See \cite%
{stewart1990matrix}, Theorem 4.4) to $\hat{A}_{H_{1}H_{2}}$ and $\hat{A}%
_{H_{1}H_{2}}^{B_{1}B_{2}}$ to obtain that%
\begin{eqnarray}
\max \left\{ L\left( \alpha ^{(0)},\hat{\alpha}^{ora}\right) ,L\left( \beta
^{(0)},\hat{\beta}^{ora}\right) \right\} &\leq &\sqrt{2}\max \left\{
\left\Vert \frac{\alpha ^{(0)}\left( \alpha ^{(0)}\right) ^{T}}{\left\Vert
\alpha ^{(0)}\right\Vert ^{2}}-\frac{\hat{\alpha}^{ora}\left( \hat{\alpha}%
^{ora}\right) ^{T}}{\left\Vert \hat{\alpha}^{ora}\right\Vert ^{2}}%
\right\Vert \right. ,  \notag \\
&&\left. \left\Vert \frac{\beta ^{(0)}\left( \beta ^{(0)}\right) ^{T}}{%
\left\Vert \beta ^{(0)}\right\Vert ^{2}}-\frac{\hat{\beta}^{ora}\left( \hat{%
\beta}^{ora}\right) ^{T}}{\left\Vert \hat{\beta}^{ora}\right\Vert ^{2}}%
\right\Vert \right\}  \notag \\
&\leq &\sqrt{2}\frac{\left\Vert \hat{A}_{H_{1}H_{2}}^{B_{1}B_{2}}-\hat{A}%
_{H_{1}H_{2}}\right\Vert }{\delta \left( \hat{A}_{H_{1}H_{2}}^{B_{1}B_{2}},%
\hat{A}_{H_{1}H_{2}}\right) },  \label{initial: result3}
\end{eqnarray}%
where $\delta \left( \hat{A}_{H_{1}H_{2}}^{B_{1}B_{2}},\hat{A}%
_{H_{1}H_{2}}\right) =\hat{l}_{1}^{B}-\hat{l}_{2}$. The result of Lemma \ref%
{lem:concen4} implies that $\hat{l}_{2}=o(1)$, and Result $2$ we just showed
further implies that $\hat{l}_{1}^{B}=(1+o(1))\hat{l}_{1}=(1+o(1))\lambda
\left\Vert \eta \right\Vert \left\Vert \theta \right\Vert $ with probability 
$1-Cp^{-2}$. Therefore we obtain that with probability $1-Cp^{-2}$, the
value $\delta \left( \hat{A}_{H_{1}H_{2}}^{B_{1}B_{2}},\hat{A}%
_{H_{1}H_{2}}\right) =(1+o(1))\lambda \left\Vert \eta \right\Vert \left\Vert
\theta \right\Vert $ is bounded below and above by constants according to
Assumption B. This fact, together with Equations (\ref{initial: result2})
and (\ref{initial: result3}) completes our proof%
\begin{equation*}
\max \left\{ L\left( \alpha ^{(0)},\hat{\alpha}^{ora}\right) ,L\left( \beta
^{(0)},\hat{\beta}^{ora}\right) \right\} =o(1),
\end{equation*}%
with probability $1-Cp^{-2}$. 
\end{proof}

\subsection{Proofs of Propositions}

\begin{proof}[Proof of Proposition \ref{lemma:: Concentration B}]
We first provide concentration inequality for each $\hat{a}_{ij}=\frac{1}{n}%
\sum_{k}\tilde{X}_{k,i}\tilde{Y}_{k,j}$. By the latent variable
representation, we have $\tilde{X}_{1,i}=\sqrt{\lambda }\alpha
_{i}Z_{1}+X_{1,i}^{\prime }$ and $\tilde{Y}_{1,j}=\sqrt{\lambda }\beta
_{j}Z_{1}+Y_{1,j}^{\prime }$ where $Z_{1}\sim N(0,1)$, $X_{1,i}^{\prime
}\sim N\Big(0,\mbox{\rm Var}\left( X_{1,i}^{\prime }\right) \Big)$ and $%
Y_{1,j}^{\prime }\sim N\Big(0,\mbox{\rm Var}\left( Y_{1,j}^{\prime }\right) %
\Big)$ are independent. This representation leads to%
\begin{equation*}
\hat{a}_{ij}=\frac{1}{n}\left( \lambda \alpha _{i}\beta
_{j}\sum_{k}Z_{k}^{2}+\sum_{k}X_{k,i}^{\prime }Y_{k,j}^{\prime }+\sqrt{%
\lambda }\alpha _{i}\sum_{k}Z_{k}Y_{k,j}^{\prime }+\sqrt{\lambda }\beta
_{j}\sum_{k}Z_{k}X_{k,i}^{\prime }\right) .
\end{equation*}
The consistency assumption Equation (\ref{Consistency condition}) in
Assumption B implies that $\mbox{\rm Var}\left( X_{1,i}^{\prime }\right)
\leq \mbox{\rm Var}\left( \tilde{X}_{1,i}\right) =(1+o(1))\hat{\omega}%
_{1,ii}=(1+o(1))\omega _{1,ii}$. Applying Lemma \ref{lem:concen1}, we obtain
that 
\begin{equation}
\mathbb{P}\left( \left\vert \frac{\lambda \alpha _{i}\beta _{j}}{n}%
\sum_{k}Z_{k}^{2}-a_{ij}\right\vert >\lambda \alpha _{i}\beta _{j}t\right)
\leq 2\exp \left( -\frac{3nt^{2}}{16}\right) .  \label{concentration1}
\end{equation}%
Following the line of proof in Lemma \ref{lem:concen3} and Proposition D.2
in \cite{ma2013sparse}, for $n$ and $p$ large enough (hence $n^{-1}\log
p\rightarrow 0$), we have the following concentration inequalities,%
\begin{eqnarray}
\mathbb{P}\left( \left\vert \frac{1}{n}\sum_{k}X_{k,i}^{\prime
}Y_{k,j}^{\prime }\right\vert >\left( 1+o(1)\right) b\sqrt{\hat{\omega}%
_{1,ii}\hat{\omega}_{2,jj}}\sqrt{\frac{\log p}{n}}\right) &\leq
&2p^{-b^{2}/2},  \label{concentration2} \\
\mathbb{P}\left( \left\vert \frac{1}{n}\sqrt{\lambda }\beta
_{j}\sum_{k}Z_{k}X_{k,i}^{\prime }\right\vert >\left( 1+o(1)\right) b\sqrt{%
\lambda }\left\vert \beta _{j}\right\vert \sqrt{\hat{\omega}_{1,ii}}\sqrt{%
\frac{\log p}{n}}\right) &\leq &2p^{-b^{2}/2},  \label{concentration3} \\
\mathbb{P}\left( \left\vert \frac{1}{n}\sqrt{\lambda }\alpha
_{i}\sum_{k}Z_{k}Y_{k,j}^{\prime }\right\vert >\left( 1+o(1)\right) b\sqrt{%
\lambda }\left\vert \alpha _{i}\right\vert \sqrt{\hat{\omega}_{2,jj}}\sqrt{%
\frac{\log p}{n}}\right) &\leq &2p^{-b^{2}/2}\mbox{\rm .}
\label{concentration4}
\end{eqnarray}%
Recall the definition of the adaptive thresholding level $t_{ij},$%
\begin{equation*}
\begin{array}{c}
t_{ij}=\frac{20\sqrt{2}}{9}\left( \sqrt{\lambda _{\max }\left( \hat{\Omega}%
_{1}\right) \hat{\omega}_{2,jj}}+\sqrt{\lambda _{\max }\left( \hat{\Omega}%
_{2}\right) \hat{\omega}_{1,ii}}\right. \\ 
\left. +\sqrt{\hat{\omega}_{1,ii}\hat{\omega}_{2,jj}}+\sqrt{8\lambda _{\max
}\left( \hat{\Omega}_{1}\right) \lambda _{\max }\left( \hat{\Omega}%
_{2}\right) /3}\right) .%
\end{array}%
\end{equation*}%
Applying the union bound to Equations (\ref{concentration1})-(\ref%
{concentration4}), we obtain the concentration inequality for $\hat{a}_{ij}$
as follows%
\begin{equation}
\mathbb{P}\left( \frac{\left\vert \hat{a}_{ij}-a_{ij}\right\vert }{t_{ij}}%
>\left( 1+o(1)\right) \frac{9b_{1}}{20\sqrt{2}}\sqrt{\frac{\log p}{n}}%
\right) \leq 8p^{-b_{1}^{2}/2},  \label{concentration5}
\end{equation}%
where we used the fact $\lambda \leq 1$ and $\left\vert \alpha
_{i}\right\vert \leq (1+o(1))\left\vert \theta _{i}\right\vert \leq
(1+o(1))\lambda _{\max }^{1/2}\left( \hat{\Omega}_{1}\right) $.

We finish our proof by bounding the probability $\mathbb{P}\left(
B_{i}\not\subset B_{i}^{++}\right) $ and $\mathbb{P}\left(
B_{i}^{--}\not\subset B_{i}\right) $ respectively. Let $j_{i}^{\ast }$ be an
integer such that $|a_{ij_{i}^{\ast }}|/t_{ij_{i}^{\ast
}}=\max_{j}|a_{ij}|/t_{ij}$. We apply the union bound to obtain%
\begin{eqnarray*}
\mathbb{P}\left( B_{1}^{--}\not\subset B_{1}\right) &\leq &\sum_{i\in
B_{1}^{--}}\mathbb{P}\left\{ \max_{_{j}}\frac{|{\hat{a}}_{ij}|}{t_{ij}}\leq 
\sqrt{\frac{\log p}{n}}\right\} \\
&\leq &\sum_{i\in B_{1}^{--}}\mathbb{P}\left\{ |\hat{a}_{ij_{i}^{\ast
}}|/t_{ij_{i}^{\ast }}\leq \sqrt{\frac{\log p}{n}}\right\} \\
&\leq &\sum_{i\in B_{1}^{--}}\mathbb{P}\left\{ |a_{ij_{i}^{\ast }}-\hat{a}%
_{ij_{i}^{\ast }}|/t_{ij_{i}^{\ast }}>(\kappa _{+}-1)\sqrt{\frac{\log p}{n}}%
\right\} \\
&\leq &p8p^{-b_{1}^{2}/2}\leq Cp^{-3},
\end{eqnarray*}%
where the last inequality follows from Equation (\ref{concentration5}) with $%
b_{1}=2\sqrt{2}$ and $\kappa _{+}=2.$ Similarly, we apply the union bound to
obtain%
\begin{eqnarray*}
\mathbb{P}\left( B_{1}\not\subset B_{1}^{++}\right) &\leq &\sum_{i\in \left(
B_{1}^{++}\right) ^{c}}\mathbb{P}\left\{ \max_{_{j}}\frac{|{\hat{a}}_{ij}|}{%
t_{ij}}\newline
>\sqrt{\frac{\log p}{n}}\right\} \\
&\leq &\sum_{i\in \left( B_{1}^{++}\right) ^{c}}\sum_{j=1}^{p_{2}}\mathbb{P}%
\left\{ |{\hat{a}}_{ij}|/t_{ij}>\sqrt{\frac{\log p}{n}}\right\} \\
&\leq &\sum_{i\in \left( B_{1}^{++}\right) ^{c}}\sum_{j=1}^{p_{2}}\mathbb{P}%
\left\{ |{\hat{a}}_{ij}-{a}_{ij}|/t_{ij}>(1-\kappa _{-})\sqrt{\frac{\log p}{n%
}}\right\} \\
&\leq &p^{2}8p^{-b_{1}^{2}/2}\leq Cp^{-2},
\end{eqnarray*}%
where the last inequality follows from Equation (\ref{concentration5}) with $%
b_{1}=2\sqrt{2}$ and $\kappa _{-}=0.09$. We can obtain the bounds for $%
\mathbb{P}\left( B_{2}\not\subset B_{2}^{++}\right) $ and $\mathbb{P}\left(
B_{2}^{--}\not\subset B_{2}\right) $, which finish our proof. 
\end{proof}

\begin{proof}[Proof of Proposition \ref{lemma:: Concentration B part2}]
Recall $a_{ij}=\lambda \alpha _{i}\beta _{j}$. To show $B_{1}^{++}\subset
H_{1}$, we only need to show that $\frac{\kappa _{-}\min_{i,j}t_{ij}}{%
\lambda \max_{j}\left\vert \beta _{j}\right\vert }>\delta _{1},$ noting $%
p_{1}\leq p$. To see this, the key part is to bound $\max_{j}\left\vert
\beta _{j}\right\vert $ from above. In fact, Assumption B implies $%
\max_{j}\left\vert \beta _{j}\right\vert \leq \left\Vert \beta \right\Vert
\leq \left\Vert \hat{\Omega}_{2}\Sigma _{2}\right\Vert \left\Vert \eta
\right\Vert =\left( 1+o(1)\right) w^{-1/2}$ with probability $1-Cp^{-2}.$
Therefore this upper bound of $\left\Vert \beta \right\Vert $, the
definition of $\delta _{1}$ and $\lambda \leq 1$ imply that, 
\begin{equation*}
\frac{\kappa _{-}\min_{j}t_{ij}}{\lambda \max_{j}\left\vert \beta
_{j}\right\vert }\geq 0.08w^{1/2}\min_{i,j}t_{ij}=\delta _{1},
\end{equation*}%
where the last equation follows from the definition of $\delta _{1}$.
Similarly, we can show that $B_{2}^{++}\subset H_{2}$ with probability $%
1-Cp^{-2}$.

To show $B_{1}^{--}\subset B_{1}^{-}$, we only need to show that $\phi
_{1}s_{2}^{\frac{1}{2-q}}>\frac{\kappa _{+}\max_{i,j}t_{ij}}{\lambda
\max_{j}\left\vert \beta _{j}\right\vert }$. This time the key part is to
bound $\max_{j}\left\vert \beta _{j}\right\vert $ from below. Note that $%
\left\Vert \beta \right\Vert \geq \left( 1-o(1)\right) \left\Vert \eta
\right\Vert \geq \left( 1-o(1)\right) W^{-1/2}$ follows from Assumption B.
For any positive integer $k$, we denote $I_{k}$ as the index set of the
largest $k$ coordinates of $\eta $ in magnitude. Then we have with
probability $1-Cp^{-2},$ 
\begin{eqnarray*}
0.9W^{-1/2} &\leq &\left\Vert \beta \right\Vert \leq k\max_{j}\left\vert
\beta _{j}\right\vert ^{2}+\sum_{j\in I_{k}^{c}}\left\vert \beta
_{j}\right\vert ^{2} \\
&\leq &k\max_{j}\left\vert \beta _{j}\right\vert ^{2}+\sum_{j\in
I_{k}^{c}}\left\vert \eta _{j}\right\vert ^{2}+\left\Vert \beta -\eta
\right\Vert  \\
&\leq &k\max_{j}\left\vert \beta _{j}\right\vert ^{2}+\frac{q}{2-q}%
s_{2}^{2/q}k^{1-2/q}+o(1).
\end{eqnarray*}%
Picking $k_{0}=\left\lceil s_{2}^{\frac{2}{2-q}}\left( 2W^{1/2}\right) ^{%
\frac{q}{2-q}}\right\rceil $, the Equation above implies that $%
k_{0}\max_{j}\left\vert \beta _{j}\right\vert ^{2}\geq 0.3W^{-1/2}$.
Consequently, we get a lower bound $\max_{j}\left\vert \beta _{j}\right\vert
\geq \sqrt{C_{l}}s_{2}^{-\frac{1}{2-q}}$, where constant $C_{l}=\frac{0.6}{%
\sqrt{W}}\left( 2W^{1/2}\right) ^{\frac{q}{2-q}}$. We complete our proof by
noting that the lower bound of $\max_{j}\left\vert \beta _{j}\right\vert $,
the definition of $\phi $ and Assumption $\lambda >C_{\lambda }$ imply 
\begin{equation*}
\frac{\kappa _{+}\max_{i,j}t_{ij}}{\lambda \max_{j}\left\vert \beta
_{j}\right\vert }\leq 2\max_{i,j}t_{ij}C_{\lambda }^{-1}C_{l}^{-1/2}s_{2}^{%
\frac{1}{2-q}}=\phi s_{2}^{\frac{1}{2-q}},
\end{equation*}%
where the last equation follows from the definition of constant $\phi $.
Similarly, we can show $B_{2}^{--}\subset B_{2}^{-}$ with probability $%
1-Cp^{-2}.$ 
\end{proof}

\begin{proof}[Proof of Proposition \ref{lemma:: Consist (B_)^c}]
Note we only assume $\theta $ and $\eta $ are in the weak $l_{q}$ ball.
Define the relatively weak signal coordinates of $\theta $ and $\eta $ as 
\begin{equation*}
B_{1}^{\prime -}=\left\{ i,\left\vert \theta _{i}\right\vert >\frac{\phi }{2}%
\sqrt{\frac{\log p}{n}}s_{1}^{\frac{1}{2-q}}\right\} ,B_{2}^{\prime
-}=\left\{ i,\left\vert \eta _{i}\right\vert >\frac{\phi }{2}\sqrt{\frac{%
\log p}{n}}s_{2}^{\frac{1}{2-q}}\right\} .
\end{equation*}%
We need the bound of cardinality of $\left( B_{i}^{-}\right) ^{c}-\left(
B_{i}^{\prime -}\right) ^{c}$. Following the lines of the proof of Lemma \ref%
{lem:cardinality}, we have that 
\begin{equation}
\left\vert \left( B_{i}^{-}\right) ^{c}-\left( B_{i}^{\prime -}\right)
^{c}\right\vert \leq C\left[ s_{i}^{\left( 1-\frac{q}{2-q}\right) }\left( 
\frac{\log p}{n}\right) ^{-q/2}+\xi _{\Omega }^{2}\left( \frac{\log p}{n}%
\right) ^{-1}s_{i}^{-\frac{2}{2-q}}\right] \mbox{\rm ,}
\label{bound card B-}
\end{equation}%
where we use the fact $\left\Vert \theta -\alpha \right\Vert \leq \left\Vert 
\hat{\Omega}_{1}-\Omega _{1}\right\Vert \left\Vert \Sigma _{1}\alpha
\right\Vert \leq C\xi _{\Omega }$ by the Assumption B. Now we bound $%
\left\Vert \alpha _{\left( B_{1}^{-}\right) ^{c}}\right\Vert ^{2}=\sum_{i\in
\left( B_{1}^{-}\right) ^{c}}\alpha _{i}^{2}$ as follows, 
\begin{eqnarray*}
\left\Vert \alpha _{\left( B_{1}^{-}\right) ^{c}}\right\Vert  &\leq
&\left\Vert \theta _{\left( B_{1}^{-}\right) ^{c}}\right\Vert +C\xi _{\Omega
}\leq \left\Vert \theta _{\left( B_{1}^{-}\right) ^{c}-\left( B_{1}^{\prime
-}\right) ^{c}}\right\Vert +\left\Vert \theta _{\left( B_{1}^{\prime
-}\right) ^{c}}\right\Vert +C\xi _{\Omega } \\
&\leq &\left\Vert \alpha _{\left( B_{1}^{-}\right) ^{c}-\left( B_{1}^{\prime
-}\right) ^{c}}\right\Vert +\left\Vert \theta _{\left( B_{1}^{\prime
-}\right) ^{c}}\right\Vert +2C\xi _{\Omega } \\
&\leq &\left[ C_{1}\left( \frac{\log p}{n}\right) ^{1-q/2}s_{1}^{2}+C_{2}\xi
_{\Omega }^{2}\right] ^{1/2},
\end{eqnarray*}%
since we can bound the first two terms by Equation (\ref{bound card B-}) and
weak $l_{q}$ ball assumption on $\theta $ as follows, 
\begin{eqnarray*}
\left\Vert \alpha _{\left( B_{1}^{-}\right) ^{c}-\left( B_{1}^{\prime
-}\right) ^{c}}\right\Vert ^{2} &\leq &\phi ^{2}\frac{\log p}{n}s_{1}^{\frac{%
2}{2-q}}\left\vert \left( B_{i}^{-}\right) ^{c}-\left( B_{i}^{\prime
-}\right) ^{c}\right\vert \leq C\left( \xi _{\Omega }^{2}+s_{1}^{2}\left( 
\frac{\log p}{n}\right) ^{1-q/2}\right) , \\
\left\Vert \theta _{\left( B_{1}^{\prime -}\right) ^{c}}\right\Vert ^{2}
&\leq &\sum_{i\in \left( B_{1}^{\prime -}\right) ^{c}}\theta _{i}^{2}\leq
\sum_{k}\left( \frac{\phi }{2}\right) ^{2}\frac{\log p}{n}s_{1}^{\frac{2}{2-q%
}}\wedge s_{1}^{2/q}k^{-2/q} \\
&\leq &C\left( \frac{\log p}{n}\right) ^{1-q/2}s_{1}^{2}\mbox{\rm .}
\end{eqnarray*}%
Similarly, we can obtain that $\left\Vert \beta _{\left( B_{2}^{-}\right)
^{c}}\right\Vert ^{2}=\sum_{i\in \left( B_{2}^{-}\right) ^{c}}\beta
_{i}^{2}\leq C_{1}\left( \frac{\log p}{n}\right) ^{1-q/2}s_{2}^{2}+C_{2}\xi
_{\Omega }^{2}$. Therefore we finished the proof of Equation (\ref{e_B Bound}%
) and the lemma. 
\end{proof}

\section{Proof of Lemma \protect\ref{lem:approxbias}}

In this section, we are going to show that the first pair of singular
vectors of $\hat{A}^{ora}$ is close to the first pair of singular vectors of 
$A=\lambda\alpha\beta^T$. We introduce an intermediate step by involving
another matrix 
\begin{equation*}
A^{ora}=%
\begin{pmatrix}
\lambda\alpha_{H_1}\beta_{H_2}^T & 0 \\ 
0 & 0%
\end{pmatrix}%
.
\end{equation*}
It is easy to see $(\alpha/||\alpha||,\beta/||\beta||)$ is the first pair of
singular vectors of $A$ and $(\alpha^{ora}/||\alpha^{ora}||,\beta^{ora}/||%
\beta^{ora}||)$ is the first pair of singular vectors of $A^{ora}$, where 
\begin{equation*}
\alpha^{ora}=%
\begin{pmatrix}
\alpha_{H_1} \\ 
0%
\end{pmatrix}%
,\quad \beta^{ora}=%
\begin{pmatrix}
\beta_{H_2} \\ 
0%
\end{pmatrix}%
.
\end{equation*}
Let $(\hat{\alpha}^{ora},\hat{\beta}^{ora})$ be the first pair of singular
vectors of $\hat{A}^{ora}$. Then, we have 
\begin{equation*}
L(\hat{\alpha}^{ora},\alpha) \leq L(\hat{\alpha}^{ora},\alpha^{ora})+L(%
\alpha^{ora},\alpha).
\end{equation*}
We have a similar inequality for $\beta$. We present a deterministic lemma
before proving the results.

\begin{lemma}
\label{lem:1st} We have 
\begin{eqnarray*}
L(\hat{\alpha}^{ora},\alpha ) &\leq &\frac{\sqrt{2}||\hat{A}%
_{H_{1}H_{2}}-A_{H_{1}H_{2}}||}{\hat{l}_{1}}+\frac{8\sqrt{2}||\theta -\alpha
||}{0.9W^{-1/2}} \\
&&+\frac{2\sqrt{2}}{0.9W^{-1/2}}\Bigg(2^{\frac{q+1}{2}}+\sqrt{\frac{2}{2-q}}%
\Bigg)\left( \delta _{1}^{2-q}s_{1}\Bigg(\frac{\log p_{1}}{n}\Bigg)%
^{1-q/2}\right) ^{1/2}, \\
L(\hat{\beta}^{ora},\beta ) &\leq &\frac{\sqrt{2}||\hat{A}%
_{H_{1}H_{2}}-A_{H_{1}H_{2}}||}{\hat{l}_{1}}+\frac{8\sqrt{2}||\eta -\beta ||%
}{0.9W^{-1/2}} \\
&&+\frac{2\sqrt{2}}{0.9W^{-1/2}}\Bigg(2^{\frac{q+1}{2}}+\sqrt{\frac{2}{2-q}}%
\Bigg)\left( \delta _{2}^{2-q}s_{2}\Bigg(\frac{\log p_{2}}{n}\Bigg)%
^{1-q/2}\right) ^{1/2}.
\end{eqnarray*}
\end{lemma}

\begin{proof}
Starting with the above triangle inequality, we need to bound $L(\hat{\alpha}%
^{ora},\alpha ^{ora})$ and $L(\alpha ^{ora},\alpha )$. For the first term,
we use Wedin's sin-theta theorem (Theorem 4.4 in \cite{stewart1990matrix}).
\begin{eqnarray}
L(\hat{\alpha}^{ora},\alpha ^{ora}) &\leq &\sqrt{2}\left\Vert \frac{\hat{%
\alpha}^{ora}\hat{\alpha}^{ora,T}}{||\hat{\alpha}^{ora}||^{2}}-\frac{\alpha
^{ora}\alpha ^{ora,T}}{||\alpha ^{ora}||^{2}}\right\Vert   \notag \\
&\leq &\frac{\sqrt{2}\left\Vert \hat{A}_{H_{1}H_{2}}\frac{\beta _{H_{2}}}{%
||\beta _{H_{2}}||}-\lambda ||\beta _{H_{2}}||\alpha _{H_{1}}\right\Vert }{%
\hat{l}_{1}}  \notag \\
&\leq &\frac{\sqrt{2}||\hat{A}_{H_{1}H_{2}}-A_{H_{1}H_{2}}||}{\hat{l}_{1}},
\label{equ 1 in proof lemma one}
\end{eqnarray}%
where we also applied the fact $l_{2}=0$ in Lemma \ref{lem:concen4}. For the
second term, we have
\begin{eqnarray}
L(\alpha ^{ora},\alpha ) &\leq &\sqrt{2}\left\Vert \frac{\alpha ^{ora}}{%
||\alpha ^{ora}||}-\frac{\alpha }{||\alpha ||}\right\Vert   \notag \\
&\leq &\frac{2\sqrt{2}||\alpha -\alpha ^{ora}||}{||\alpha ||}  \notag \\
&\leq &\frac{2\sqrt{2}||\alpha _{L_{1}}||}{0.9W^{-1/2}},
\label{equ 2 in proof of of lemma one}
\end{eqnarray}%
where the last inequality follows from Assumption B which leads to $||\alpha
||=\left( 1+o(1)\right) ||\theta ||\geq 0.9||\theta ||\geq 0.9W^{-1/2}$.
Notice that
\begin{eqnarray*}
||\alpha _{L_{1}}|| &\leq &||\theta _{L_{1}}||+||\theta -\alpha || \\
&\leq &||\theta _{L_{1}-L_{1}^{\prime }}||+||\theta _{L_{1}^{\prime
}}||+||\theta -\alpha || \\
&\leq &||\alpha _{L_{1}-L_{1}^{\prime }}||+||\theta _{L_{1}^{\prime
}}||+2||\theta -\alpha || \\
&\leq &\Bigg(2^{\frac{1+q}{2}}+\sqrt{\frac{2}{2-q}}\Bigg)\left( \delta
_{1}^{2-q}s_{1}\Bigg(\frac{\log p_{1}}{n}\Bigg)^{1-q/2}\right)
^{1/2}+4||\theta -\alpha ||,
\end{eqnarray*}%
because $||\alpha _{L_{1}-L_{1}^{\prime }}||^{2}\leq \delta _{1}^{2}\frac{%
\log p_{1}}{n}|L_{1}-L_{1}^{\prime }|\leq 2^{1+q}\delta _{1}^{2-q}s_{1}\Bigg(%
\frac{\log p_{1}}{n}\Bigg)^{1-q/2}+4||\theta -\alpha ||^{2}$, and
\begin{eqnarray*}
||\theta _{L_{1}^{\prime }}||^{2} &=&\sum_{k\in L_{1}^{\prime }}\theta
_{k}^{2} \\
&\leq &\sum_{k=1}^{p_{1}}\Bigg((\delta _{1}/2)^{2}\frac{\log p_{1}}{n}\wedge
s_{1}^{2/q}k^{-2/q}\Bigg) \\
&\leq &\int_{0}^{\infty }\Bigg((\delta _{1}/2)^{2}\frac{\log p_{1}}{n}\Bigg)%
\wedge s_{1}^{2/q}x^{-2/q}dx \\
&\leq &(\delta _{1}/2)^{2-q}s_{1}\Bigg(\frac{\log p_{1}}{n}\Bigg)^{1-q/2}+%
\frac{q}{2-q}(\delta _{1}/2)^{2-q}s_{1}\Bigg(\frac{\log p_{1}}{n}\Bigg)%
^{1-q/2} \\
&\leq &\frac{2}{2-q}\delta _{1}^{2-q}s_{1}\Bigg(\frac{\log p_{1}}{n}\Bigg)%
^{1-q/2}.
\end{eqnarray*}%
Combining Equations (\ref{equ 1 in proof lemma one}) and (\ref{equ 2 in
proof of of lemma one}), together with the bounds above, we finish our proof
for $L(\hat{\alpha}^{ora},\alpha ).$ Similar analysis works for $L(\hat{\beta%
}^{ora},\beta )$ and we omit the details. 
\end{proof}

\begin{proof}[Proof of Lemma \ref{lem:approxbias}]
The proof directly follows from the deterministic bound in Lemma \ref{lem:1st} and the probabilistic bound in Lemma \ref{lem:concen4}.
\end{proof}

\section{Proofs of Lemma \protect\ref{lem:iteration} and Lemma \protect\ref%
{lem:eigengap}}

In this section, we are going to give the convergence bound for the
Algorithms \ref{alg:: ITSCCA} and \ref{alg:: CTSCCA} applied on the oracle
matrix $\hat{A}^{ora}$. According to Lemma \ref{lemma:: Initial}, we obtain
that the initializer applied to the oracle matrix $\hat{A}^{ora}$ by
Algorithm \ref{alg:: CTSCCA} is idential to that applied to data matrix $%
\hat{A}$, i.e. $(\alpha ^{(0)},\beta ^{(0)})=(\alpha ^{(0),ora},\beta
^{(0),ora})$ since $B_{i}\subset H_{i}$. As a consequence, the properties in
Lemma \ref{lemma:: Initial} also hold for $(\alpha ^{(0),ora},\beta
^{(0),ora})$.

The following three lemmas are helpful for us to understand the convergence
of oracle sequence $(\alpha ^{(k),ora},\beta ^{(k),ora})$ obtained from
Algorithm \ref{alg:: ITSCCA}. By definition, the initializer satisfies $%
\alpha _{L_{1}}^{(0),ora}=0$ and $\beta _{L_{2}}^{(0),ora}=0$. Then,
Algorithm \ref{alg:: ITSCCA} only involves the sub-matrix $\hat{A}%
_{H_{1}H_{2}}$. We first state the basic one-step analysis for power SVD
method in the following lemma.

\begin{lemma}
\label{lem:powerSVD} Let $A$ be a matrix with first and second eigenvalues $%
(l_1,l_2)$ satisfying $|l_1|>|l_2|$. Let $(u,v)$ be the first pair of
singular vectors and $(\alpha,\beta)$ be any vectors. Define $\bar{\alpha}%
=A\beta$ and $\bar{\beta}=A^T\alpha$. We have 
\begin{equation*}
\frac{1}{2}L(\bar{\alpha},u)^2\leq\frac{\frac{1}{2}L(\beta,v)^2}{1-\frac{1}{2%
}L(\beta,v)^2}\left|\frac{l_2}{l_1}\right|^2,
\end{equation*}
\begin{equation*}
\frac{1}{2}L(\bar{\beta},v)^2\leq\frac{\frac{1}{2}L(\alpha,v)^2}{1-\frac{1}{2%
}L(\alpha,v)^2}\left|\frac{l_2}{l_1}\right|^2.
\end{equation*}
\end{lemma}

\begin{proof}
We omit the proof because it is almost identical to the proof of Theorem 8.2.2 in \cite{golub1996matrix}.
\end{proof}

Applying the above lemma in our case on the matrix $\hat{A}_{H_{1}H_{2}}$
and with unit-vector initializers $(\alpha _{H_{1}}^{(0)},\beta
_{H_{2}}^{(0)})$. For simplicity of notations, we drop the subscript and
write $\hat{A}$ and $(\alpha ^{(0)},\beta ^{(0)})$ in this section. At the $%
k $-th step, we write 
\begin{equation*}
\bar{\alpha}^{(k)}=\hat{A}\beta ^{(k-1)},\quad \bar{\beta}^{(k)}=\hat{A}%
^{T}\alpha ^{(k-1)},
\end{equation*}%
\begin{equation*}
\alpha ^{(k)}=\frac{T(\bar{\alpha}^{(k)},\gamma _{1})}{||T(\bar{\alpha}%
^{(k)},\gamma _{1})||},\quad \beta ^{(k)}=\frac{T(\bar{\beta}^{(k)},\gamma
_{2})}{||T(\bar{\beta}^{(k)},\gamma _{2})||}.
\end{equation*}%
Now we give a one-step analysis for Algorithm \ref{alg:: ITSCCA}.

\begin{lemma}
\label{lem:2ndprep} Let $(\hat{\alpha},\hat{\beta})$ be the first pair of
singular vectors of $\hat{A}$ and $(\hat{l}_1,\hat{l}_2)$ be the first and
second singular values. We assume 
\begin{equation*}
\frac{1}{2}L(\alpha^{(0)},\hat{\alpha})^2\leq\frac{1}{2},\quad \frac{1}{2}%
L(\beta^{(0)},\hat{\beta})^2\leq\frac{1}{2},
\end{equation*}
and 
\begin{equation*}
\left|\frac{\hat{l}_2}{\hat{l}_1}\right|^2+\frac{8\big(\gamma_1^2|H_1|\vee%
\gamma_2^2|H_2|\big)}{|\hat{l}_1|^2}\leq\frac{1}{4}.
\end{equation*}
Then, we have 
\begin{equation*}
\frac{1}{4}L({\alpha}^{(k)},\hat{\alpha})^2\leq L(\beta^{(k-1)},\hat{\beta}%
)^2\left|\frac{\hat{l}_2}{\hat{l}_1}\right|^2+\frac{8\gamma_1^2|H_1|}{|\hat{l%
}_1|^2},
\end{equation*}
\begin{equation*}
\frac{1}{4}L({\beta}^{(k)},\hat{\beta})^2\leq L(\alpha^{(k-1)},\hat{\alpha}%
)^2\left|\frac{\hat{l}_2}{\hat{l}_1}\right|^2+\frac{8\gamma_2^2|H_2|}{|\hat{l%
}_1|^2}.
\end{equation*}
\end{lemma}

\begin{proof}
By triangle inequality, we have 
\begin{equation*}
L(\alpha ^{(k)},\hat{\alpha})\leq L(\alpha ^{(k)},\bar{\alpha}^{(k)})+L(\bar{%
\alpha}^{(k)},\hat{\alpha}).
\end{equation*}%
The second term on the right hand side above is bounded in Lemma \ref%
{lem:powerSVD}. The first term is bounded as 
\begin{eqnarray*}
L(\alpha ^{(k)},\bar{\alpha}^{(k)}) &\leq &\frac{2\sqrt{2}||T(\bar{\alpha}%
^{(k)},\gamma _{1})-\bar{\alpha}^{(k)}||}{||\bar{\alpha}^{(k)}||} \\
&\leq &\frac{2\sqrt{2}\gamma _{1}\sqrt{|H_{1}|}}{||\hat{A}\beta ^{(k-1)}||},
\end{eqnarray*}%
where $||\hat{A}\beta ^{(k-1)}||^{2}\geq |\hat{l}_{1}|^{2}|\hat{\beta}%
^{T}\beta ^{(k-1)}|^{2}=|\hat{l}_{1}|^{2}\Big(1-\frac{1}{2}L(\beta ^{(k-1)},%
\hat{\beta})^{2}\Big)$. Therefore, we have 
\begin{equation*}
\frac{1}{4}L(\alpha ^{(k)},\hat{\alpha})^{2}\leq \frac{\frac{1}{2}L(\beta
^{(k-1)},\hat{\beta})^{2}}{1-\frac{1}{2}L(\beta ^{(k-1)},\hat{\beta})^{2}}%
\left\vert \frac{\hat{l}_{2}}{\hat{l}_{1}}\right\vert ^{2}+\frac{4\gamma
_{1}^{2}|H_{1}|/|\hat{l}_{1}|^{2}}{1-\frac{1}{2}L(\beta ^{(k-1)},\hat{\beta}%
)^{2}}.
\end{equation*}%
In the same way, we have 
\begin{equation*}
\frac{1}{4}L(\beta ^{(k)},\hat{\beta})^{2}\leq \frac{\frac{1}{2}L(\alpha
^{(k-1)},\hat{\alpha})^{2}}{1-\frac{1}{2}L(\alpha ^{(k-1)},\hat{\alpha})^{2}}%
\left\vert \frac{\hat{l}_{2}}{\hat{l}_{1}}\right\vert ^{2}+\frac{4\gamma
_{2}^{2}|H_{2}|/|\hat{l}_{1}|^{2}}{1-\frac{1}{2}L(\alpha ^{(k-1)},\hat{\alpha%
})^{2}}.
\end{equation*}%
Suppose $L(\alpha ^{(k-1)},\hat{\alpha})^{2}\vee L(\beta ^{(k-1)},\hat{\beta}%
)^{2}\leq 1$, then it is easy to see that $L(\alpha ^{(k)},\hat{\alpha}%
)^{2}\vee L(\beta ^{(k)},\hat{\beta})^{2}\leq 1$ under the assumption. Using
mathematical induction, $L(\alpha ^{(k-1)},\hat{\alpha})^{2}\vee L(\beta
^{(k-1)},\hat{\beta})^{2}\leq 1$ is true for each $k$. Therefore, we deduce
the desired result. 
\end{proof}

The above Lemma \ref{lem:2ndprep} implies the following convergence rate of
oracle sequence.

\begin{lemma}
\label{lem:2nd} Suppose the assumptions of Lemma \ref{lem:2ndprep} hold, and
we further assume $|\hat{l}_2|/|\hat{l}_1|< 32^{-1/4}$, and then we have for
all $k\geq 2$, 
\begin{equation*}
L({\alpha}^{(k)},\hat{\alpha})^2\leq \max\Bigg(4\left|\frac{\hat{l}_2}{\hat{l%
}_1}\right|^2\frac{64\gamma_2^2|H_2|}{|\hat{l}_1|^2}+\frac{64\gamma_1^2|H_1|%
}{|\hat{l}_1|^2},\Big(32\left|\frac{\hat{l}_2}{\hat{l}_1}\right|^4\Big)%
^{[k/2]}\Bigg),
\end{equation*}
\begin{equation*}
L({\beta}^{(k)},\hat{\beta})^2\leq \max\Bigg(4\left|\frac{\hat{l}_2}{\hat{l}%
_1}\right|^2\frac{64\gamma_1^2|H_1|}{|\hat{l}_1|^2}+\frac{64\gamma_2^2|H_2|}{%
|\hat{l}_1|^2},\Big(32\left|\frac{\hat{l}_2}{\hat{l}_1}\right|^4\Big)^{[k/2]}%
\Bigg).
\end{equation*}
\end{lemma}

\begin{proof}
We only prove the bound for $L({\alpha}^{(k)},\hat{\alpha})$. Using the
previous lemma, we derive a formula of a two-step analysis 
\begin{equation*}
L({\alpha}^{(k)},\hat{\alpha})^2\leq L(\alpha^{(k-2)},\hat{\alpha}%
)^2\rho+\omega_1,
\end{equation*}
where 
\begin{equation*}
\rho=\Bigg(4\left|\frac{\hat{l}_2}{\hat{l}_1}\right|^2\Bigg)^2,\quad
\omega_1=4\left|\frac{\hat{l}_2}{\hat{l}_1}\right|^2\frac{32\gamma_2^2|H_2|}{%
|\hat{l}_1|^2}+\frac{32\gamma_1^2|H_1|}{|\hat{l}_1|^2}.
\end{equation*}
Therefore, for each $k$, we have 
\begin{equation*}
L({\alpha}^{(k)},\hat{\alpha})^2\leq \max\Big(2\omega_1, L(\alpha^{(k-2)},%
\hat{\alpha})^2(2\rho)\Big).
\end{equation*}
We are going to prove for each $k$, 
\begin{equation*}
L({\alpha}^{(2k)},\hat{\alpha})^2\leq\max\Big(2\omega_1, L(\alpha^{(0)},\hat{%
\alpha})^2(2\rho)^k\Big).
\end{equation*}
It is obvious that this is true for $k=0$. Suppose this is true for $k-1$,
then we have 
\begin{eqnarray*}
L({\alpha}^{(2k)},\hat{\alpha})^2 &\leq& \max\Big(2\omega_1,L(%
\alpha^{(2(k-1))},\hat{\alpha})^2(2\rho)\Big) \\
&\leq& \max\Bigg(2\omega_1, \max\Big(2\omega_1,L(\alpha^{(0)},\hat{\alpha}%
)^2(2\rho)^{k-1}\Big)(2\rho)\Bigg) \\
&\leq& \max\Big(2\omega_1,2\omega_1(2\rho),L(\alpha^{(0)},\hat{\alpha}%
)^2(2\rho)^k\Big) \\
&=& \max\Big(2\omega_1,L(\alpha^{(0)},\hat{\alpha})^2(2\rho)^k\Big),
\end{eqnarray*}
where the last inequality follows from the assumption $|\hat{l}_2/\hat{l}%
_1|\leq 32^{-1/4}$. By mathematical induction, the inequality is true for
each $k$. Similarly, we can show that for each $k$, 
\begin{equation*}
L(\alpha^{(2k+1)},\hat{\alpha})^2\leq \max\Big(2\omega_1, L(\alpha^{(1)},%
\hat{\alpha})^2(2\rho)^{k}\Big).
\end{equation*}
Therefore, 
\begin{equation*}
L(\alpha^{(k)},\hat{\alpha})^2\leq \max\Big(2\omega_1,(2\rho)^{[k/2]}\Big),
\end{equation*}
and the proof is complete by a similar argument for $\beta^{(k)}$. 
\end{proof}

\begin{proof}[Proof of Lemma \ref{lem:iteration}]
It is sufficient to check the conditions of Lemma \ref{lem:2ndprep} and
Lemma \ref{lem:2nd}. The first condition of Lemma \ref{lem:2ndprep} 
\begin{equation*}
\frac{1}{2}L^{2}(\alpha ^{(0)},\hat{\alpha})^{2}\leq \frac{1}{2},\quad \frac{%
1}{2}L(\beta ^{(0)},\hat{\beta})^{2}\leq \frac{1}{2},
\end{equation*}%
is directly by the fact that the initializer is consistent, which is
guaranteed by Lemma \ref{lemma:: Initial}. The second condition of Lemma \ref%
{lem:2ndprep} is deduced from Lemma \ref{lem:concen4} (bounds of $|\hat{l}%
_{1}|$ and $|\hat{l}_{2}|$) and Lemma \ref{lem:cardinality} (bounds of $%
|H_{1}|$ and $|H_{2}|$). Finally, the condition of Lemma \ref{lem:2nd}
follows from Lemma \ref{lem:concen4}. The conclusions of Lemma \ref%
{lem:iteration} are the conclusions of Lemma \ref{lem:2nd} and Lemma \ref%
{lem:2ndprep} respectively. 
\end{proof}

\begin{proof}[Proof of Lemma \ref{lem:eigengap}]
The bounds on $|H_1|$ and $|H_2|$ have been established in Lemma \ref%
{lem:cardinality}. The bounds for $|\hat{l}_2|^2$ and $|\hat{l}_1|^{-2}$ are
from Lemma \ref{lem:concen4}. 
\end{proof}

\section{Proofs of Lemma \protect\ref{lem:same0} and Lemma \protect\ref%
{lem:same}}

We first present a deterministic bound and then prove the results by
probabilistic arguments.

\begin{lemma}
\label{lem:3rddeterm} For any unit vectors $a\in \mathbb{R}^{\left\vert
H_{1}\right\vert }$ and $b\in \mathbb{R}^{\left\vert H_{2}\right\vert }$, we
have 
\begin{eqnarray*}
||\hat{A}_{L_{1}H_{2}}b||_{\infty } &\leq &\frac{1}{n}%
\sum_{i=1}^{n}Z_{i}^{2}||\beta ||\delta _{1}\sqrt{\frac{\log p_{1}}{n}}%
+\delta _{1}\sqrt{\frac{\log p_{1}}{n}}\left\Vert \frac{1}{n}%
\sum_{i=1}^{n}Z_{i}Y_{i,H_{2}}^{\prime T}\right\Vert +||\beta ||\left\Vert 
\frac{1}{n}\sum_{i=1}^{n}Z_{i}X_{i,L_{1}}^{\prime }\right\Vert _{\infty } \\
&&+\frac{||Y_{H_{2}}^{\prime T}Y_{H_{2}}^{\prime }||^{1/2}}{n}\max_{k\in
L_{1}}\sqrt{\mbox{\rm Var}(X_{1,k}^{\prime })}\left\vert \frac{%
\sum_{i=1}^{n}X_{i,k}^{\prime }Y_{i,H_{2}}^{\prime T}b}{\sqrt{\mbox{\rm Var}%
(X_{1,k}^{\prime })b^{T}Y_{H_{2}}^{\prime T}Y_{H_{2}}^{\prime }b}}%
\right\vert , \\
||\hat{A}_{H_{1}L_{2}}^{T}a||_{\infty } &\leq &\frac{1}{n}%
\sum_{i=1}^{n}Z_{i}^{2}||\alpha ||\delta _{2}\sqrt{\frac{\log p_{2}}{n}}%
+\delta _{2}\sqrt{\frac{\log p_{2}}{n}}\left\Vert \frac{1}{n}%
\sum_{i=1}^{n}Z_{i}X_{i,H_{1}}^{\prime T}\right\Vert +||\alpha ||\left\Vert 
\frac{1}{n}\sum_{i=1}^{n}Z_{i}Y_{i,L_{2}}^{\prime }\right\Vert _{\infty } \\
&&+\frac{||X_{H_{1}}^{\prime T}X_{H_{1}}^{\prime }||^{1/2}}{n}\max_{k\in
L_{2}}\sqrt{\mbox{\rm Var}(Y_{1,k}^{\prime })}\left\vert \frac{%
\sum_{i=1}^{n}Y_{i,k}^{\prime }X_{i,H_{1}}^{T}a}{\sqrt{\mbox{\rm Var}%
(Y_{1,k}^{\prime })a^{T}Y_{H_{1}}^{\prime T}Y_{H_{1}}^{\prime }a}}%
\right\vert .
\end{eqnarray*}
\end{lemma}

\begin{proof}
Using the latent representation in Lemma \ref{lem:latent}, we have
$$\tilde{X}_{i,L_1}\tilde{Y}_{i,H_2}^T=\lambda Z_i^2\alpha_{L_1}\beta_{H_2}^T+\sqrt{\lambda}Z_i\alpha_{L_1}Y_{i,H_2}'^T+\sqrt{\lambda}Z_iX_{i,L_1}'\beta_{H_2}^T+X_{i,L_1}'Y_{i,H_2}'^T.$$
Therefore,
\begin{eqnarray*}
||\hat{A}_{L_1H_2}b||_{\infty} &\leq& \lambda \frac{1}{n}\sum_{i=1}^nZ_i^2||\alpha_{L_1}||_{\infty}|\beta_{H_2}^Tb|  +  \sqrt{\lambda}||\alpha_{L_1}||_{\infty}\left|\frac{1}{n}\sum_{i=1}^nZ_i Y_{i,H_2}'^Tb\right| \\
&& + \sqrt{\lambda}\left\|\frac{1}{n}\sum_{i=1}^n Z_iX_{i,L_1}'\right\|_{\infty}|\beta_{H_2}^Tb|  + \left\|\frac{1}{n}\sum_{i=1}^nX_{i,L_1}'Y_{i,H_2}'^Tb\right\|_{\infty},
\end{eqnarray*}
where the first term is bounded by
$$\frac{1}{n}\sum_{i=1}^nZ_i^2||\beta||\delta_1\sqrt{\frac{\log p_1}{n}},$$
the second term is bounded by
$$\delta_1\sqrt{\frac{\log p_1}{n}}\left\|\frac{1}{n}\sum_{i=1}^nZ_i Y_{i,H_2}'^T\right\|,$$
the third term is bounded by
$$||\beta||\left\|\frac{1}{n}\sum_{i=1}^n Z_iX_{i,L_1}'\right\|_{\infty},$$
and the last term is bounded by
\begin{eqnarray*}
 \max_{k\in L_1} \left|\frac{1}{n}\sum_{i=1}^n X_{i,k}'Y_{i,H_2}^Tb\right| &\leq& \max_{k\in L_1}\sqrt{\text{Var}(X_{1,k}')}\left|\frac{1}{n}\sum_{i=1}^n\frac{X_{i,k}'}{\sqrt{\text{Var}(X_{1,k}')}}Y_{i,H_2}'^Tb\right| \\
 &\leq& \frac{||Y_{H_2}'^TY_{H_2}'||^{1/2}}{n}\max_{k\in L_1}\sqrt{\text{Var}(X_{1,k}')}\left|\frac{\sum_{i=1}^n X_{i,k}'Y_{i,H_2}^Tb}{\sqrt{\text{Var}(X_{1,k}')b^TY_{H_2}'^TY_{H_2}'b}}\right|,
\end{eqnarray*}
where $Y_{H_2}'$ is an $n\times |H_2|$ matrix with the $i$-th row $Y_{i,H_2}'^T$. Summing up the bounds, the proof is complete.
\end{proof}

\begin{proof}[Proof of Lemma \ref{lem:same0}]
This is a corollary of Result 1 of Lemma \ref{lemma:: Initial}. By $B_i\subset H_i$ for $i=1,2$, we have $B_i^{ora}=B_i$ for $i=1,2$. Thus, $(\alpha^{(0),ora},\beta^{(0),ora})=(\alpha^{(0)},\beta^{(0)})$.
\end{proof}

\begin{proof}[Proof of Lemma \ref{lem:same}]
We upper bound each term in the conclusion of Lemma \ref{lem:3rddeterm}.
According to the concentration inequalities we have established,
$$\frac{1}{n}\sum_{i=1}^n Z_i^2\leq 2,$$
with probability at least $1-2e^{-3n/16}$ by Lemma \ref{lem:concen1}.
$$\left\|\frac{1}{n}\sum_{i=1}^nZ_i Y_{i,H_2}'^T\right\|\leq 2.06||\hat{\Omega}_2||^{1/2},$$
with probability at least $1-\Big(e^{-3n/64}+e^{-(\sqrt{n}-1)^2/2}\Big)$ by Lemma \ref{lem:concen3}.
$$\left\|\frac{1}{n}\sum_{i=1}^n Z_iX_{i,L_1}'\right\|_{\infty}\leq 2.07||\hat{\Omega}_1||^{1/2}\sqrt{\frac{\log p}{n}},$$
with probability at least $1-\Big(e^{-3n/64+\log p_1}+p^{-2}\Big)$ by Lemma \ref{lem:concen3}.
$$ \frac{||Y_{H_2}'^TY_{H_2}'||^{1/2}}{n}\leq 3.02||\hat{\Omega}_2||^{1/2}n^{-1/2},$$
with probability at least $1-2e^{-n/2}$ by Lemma \ref{lem:concen2}. Using union bound and by Lemma \ref{lem:3rddeterm}, we have
\begin{eqnarray*}
||\hat{A}_{L_1H_2}b^{(k)}||_{\infty} &\leq& \Big(2\delta_1||\beta||+2.06\delta_1||\hat{\Omega}_2||^{1/2}+2.07||\beta||||\hat{\Omega}_1||^{1/2}\Big)\sqrt{\frac{\log p}{n}} \\
&& + 3.02||\hat{\Omega}_2||^{1/2}n^{-1/2}\max_{l\in L_1}\sqrt{\text{Var}(X_{1,l}')}\left|\frac{\sum_{i=1}^n X_{i,l}'Y_{i,H_2}'^Tb^{(k)}}{\sqrt{\text{Var}(X_{1,l}')b^{(k),T}Y_{H_2}'^TY_{H_2}'b^{(k)}}}\right|,
\end{eqnarray*}
with probability at least $1-O(p^{-2})$
for all $k$. Notice by Lemma \ref{lem:latent}
$$\max_{l\in L_1}\sqrt{\text{Var}(X_{1,l}')}\left|\frac{\sum_{i=1}^n X_{i,l}'Y_{i,H_2}'^Tb^{(k)}}{\sqrt{\text{Var}(X_{1,l}')b^{(k),T}Y_{H_2}'^TY_{H_2}'b^{(k)}}}\right|\leq 1.01||\hat{\Omega}_2||^{1/2}\max_{l\in L_1}\left|\frac{\sum_{i=1}^n X_{i,l}'Y_{i,H_2}'^Tb^{(k)}}{\sqrt{\text{Var}(X_{1,l}')b^{(k),T}Y_{H_2}'^TY_{H_2}'b^{(k)}}}\right|.$$
Since $b^{(k)}$ only depends on $\hat{A}_{H_1H_2}$, $Y_{H_2}'$ and $b^{(k)}$ are jointly independent of $X_{L_1}'$. Therefore, conditioning on $Y_{H_2}$ and $b^{(k)}$,
$$\frac{\sum_{i=1}^n X_{i,l}'Y_{i,H_2}'^Tb^{(k)}}{\sqrt{\text{Var}(X_{1,l}')b^{(k),T}Y_{H_2}'^TY_{H_2}'b^{(k)}}}$$
is a standard Gaussian. Therefore, by union bound,
$$\max_{l\in L_1}\left|\frac{\sum_{i=1}^n X_{i,l}'Y_{i,H_2}'^Tb^{(k)}}{\sqrt{\text{Var}(X_{1,l}')b^{(k),T}Y_{H_2}'^TY_{H_2}'b^{(k)}}}\right|\leq \sqrt{6\log p},$$
with probability at least $1-Kp^{-2}$, for all $k=1,2...,K$. Finally we have
$$||\hat{A}_{L_1H_2}b^{(k)}||_{\infty} \leq  \Big(2\delta_1||\beta||+2.06\delta_1||\hat{\Omega}_2||^{1/2}+2.07||\beta||||\hat{\Omega}_1||^{1/2}+7.5||\hat{\Omega}_2||\Big)\sqrt{\frac{\log p}{n}},$$
for all $k=1,2,...,K$, with probability at least $1-O(p^{-2})$.
The same analysis also applies to $||\hat{A}_{H_1L_2}^Ta^{(k)}||_{\infty}$. The result is proved by $||\alpha||\leq 1.01||\hat{\Omega}_1||^{1/2}$ and $||\beta||\leq 1.01||\hat{\Omega}_2||^{1/2}$ and the choice of $\gamma_1$ and $\gamma_2$ in Section \ref{sec:threshconst}.
\end{proof}

\section{Proofs of Technical Lemmas}

\begin{proof}[Proof of Lemma \ref{lem:cardinality}]
By definition,
\begin{eqnarray*}
|H_{1}^{\prime }| &\leq &\left\vert \left\{ k:|\theta _{k}|\geq \frac{1}{2}\delta _{1}%
\sqrt{\frac{\log p_{1}}{n}}\right\} \right\vert  \\
&=&\left\vert \left\{ k:|\theta _{(k)}|\geq \frac{1}{2}\delta _{1}\sqrt{\frac{\log p_{1}%
}{n}}\right\} \right\vert  \\
&\leq &\left\vert \left\{ k:s_{1}k^{-1}\geq \Bigg(\frac{1}{2}\delta _{1}\sqrt{\frac{%
\log p_{1}}{n}}\Bigg)^{q}\right\} \right\vert  \\
&\leq &(\delta _{1}/2)^{-q}s_{1}\Bigg(\frac{\log p_{1}}{n}\Bigg)^{-q/2}.
\end{eqnarray*}%
Notice
\begin{eqnarray*}
H_{1} &\subset &\left\{ k:|\theta _{k}|+|\alpha _{k}-\theta _{k}|\geq \delta
_{1}\sqrt{\frac{\log p_{1}}{n}}\right\}  \\
&\subset &H_{1}^{\prime }\bigcup \left\{ k:|\alpha _{k}-\theta _{k}|\geq
\frac{1}{2}\delta _{1}\sqrt{\frac{\log p_{1}}{n}}\right\} .
\end{eqnarray*}%
Since
\begin{equation*}
\left\vert \left\{ k:|\alpha _{k}-\theta _{k}|\geq \frac{1}{2}\delta _{1}\sqrt{\frac{%
\log p_{1}}{n}}\right\} \right\vert \Bigg(\frac{1}{2}\delta _{1}\sqrt{\frac{\log p_{1}}{%
n}}\Bigg)^{2}\leq ||\alpha -\theta ||^{2},
\end{equation*}%
we have
\begin{eqnarray*}
|H_{1}| &\leq &|H_{1}^{\prime }|+\left\vert \left\{ k:|\alpha _{k}-\theta
_{k}|\geq \frac{1}{2}\delta _{1}\sqrt{\frac{\log p_{1}}{n}}\right\} \right\vert  \\
&\leq &(\delta _{1}/2)^{-q}s_{1}\Bigg(\frac{\log p_{1}}{n}\Bigg)^{-q/2}+(\delta
_{1}/2)^{-2}\Bigg(\frac{\log p_{1}}{n}\Bigg)^{-1}||\theta -\alpha ||^{2}.
\end{eqnarray*}%
We also have
\begin{equation*}
|L_{1}-L_{1}^{\prime }|=|H_{1}^{\prime }-H_{1}|\leq |H_{1}|+|H_{1}^{\prime
}|\leq 2(\delta _{1}/2)^{-q}s_{1}\Bigg(\frac{\log p_{1}}{n}\Bigg)^{-q/2}+(\delta
_{1}/2)^{-2}\Bigg(\frac{\log p_{1}}{n}\Bigg)^{-1}||\theta -\alpha ||^{2}.
\end{equation*}%
Similar arguments apply to $|H_{2}|,|H_{2}^{\prime }|,|H_{2}-H_{2}^{\prime
}|,|L_{2}-L_{2}^{\prime }|$.
\end{proof}

\begin{proof}[Proof of Lemma \ref{lem:latent}]
It is not hard to find the formula
$$\text{Cov}(X')=\hat{\Omega}_1\Sigma_1\hat{\Omega}_1-\lambda\alpha\alpha^T,\quad \text{Cov}(Y')=\hat{\Omega}_2\Sigma_2\hat{\Omega}_2-\lambda\beta\beta^T.$$
Plugging $\alpha=\hat{\Omega}_1\Sigma_1\theta$, we have
$$\text{Cov}(X')=\hat{\Omega}_1\Sigma_1^{1/2}\Big(I-\lambda\Sigma_1^{1/2}\theta\theta^T\Sigma_1^{1/2}\Big)\Sigma_1^{1/2}\hat{\Omega}_1.$$
Since $||\Sigma^{1/2}\theta||=1$ and $\lambda\geq 1$, we have $\text{Cov}(X')\geq 0$. We proceed to prove the spectral bound as follows.
\begin{eqnarray*}
||\text{Cov}(X')|| &\leq& ||\hat{\Omega}_1\Sigma_1\hat{\Omega}_1|| \leq \Big(1+o(1)\Big)||\hat{\Omega}_1||,
\end{eqnarray*}
where the last inequality follows from Equation (\ref{Consistency condition}) in Assumption B. The same results also hold for $\text{Cov}(Y')$.
\end{proof}

\begin{proof}[Proof of Lemma \ref{lem:concen2}]
Let us denote the covariance matrix of each row of the matrix $X_{H_1}'$ by $R$. Then, we have $||X_{H_1}'^TX_{H_1}'||\leq ||R||||U_{H_1}^TU_{H_1}||$, where $U_{H_1}$ is an $n\times |H_1|$ Gaussian random matrix. We bound $||R||$ according to Lemma \ref{lem:latent} by
$$||R||\leq ||\text{Cov}(X')||\leq 1.01||\hat{\Omega}_1||.$$
For $||U_{H_1}^TU_{H_1}||$, we have the bound
\begin{eqnarray*}
&& \mathbb{P}\Bigg(\frac{1}{n}||U_{H_1}^TU_{H_1}||-1>2\Bigg(\sqrt{\frac{|H_1|}{n}}+t\Bigg)+\Bigg(\sqrt{\frac{|H_1|}{n}}+t\Bigg)^2\Bigg) \\
&\leq& \mathbb{P}\Bigg(||\frac{1}{n}U_{H_1}^TU_{H_1}-I||>2\Bigg(\sqrt{\frac{|H_1|}{n}}+t\Bigg)+\Bigg(\sqrt{\frac{|H_1|}{n}}+t\Bigg)^2\Bigg) \\
&\leq& 2e^{-nt^2/2},
\end{eqnarray*}
where the last inequality is from Proposition D.1 in \cite{ma2013sparse}. In the similar way, we obtain the bound for $||Y_{H_2}'^TY_{H_2}'||$. Now we bound $||X_{H_1}'^TY_{H_2}'||$. Denote the covariance of each row of the matrix $Y_{H_2}'$ by $S$. Then we have $||X_{H_1}'^TY_{H_2}'||\leq ||R||^{1/2}||S||^{1/2}||U_{H_1}^TV_{H_2}||$, where $V_{H_2}$ is an $n\times |H_2|$ Gaussian random matrix. We have $||R||^{1/2}||S||^{1/2}\leq 1.01||\hat{\Omega}_1||^{1/2}||\hat{\Omega}_2||^{1/2}$ by Lemma \ref{lem:latent}, and
\begin{eqnarray*}
\mathbb{P}\Bigg(||U_{H_1}^TV_{H_2}||> 1.01\Big(\sqrt{|H_1|n}+\sqrt{|H_2|n}+t\sqrt{n}\Big)\Bigg)\leq \big(|H_1|\wedge|H_2|\big)e^{-3n/64}+e^{-t^2/2},
\end{eqnarray*}
from Proposition D.2 in \cite{ma2013sparse}. Thus, the proof is complete.
\end{proof}

\begin{proof}[Proof of Lemma \ref{lem:concen3}]
Define $Z$ be the vector $(Z_1,...,Z_n)^T$. We keep the notations in the proof of the above lemma.  Then, we have
$$\left\|\sum_{i=1}^n Z_i X_{i,H_1}'\right\|=||Z^TX_{H_1}'||\leq 1.01||\hat{\Omega}_1||^{1/2}||Z^TU_{H_1}||,$$
where $||Z^TU_{H_1}||$ is upper bounded by
$$\mathbb{P}\Bigg(||Z^TU_{H_1}||>1.01\Big((t+1)\sqrt{n}+\sqrt{|H_1|n}\Big)\Bigg)\leq e^{-3n/64}+e^{-t^2/2},$$
by Proposition D.2 in \cite{ma2013sparse}. The similar analysis also applies to $\left\|\sum_{i=1}^n Z_i Y_{i,H_2}'\right\|$. For the third inequality, we have
\begin{eqnarray*}
&& \mathbb{P}\Bigg(\left\|\sum_{i=1}^n Z_iX_{i,L_1}'\right\|_{\infty}> 1.03||\hat{\Omega}_1||^{1/2}(t+2)\sqrt{n}\Bigg) \\
&\leq& \sum_{k\in L_1}\mathbb{P}\Bigg(\left|\sum_{i=1}^n Z_iX_{ik}'\right|> 1.03||\hat{\Omega}_1||^{1/2}(t+2)\sqrt{n}\Bigg) \\
&\leq& \sum_{k\in L_1}\mathbb{P}\Bigg(\left|\sum_{i=1}^n Z_i\frac{X_{ik}'}{\sqrt{\text{Var}(X_{ik}')}}\right|>1.01(t+2)\sqrt{n}\Bigg) \\
&\leq& |L_1|\Big(e^{-3n/64}+e^{-t^2/2}\Big),
\end{eqnarray*}
where we have used Proposition D.2 in \cite{ma2013sparse} again. Similarly, we obtain the last inequality. The proof is complete.
\end{proof}

\begin{proof}[Proof of Lemma \ref{lem:concen4}]
Using the latent representation in Lemma \ref{lem:latent}, we have 
\begin{eqnarray*}
\hat{A}_{H_{1}H_{2}} &=&A_{H_{1}H_{2}}+A_{H_{1}H_{2}}\Bigg(\frac{1}{n}%
\sum_{i=1}^{n}Z_{i}^{2}-1\Bigg)+\sqrt{\lambda }\alpha _{H_{1}}\Bigg(\frac{1}{%
n}\sum_{i=1}^{n}Z_{i}Y_{i,H_{2}}^{\prime }{}^{T}\Bigg) \\
&&+\sqrt{\lambda }\Bigg(\frac{1}{n}\sum_{i=1}^{n}Z_{i}X_{i,H_{1}}^{\prime
}{}^{T}\Bigg)\beta _{H_{2}}^{T}+\frac{1}{n}\sum_{i=1}^{n}X_{i,H_{1}}^{\prime
}Y_{i,H_{2}}^{\prime }{}^{T}.
\end{eqnarray*}
Therefore 
\begin{eqnarray}
||\hat{A}_{H_{1}H_{2}}-A_{H_{1}H_{2}}|| &\leq &||A_{H_{1}H_{2}}||\left\vert 
\frac{1}{n}\sum_{i=1}^{n}Z_{i}^{2}-1\right\vert +||\alpha ||\left\Vert \frac{%
1}{n}\sum_{i=1}^{n}Z_{i}Y_{i,H_{2}}^{\prime }{}^{T}\right\Vert   \notag \\
&&+||\beta ||\left\Vert \frac{1}{n}\sum_{i=1}^{n}Z_{i}X_{i,H_{1}}^{\prime
}{}^{T}\right\Vert +\left\Vert \frac{1}{n}\sum_{i=1}^{n}X_{i,H_{1}}^{\prime
}Y_{i,H_{2}}^{\prime }{}^{T}\right\Vert .  \label{equ 1 in proof lemmaA6}
\end{eqnarray}
Now we control the upper bounds of four terms above. By picking $t=s^{1/2}%
\Bigg(\frac{\log p}{n}\Bigg)^{1/2-q/4}$, Lemma \ref{lem:concen1} implies 
\begin{equation*}
\left\vert \frac{1}{n}\sum_{i=1}^{n}Z_{i}^{2}-1\right\vert \leq s^{1/2}\Bigg(%
\frac{\log p}{n}\Bigg)^{1/2-q/4},
\end{equation*}
with probability at least $1-O(p^{-2})$. Moreover, Lemma \ref{lem:concen3}
implies that 
\begin{eqnarray*}
\left\Vert \frac{1}{n}\sum_{i=1}^{n}Z_{i}Y_{i,H_{2}}^{\prime
}{}^{T}\right\Vert  &\leq &C||\hat{\Omega}_{2}||^{1/2}\left( s^{1/2}\Big(%
\frac{\log p}{n}\Big)^{1/2-q/4}+\frac{\sqrt{|H_{2}|}}{\sqrt{n}}\right)  \\
&\leq &C||\hat{\Omega}_{2}||^{1/2}\left( s^{1/2}\Big(\frac{\log p}{n}\Big)%
^{1/2-q/4}+||\beta -\eta ||\right) 
\end{eqnarray*}
with probability at least $1-O(p^{-2})$ and 
\begin{eqnarray*}
\left\Vert \frac{1}{n}\sum_{i=1}^{n}Z_{i}X_{i,H_{1}}^{\prime
}{}^{T}\right\Vert  &\leq &C||\hat{\Omega}_{1}||^{1/2}\left( s^{1/2}\Big(%
\frac{\log p}{n}\Big)^{1/2-q/4}+\frac{\sqrt{|H_{1}|}}{\sqrt{n}}\right)  \\
&\leq &C||\hat{\Omega}_{1}||^{1/2}\left( s^{1/2}\Big(\frac{\log p}{n}\Big)%
^{1/2-q/4}+||\alpha -\theta ||\right) 
\end{eqnarray*}
with probability at least $1-O(p^{-2})$, where we also used Lemma \ref%
{lem:cardinality} to control $|H_{1}|$ and $|H_{2}|$. Similarly, Lemma \ref%
{lem:cardinality} and Lemma \ref{lem:concen2} imply 
\begin{equation*}
\left\Vert \frac{1}{n}\sum_{i=1}^{n}X_{i,H_{1}}^{\prime }Y_{i,H_{2}}^{\prime
}{}^{T}\right\Vert \leq C||\hat{\Omega}_{1}||^{1/2}||\hat{\Omega}%
_{2}||^{1/2}\left( s^{1/2}\Big(\frac{\log p}{n}\Big)^{1/2-q/4}+||\alpha
-\theta ||+||\beta -\eta ||\right) ,
\end{equation*}
with probability at least $1-O(p^{-2})$. Besides, Assumption B guarantees
that $||A_{H_{1}H_{2}}||$, $||\hat{\Omega}_{1}||^{1/2}$, $||\hat{\Omega}%
_{2}||^{1/2}$, $||\alpha ||$ and $||\beta ||$ are bounded above by some
constant. Equation (\ref{equ 1 in proof lemmaA6}), together with above
bounds, completes our proof for $||\hat{A}_{H_{1}H_{2}}-A_{H_{1}H_{2}}||$.
The bound for $\max_{i=1,2}|\hat{l}_{i}-l_{i}|$ directly follows from Wely's
theorem,
\begin{equation*}
\max_{i=1,2}|\hat{l}_{i}-l_{i}|\leq ||\hat{A}_{H_{1}H_{2}}-A_{H_{1}H_{2}}||
\mbox{\rm .}
\end{equation*}
For the last result, it's clear that $A_{H_{1}H_{2}}=\lambda \alpha
_{H_{1}}\beta _{H_{2}}$ is of rank one with $l_{2}=0$ and $l_{1}=\lambda
\left\Vert \alpha _{H_{1}}\right\Vert \left\Vert \beta _{H_{2}}\right\Vert
\geq \lambda \left( \left\Vert \alpha \right\Vert -\left\Vert \alpha
_{L_{1}}\right\Vert \right) \left( \left\Vert \beta \right\Vert -\left\Vert
\beta _{L_{2}}\right\Vert \right) $. Assumption B implies that $\left\Vert
\alpha \right\Vert =(1+o(1))\left\Vert \theta \right\Vert \geq 0.9W^{-1/2}$, 
$\left\Vert \beta \right\Vert \geq 0.9W^{-1/2}$ and $\lambda \geq C_{\lambda
}$.  To finish our proof that $l_{1}$ is bounded below away from zero $%
l_{1}\geq C^{-1}$, we only need to show that $\left\Vert \alpha
_{L_{1}}\right\Vert =o(1)$ and $\left\Vert \beta _{L_{2}}\right\Vert =o(1)$.
This can be seen from our previous results Equation (\ref{e_B Bound}) $\max
\left\{ \left\Vert \alpha _{\left( B_{1}^{-}\right) ^{c}}\right\Vert
,\left\Vert \beta _{\left( B_{2}^{-}\right) ^{c}}\right\Vert \right\} =o(1)$
and Equation\ (\ref{Range of B}) $L_{i}\subset \left( B_{i}^{-}\right) ^{c}$. 
\end{proof}

\section{Proof of Theorem \protect\ref{thm:lower}}

The main tool for our proof is the Fano's Lemma, which is based on multiple
hypotheses testing argument. To introduce Fano's Lemma, we first need to
introduce a few notations. For two probability measures $\mathbb{P}$ and $%
\mathbb{Q}$ with density $p$ and $q$ with respect to a common dominating
measure $\mu $, write the Kullback-Leibler divergence as $K(\mathbb{P},%
\mathbb{Q})=\int p\log {\frac{p}{q}}d\mu $. The following lemma, which can
be viewed as a version of Fano's Lemma, gives a lower bound for the minimax
risk over the parameter set $\Omega =\left\{ \omega _{0},\omega _{1},\ldots
,\omega _{m_{\ast }}\right\} $ with the loss function $d\left( \cdot ,\cdot
\right) $. See \cite{Tsybakov2009introduction}, Section 2.6 for more
detailed discussions.

\begin{lemma}[Fano]
\label{fanno.lem} Let $\Omega =\{\omega _{k}:\;k=0,...,m_{\ast }\}$ be a
parameter set, where $d$ is a distance over $\Omega $. Let $\{\mathbb{P}%
_{\omega }:\omega \in \Omega \}$ be a collection of probability
distributions satisfying 
\begin{equation}
\frac{1}{m_{\ast }}\sum_{1\leq k\leq m_{\ast }}K\left( \mathbb{P}_{\omega
_{k}},\mathbb{P}_{\omega _{0}}\right) \leq c\log m_{\ast }
\label{KL condition for Fano}
\end{equation}%
with $0<c<1/8$. Let $\hat{\omega}$ be any estimator based on an observation
from a distribution in $\left\{ \mathbb{P}_{\theta },\theta \in \Theta
\right\} $. Then%
\begin{equation*}
\sup_{\omega \in \Omega }\mathbb{E}d^{2}\left( \hat{\omega},\omega \right)
\geq \min_{i\neq j}\frac{d^{2}\left( \omega _{i},\omega _{j}\right) }{4}%
\frac{\sqrt{m_{\ast }}}{1+\sqrt{m_{\ast }}}\left( 1-2c-\sqrt{\frac{2c}{\log
m_{\ast }}}\right) .
\end{equation*}
\end{lemma}

To apply the Fano's Lemma, we need to find a collection of least favorable
parameters $\Omega =\left\{ \omega _{0},\omega _{1},\ldots ,\omega _{m_{\ast
}}\right\} $ such that the difficulty of estimation among this subclass is
almost the same as that among the whole sparsity class $\mathcal{F}%
_{q}^{p_{1},p_{2}}\left( s_{1},s_{2},C_{\lambda }\right) .$ To be specific,
we check that the distance $d^{2}\left( \omega _{i},\omega _{j}\right) $
among this collection of least favorable parameters is lower bounded by the
sharp rate of the convergence and the average Kullback-Leibler divergence is
indeed bounded above by the logarithm cardinality of the collection, i.e.
Equation (\ref{KL condition for Fano}). In the proof we will show this via
three main steps. Before that, we need two auxiliary lemmas.

\begin{lemma}
\label{V-G bound} Let $\left\{ 0,1\right\} ^{p_{1}-1}$ be equipped with
Hamming distance $\delta $. For integer $0<d<\frac{p_{1}-1}{4},$ there
exists some subset $\Phi =\left\{ \phi _{1},\ldots ,\phi _{m}\right\}
\subset \left\{ 0,1\right\} ^{p_{1}-1}$ such that%
\begin{eqnarray}
\delta \left( \phi _{i},\phi _{j}\right) &\geq &\frac{d}{2}\quad 
\mbox{\rm  for
any }\quad \phi _{i}\neq \phi _{j},  \label{separate property} \\
\delta \left( \phi _{i},\vec{0}\right) &=&d\quad \mbox{\rm  for any }\quad
\phi _{i},  \label{sparseness property} \\
\log m &\geq &C_{0}d\log \left( \frac{p_{1}}{d}\right) 
\mbox{\rm  for some
constant }C_{0}.  \label{cardinality property}
\end{eqnarray}
\end{lemma}

See \cite{massart2007concentration}, Lemma 4.10 for more details.

\begin{lemma}
\label{KL property} For $i=1,2$, let $\theta _{i}$ and $\eta _{i}$ be some
unit vectors and $\mathbb{P}_{i}$ be the distribution of $n$ i.i.d. $N\left(
0,\Sigma _{i}\right) $, where the covariance matrix is defined as 
\begin{equation*}
\Sigma _{i}=%
\begin{pmatrix}
I_{p_{1}\times p_{1}} & \lambda \theta _{i}\eta _{i}^{T} \\ 
\lambda \eta _{i}\theta _{i}^{T} & I_{p_{2}\times p_{2}}%
\end{pmatrix}%
.
\end{equation*}%
Then we have%
\begin{equation*}
K\left( \mathbb{P}_{1},\mathbb{P}_{2}\right) \leq \frac{n\lambda ^{2}}{%
2\left( 1-\lambda ^{2}\right) }\left( \left\Vert \theta _{1}-\theta
_{2}\right\Vert ^{2}+\left\Vert \eta _{1}-\eta _{2}\right\Vert ^{2}\right) .
\end{equation*}
\end{lemma}

We shall divide the proof into three main steps.

\noindent \textbf{Step 1: Constructing the parameter set. }Without loss of
generality we assume $C_{\lambda }<\frac{1}{2}$. The subclass of parameters
we will pick can be described in the following form:%
\begin{equation*}
\Sigma _{\left( p_{1}+p_{2}\right) \times \left( p_{1}+p_{2}\right) }=%
\begin{pmatrix}
I_{p_{1}\times p_{1}} & \frac{1}{2}\theta \eta ^{T} \\ 
\frac{1}{2}\eta \theta ^{T} & I_{p_{2}\times p_{2}}%
\end{pmatrix}%
,
\end{equation*}%
where we will pick a collection of $\theta $ or $\eta $ such that they are
separated with the right rate of convergence. Without loss of generality, we
assume that $s_{1}\geq s_{2}$ and hence $s=s_{1}$. If the inequality $%
s_{1}\geq s_{2}$ holds in the other direction, we only need to switch the
roles of $\theta $ and $\eta $. In this case, we will pick the collection of
least favorable parameters indexed by the canonical pair $\omega =\left(
\theta ,\eta \right) $. Specifically, define $\Omega =\left\{ \omega
_{0},\omega _{1},\ldots ,\omega _{m_{\ast }}\right\} $ where $\omega
_{i}=\left( \theta _{i},e_{1}\right) $ and $e_{1}$ is the unit vector in $%
\mathbb{R}^{p_{2}}$ with the first coordinate $1$ and all others $0.$ The
number $m_{\ast }$ will be determined by the a version of Varshamov-Gilbert
bound in Lemma \ref{V-G bound}.

Now we define $m_{\ast }=m-1$ and each $\theta _{i-1}=\left( \left(
1-\epsilon ^{2}\right) ^{1/2},\phi _{i}\epsilon d^{-1/2}\right) $, where we
pick 
\begin{equation*}
\epsilon =c_{1}\left( s_{1}-1\right) ^{1/2}\left( \frac{\log p_{1}}{n}%
\right) ^{\frac{1}{2}-\frac{q}{4}}\mbox{\rm and }d=\left( s_{1}-1\right)
\left( \frac{\log p_{1}}{n}\right) ^{-\frac{q}{2}}\mbox{\rm ,}
\end{equation*}%
while $\phi _{i}$ and $m$ are determined in Lemma \ref{V-G bound}
accordingly. The constant $c_{1}\in \left( 0,1\right) $\ is to be determined
later. It's easy to check that each $\theta _{i-1}$ is a unit vector. By our
sparsity assumption $1-\epsilon ^{2}\geq 0.5$ by picking a sufficient small
constant $c_{1}$ and consequently the first coordinate is the largest one in
magnitude. Clearly $\left\vert \theta _{i,(1)}\right\vert ^{q}\leq s_{1}$.
Moreover, we have%
\begin{eqnarray*}
\left\vert \theta _{i,(k)}\right\vert ^{q} &=&\epsilon
^{q}d^{-q/2}=c_{1}^{q}\left( \frac{\log p_{1}}{n}\right) ^{\frac{q}{2}}\leq
s_{1}k^{-1},\mbox{\rm  }2\leq k\leq d+1, \\
\left\vert \theta _{i,(k)}\right\vert ^{q} &=&0\leq s_{1}k^{-1},\mbox{\rm  }%
k>d+1.
\end{eqnarray*}%
Hence each $\theta _{i}$ is in the corresponding weak $l_{q}$ ball.
Therefore our parameter subclass $\Omega \subset \mathcal{F}%
_{q}^{p_{1},p_{2}}\left( s_{1},s_{2},C_{\lambda }\right) $.

\noindent\textbf{Step 2: Bounding }$d^{2}\left( \omega _{i},\omega
_{j}\right) $. The loss function we considered in this section for $\omega
_{i}$ and $\omega _{j}$ can be simplified as 
\begin{equation*}
d^{2}\left( \omega _{i},\omega _{j}\right) =L^{2}(\theta _{i},\theta
_{j})=\left\Vert \theta _{i}\theta _{i}^{T}-\theta _{j}\theta
_{j}^{T}\right\Vert _{F}^{2}\geq \left\Vert \theta _{j}-\theta
_{i}\right\Vert ^{2},
\end{equation*}%
whenever $\left\Vert \theta _{j}-\theta _{i}\right\Vert ^{2}\leq 2$ which is
satisfied in our setting since $\left\Vert \theta _{j}-\theta
_{i}\right\Vert ^{2}\leq 2\epsilon ^{2}\leq 1$. The Equation (\ref{separate
property}) in Lemma \ref{V-G bound} implies that%
\begin{eqnarray}
\min_{i\neq j}d^{2}\left( \omega _{i},\omega _{j}\right) &\geq &\frac{d}{2}%
\left( \epsilon d^{-1/2}\right) ^{2}  \notag \\
&\geq &\frac{\epsilon ^{2}}{2}=\frac{c_{1}^{2}}{2}\left( s_{1}-1\right)
\left( \frac{\log p_{1}}{n}\right) ^{1-\frac{q}{2}}.  \label{bounding d^2}
\end{eqnarray}%
which is the sharp rate of convergence, noting that the Equation (\ref%
{bounding d^2}) is still true up to a constant when we replace $p_{1}$ by $p$
and $s_{1}$ by $s$.

\noindent\textbf{Step 3: Bounding the Kullback-Leibler divergence.}

Note that in our case $\eta _{1}=\eta _{2}=e_{1}.$ Lemma \ref{KL property},
together with Equations (\ref{sparseness property}) and (\ref{cardinality
property}), imply%
\begin{eqnarray*}
\frac{1}{m_{\ast }}\sum_{1\leq k\leq m_{\ast }}K\left( \mathbb{P}_{\omega
_{k}},\mathbb{P}_{\omega _{0}}\right) &\leq &\frac{n\lambda ^{2}}{2\left(
1-\lambda ^{2}\right) }\frac{d}{2}\left( \epsilon d^{-1/2}\right) ^{2} \\
&=&\frac{n\lambda ^{2}}{4\left( 1-\lambda ^{2}\right) }c_{1}^{2}\left(
s_{1}-1\right) \left( \frac{\log p_{1}}{n}\right) ^{1-\frac{q}{2}} \\
&=&\frac{n}{12}c_{1}^{2}d\frac{\log p_{1}}{n}<\frac{1}{10}\log m,
\end{eqnarray*}%
where the last inequality is followed by picking a sufficiently small
constant $c_{1}>0$ and noting that $\lambda =\frac{1}{2}$, $s_{i}\left( 
\frac{n}{\log p_{i}}\right) ^{q/2}=o(p_{i})$ in the assumption. Therefore we
could apply Lemma \ref{fanno.lem} and Equation (\ref{bounding d^2}) to
obtain the sharp rate of convergence, which completes our proof.%
\begin{equation*}
\sup_{P\in \mathcal{F}}\mathbb{E}_{P}\left( L^{2}(\hat{\theta},\theta )\vee
L^{2}(\hat{\eta},\eta )\right) \geq Cs\left( \frac{\log p}{n}\right) ^{1-%
\frac{q}{2}},
\end{equation*}%
for any estimator $(\hat{\theta},\hat{\eta})$, where $\mathcal{F}=\mathcal{F}%
_{q}^{p_{1},p_{2}}\left( s_{1},s_{2},C_{\lambda }\right) $. Hence, Theorem %
\ref{thm:lower} is proved.

We finally prove Lemma \ref{KL property} to complete the whole proof. 
\begin{proof}[Proof of Lemma \ref{KL property}]
Let's rewrite matrix $\Sigma _{i}$ in the following way%
\begin{equation*}
\Sigma _{i}=I+\frac{\lambda }{2}\left(
\begin{array}{c}
\theta _{i} \\
\eta _{i}%
\end{array}%
\right) \left(
\begin{array}{c}
\theta _{i} \\
\eta _{i}%
\end{array}%
\right) ^{T}-\frac{\lambda }{2}\left(
\begin{array}{c}
\theta _{i} \\
-\eta _{i}%
\end{array}%
\right) \left(
\begin{array}{c}
\theta _{i} \\
-\eta _{i}%
\end{array}%
\right) ^{T}.
\end{equation*}%
Note that $\left\Vert \left( \theta _{i}^{T},\eta _{i}^{T}\right)
\right\Vert =\sqrt{2},$ so $\Sigma _{1}$ and $\Sigma _{2}$ have the same
eigenvalues $1+\lambda ,1,1,\ldots ,1-\lambda .$ Thus we have%
\begin{eqnarray}
K\left( \mathbb{P}_{1},\mathbb{P}_{2}\right) &=&\frac{n}{2}\left[
\mbox{\rm
tr}\left( \Sigma _{2}^{-1}\Sigma _{1}\right) -p-\log \det \left( \Sigma
_{2}^{-1}\Sigma _{1}\right) \right]  \notag \\
&=&\frac{n}{2}\left[ \mbox{\rm tr}\left( \Sigma _{2}^{-1}\Sigma _{1}\right)
-p\right] =\frac{n}{2}\mbox{\rm tr}\left( \Sigma _{2}^{-1}\left( \Sigma
_{1}-\Sigma _{2}\right) \right)  \notag \\
&=&\frac{n\lambda }{2}\mbox{\rm tr}\left( \Sigma _{2}^{-1}\left(
\begin{array}{cc}
0 & A_{1}-A_{2} \\
A_{1}^{T}-A_{2}^{T} & 0%
\end{array}%
\right) \right) ,  \label{KL matrix}
\end{eqnarray}%
where $p=p_{1}+p_{2}$ and $A_{i}=\theta _{i}\eta _{i}^{T}.$ To explicitly
write down the inverse of $\Sigma _{2}$, we use its eigen-decomposition%
\begin{equation*}
\Sigma _{2}=I-\left(
\begin{array}{cc}
\theta _{2}\theta _{2}^{T} & 0 \\
0 & \eta _{2}\eta _{2}^{T}%
\end{array}%
\right) +\frac{1+\lambda }{2}\left(
\begin{array}{c}
\theta _{2} \\
\eta _{2}%
\end{array}%
\right) \left(
\begin{array}{c}
\theta _{2} \\
\eta _{2}%
\end{array}%
\right) ^{T}+\frac{1-\lambda }{2}\left(
\begin{array}{c}
\theta _{2} \\
-\eta _{2}%
\end{array}%
\right) \left(
\begin{array}{c}
\theta _{2} \\
-\eta _{2}%
\end{array}%
\right) ^{T}.
\end{equation*}%
Thus we have%
\begin{eqnarray*}
\Sigma _{2}^{-1} &=&I-\left(
\begin{array}{cc}
\theta _{2}\theta _{2}^{T} & 0 \\
0 & \eta _{2}\eta _{2}^{T}%
\end{array}%
\right) +\frac{1}{2\left( 1+\lambda \right) }\left(
\begin{array}{c}
\theta _{2} \\
\eta _{2}%
\end{array}%
\right) \left(
\begin{array}{c}
\theta _{2} \\
\eta _{2}%
\end{array}%
\right) ^{T}+\frac{1}{2\left( 1-\lambda \right) }\left(
\begin{array}{c}
\theta _{2} \\
-\eta _{2}%
\end{array}%
\right) \left(
\begin{array}{c}
\theta _{2} \\
-\eta _{2}%
\end{array}%
\right) ^{T}, \\
&=&I-\frac{\lambda }{2\left( 1+\lambda \right) }\left(
\begin{array}{c}
\theta _{2} \\
\eta _{2}%
\end{array}%
\right) \left(
\begin{array}{c}
\theta _{2} \\
\eta _{2}%
\end{array}%
\right) ^{T}+\frac{\lambda }{2\left( 1-\lambda \right) }\left(
\begin{array}{c}
\theta _{2} \\
-\eta _{2}%
\end{array}%
\right) \left(
\begin{array}{c}
\theta _{2} \\
-\eta _{2}%
\end{array}%
\right) ^{T}
\end{eqnarray*}%
Plugging this representation into Equation (\ref{KL matrix}), we obtain that%
\begin{eqnarray*}
K\left( \mathbb{P}_{1},\mathbb{P}_{2}\right) &=&\frac{n\lambda ^{2}}{2}\left[
\mbox{\rm tr}\left( \frac{A_{2}\left( A_{2}^{T}-A_{1}^{T}\right) }{1-\lambda
^{2}}\right) +\mbox{\rm tr}\left( \frac{A_{2}^{T}\left( A_{2}-A_{1}\right) }{%
1-\lambda ^{2}}\right) \right] \\
&=&\frac{n\lambda ^{2}}{1-\lambda ^{2}}\mbox{\rm tr}\left( A_{2}\left(
A_{2}^{T}-A_{1}^{T}\right) \right) \\
&=&\frac{n\lambda ^{2}}{1-\lambda ^{2}}\left( 1-\left( \theta _{1}^{T}\theta
_{2}\right) \left( \eta _{1}^{T}\eta _{2}\right) \right) \\
&=&\frac{n\lambda ^{2}}{1-\lambda ^{2}}\left( 1-\left( 1-\frac{\left\Vert
\theta _{1}-\theta _{2}\right\Vert ^{2}}{2}\right) \left( 1-\frac{\left\Vert
\eta _{1}-\eta _{2}\right\Vert ^{2}}{2}\right) \right) \\
&\leq &\frac{n\lambda ^{2}}{2\left( 1-\lambda ^{2}\right) }\left( \left\Vert
\theta _{1}-\theta _{2}\right\Vert ^{2}+\left\Vert \eta _{1}-\eta
_{2}\right\Vert ^{2}\right) ,
\end{eqnarray*}%
where we used that $\theta _{i}$ and $\eta _{i}$ are unit vectors in the
fourth equation above.
\end{proof}

\bibliographystyle{Chicago}
\bibliography{CCA}
%\end{raggedright}           % Comment this out if you don't want ragged edges.

\end{document}